\documentclass[12pt]{amsart}
\usepackage{comment}
\usepackage[utf8]{inputenc}
\usepackage[T1]{fontenc}
\DeclareMathAlphabet{\mathpzc}{OT1}{pzc}{m}{it}
\usepackage{geometry}

\usepackage[dvipsnames]{xcolor}

\newcommand*\ja[1]{\textcolor{Fuchsia}{\textbf{*Jason: #1*}}}

\newcommand*\gf[1]{\textcolor{RoyalBlue}{\textbf{*Gary: #1*}}}

\newcommand*\sv[1]{\textcolor{PineGreen}{\textbf{*Sandro: #1*}}}

\usepackage{amsfonts}
\usepackage{amssymb}
\usepackage{amsthm}
\usepackage{amsmath}
\usepackage{amscd}
\usepackage[shortlabels]{enumitem}
\usepackage{mathrsfs}
\usepackage{tikz}
\usetikzlibrary{calc,arrows,decorations.pathreplacing}
\usepackage{nicefrac, xfrac}
\usepackage{mathtools,xparse}

\usepackage[pagebackref, colorlinks = true,
linkcolor = red,
urlcolor  = blue,
citecolor = blue,
anchorcolor = blue]{hyperref}
\hypersetup{pdftitle={\@title},pdfauthor={\@author}}

\theoremstyle{plain}

\newtheorem{theorem}{Theorem}[section]
\newtheorem{corollary}[theorem]{Corollary}
\newtheorem{lemma}[theorem]{Lemma}

\newtheorem{proposition}[theorem]{Proposition}

\newtheorem{claim}{Claim}[theorem]

\theoremstyle{definition}
\newtheorem{definition}[theorem]{Definition}

\newtheorem{remark}[theorem]{Remark}

\newtheorem*{theorem*}{Theorem}

\newenvironment{subproof}[1][\proofname]{%
	\begin{proof}[#1]%
	}{%
	\end{proof}%
}

\renewcommand{\phi}{\varphi}
\renewcommand{\epsilon}{\varepsilon}

\newcommand{\spn}{\ensuremath{\mathrm{span}}}

\newcommand{\sub}{\ensuremath{ \subseteq}} 
\newcommand{\bus}{\ensuremath{ \supseteq}} 
\newcommand{\set}[1]{\ensuremath{ \left\{#1\right\} }} 
\newcommand{\norm}[1]{\ensuremath{ \|#1\| }} 

\newcommand{\vertiii}[1]{{\left\vert\kern-0.25ex\left\vert\kern-0.25ex\left\vert #1
		\right\vert\kern-0.25ex\right\vert\kern-0.25ex\right\vert}}
\newcommand{\trinorm}[1]{\ensuremath{\vertiii{#1}}} 
\newcommand{\absval}[1]{\ensuremath{ \left\lvert#1\right\rvert }} 
\newcommand{\spot}{\ensuremath{ \makebox[1ex]{\textbf{$\cdot$}} }}

\newcommand{\ol}[1]{\overline{#1}}
\newcommand{\Ul}[1]{\underline{#1}}

\newcommand{\bs}{\ensuremath{\backslash}}

\renewcommand\~{\tilde}

\newcommand{\lt}{\ensuremath{\left}}
\newcommand{\rt}{\ensuremath{\right}}

\newcommand{\Om}{\ensuremath{\Omega}}
\newcommand{\Lm}{\ensuremath{\Lambda}}
\newcommand{\Gm}{\ensuremath{\Gamma}}

\newcommand{\Sg}{\ensuremath{\Sigma}}

\newcommand{\Dl}{\ensuremath{\Delta}}

\newcommand{\om}{\ensuremath{\omega}}
\newcommand{\lm}{\ensuremath{\lambda}}
\newcommand{\gm}{\ensuremath{\gamma}}
\newcommand{\al}{\ensuremath{\alpha}}
\newcommand{\bt}{\ensuremath{\beta}}
\newcommand{\sg}{\ensuremath{\sigma}}
\newcommand{\ep}{\ensuremath{\epsilon}}
\newcommand{\eps}{\ensuremath{\epsilon}}
\newcommand{\dl}{\ensuremath{\delta}}
\newcommand{\kp}{\ensuremath{\kappa}}
\newcommand{\zt}{\ensuremath{\zeta}}
\newcommand{\ta}{\ensuremath{\theta}}

\newcommand{\vta}{\ensuremath{\vartheta}}

\newcommand{\vrho}{\ensuremath{\varrho}}

\DeclareSymbolFont{bbold}{U}{bbold}{m}{n}
\DeclareSymbolFontAlphabet{\mathbbold}{bbold}

\newcommand{\ind}{\ensuremath{\mathbbold{1}}}
\DeclareMathOperator*{\esssup}{ess\sup}
\DeclareMathOperator*{\essinf}{ess\inf}

\DeclareMathOperator{\supp}{supp}

\newcommand{\cA}{\ensuremath{\mathcal{A}}}
\newcommand{\cB}{\ensuremath{\mathcal{B}}}
\newcommand{\cC}{\ensuremath{\mathcal{C}}}
\newcommand{\cD}{\ensuremath{\mathcal{D}}}
\newcommand{\cE}{\ensuremath{\mathcal{E}}}
\newcommand{\cF}{\ensuremath{\mathcal{F}}}

\newcommand{\cJ}{\ensuremath{\mathcal{J}}}

\newcommand{\cL}{\ensuremath{\mathcal{L}}}

\newcommand{\tcL}{\ensuremath{\tilde{\mathcal{L}}}}

\newcommand{\cP}{\ensuremath{\mathcal{P}}}

\newcommand{\cX}{\ensuremath{\mathcal{X}}}

\newcommand{\cZ}{\ensuremath{\mathcal{Z}}}

\newcommand{\sA}{\ensuremath{\mathscr{A}}}
\newcommand{\sB}{\ensuremath{\mathscr{B}}}

\newcommand{\sF}{\ensuremath{\mathscr{F}}}

\newcommand{\sL}{\ensuremath{\mathscr{L}}}

\newcommand{\CC}{\ensuremath{\mathbb C}} 

\newcommand{\NN}{\ensuremath{\mathbb N}}

\newcommand{\RR}{\ensuremath{\mathbb R}}

\newcommand{\ZZ}{\ensuremath{\mathbb Z}} 

\def\lra{\longrightarrow}
\def\var{\text{{\rm var}}}
\def\Var{\text{{\rm Var}}}
\def\BV{\text{{\rm BV}}}
\def\Leb{\text{{\rm Leb}}}
\DeclareMathOperator*{\nuessinf}{\ensuremath{\nu\text{-ess}\inf}}
\DeclareMathOperator*{\Lebessinf}{\ensuremath{\text{Leb-ess}\inf}}
\DeclareMathOperator*{\nuesssup}{\ensuremath{\nu\text{-ess}\sup}}
\DeclareMathOperator*{\Lebesssup}{\ensuremath{\text{Leb-ess}\sup}}

\def\lt{\left}
\def\rt{\right}

\providecommand{\phantomsection}{}
\AtBeginDocument{\let\textlabel\label}
\makeatletter
\newcommand{\mylabel}[2]{\raisebox{.7\normalbaselineskip}{\phantomsection}(#1)%
	\def\@currentlabel{#1}\textlabel{#2}}
\makeatother

\makeatletter
\newcommand\xlabel[2][]{\phantomsection\def\@currentlabelname{#1}\label{#2}}
\makeatother

\newcommand{\bcomma}{,\allowbreak}

\ExplSyntaxOn
\NewDocumentCommand{\mathlist}{ O{,} m m }
 {
  \egreg_mathlist:nnn { #1 } { #2 } { #3 }
 }

\seq_new:N \l__egreg_mathlist_seq
\cs_new_protected:Npn \egreg_mathlist:nnn #1 #2 #3
 {
  \seq_set_split:Nnn \l__egreg_mathlist_seq { #1 } { #3 }
  \seq_use:Nnnn \l__egreg_mathlist_seq { #2 } { #2 } { #2 }
 }
\ExplSyntaxOff

\newcommand{\flag}[1]{\textbf{**[#1]**}}

\newcommand{\cEP}{\ensuremath{\cE P}} 
\allowdisplaybreaks

\numberwithin{equation}{section}
\title[]{Perturbation formulae for quenched random dynamics with applications to open systems and extreme value theory}
\date{\today}

\author{Jason Atnip}
\address{School of Mathematics and Statistics, University of New South Wales, Sydney, NSW 2052, Australia}
\email{\href{j.atnip@unsw.edu.au}{j.atnip@unsw.edu.au} }

\author{Gary Froyland}
\address{School of Mathematics and Statistics, University of New South Wales, Sydney, NSW 2052, Australia}
\email{\href{g.froyland@unsw.edu.au}{g.froyland@unsw.edu.au} }

\author{Cecilia Gonz\'alez-Tokman}
\address{School of Mathematics and Physics, The University of Queensland, St Lucia, QLD 4072, Australia}
\email{\href{cecilia.gt@uq.edu.au}{cecilia.gt@uq.edu.au} }

\author{Sandro Vaienti}
\address{Aix Marseille Université, Université de Toulon, CNRS, CPT, 13009 Marseille, France}
\email{\href{vaienti@cpt.univ-mrs.fr}{vaienti@cpt.univ-mrs.fr} }

\begin{document}

	\maketitle
	
	\begin{abstract}
		We consider quasi-compact linear operator cocycles $\mathcal{L}^{n}_\omega:=\mathcal{L}_{\sigma^{n-1}\omega}\circ\cdots\circ\mathcal{L}_{\sigma\omega}\circ \mathcal{L}_{\omega}$ driven by an invertible ergodic process $\sigma:\Omega\to\Omega$, and their small perturbations $\mathcal{L}_{\omega,\epsilon}^{n}$.
		We prove an abstract $\omega$-wise first-order formula for the leading Lyapunov multipliers  $\lambda_{\omega,\epsilon}=\lambda_\omega-\theta_{\omega}\Delta_{\omega,\epsilon}+o(\Delta_{\omega,\epsilon})$, where $\Delta_{\omega,\epsilon}$ quantifies the closeness of $\mathcal{L}_{\omega,\epsilon}$ and $\mathcal{L}_{\omega}$.  
		We then consider the situation where $\mathcal{L}_{\omega}^{n}$ is a transfer operator cocycle for a random map cocycle $T_\omega^{n}:=T_{\sigma^{n-1}\omega}\circ\cdots\circ T_{\sigma\omega}\circ  T_{\omega}$ 
		and the perturbed transfer operators $\mathcal{L}_{\omega,\epsilon}$ are defined by the introduction of small random holes $H_{\omega,\epsilon}$ in $[0,1]$, creating a random open dynamical system.
		We obtain a first-order perturbation formula in this setting, which reads $\lambda_{\omega,\epsilon}=\lambda_\omega-\theta_{\omega}\mu_{\omega}(H_{\omega,\epsilon})+o(\mu_\omega(H_{\omega,\epsilon})),$ where $\mu_\omega$ is the unique equivariant random measure (and equilibrium state) for the original closed random dynamics. 
		Our new machinery is then deployed to create a spectral approach for a quenched extreme value theory that considers random dynamics with general ergodic invertible driving, and random observations.
		An extreme value law is derived using the first-order terms $\theta_\omega$.
		Further, in the setting of random piecewise expanding interval maps, we establish the existence of random equilibrium states and conditionally invariant measures for random open systems via a random perturbative approach.
		Finally we prove quenched statistical limit theorems for random equilibrium states arising from contracting potentials.
		We illustrate the theory with a variety of explicit examples.
		
	\end{abstract}
	\newpage
	\tableofcontents
	\newpage	
	
	\section{Introduction}
	The spectral approach to studying deterministic dynamical systems $T:X\to X$ on a phase space $X$, centres on the analysis of a transfer operator $\mathcal{L}:\mathcal{B}(X)\to\mathcal{B}(X)$, given by $\mathcal{L}f(y)=\sum_{x\in T^{-1}y}g(x)f(x)$, for $f$ in a suitable Banach space $\mathcal{B}(X)$ and for suitable weight functions $g:X\to\mathbb{R}^+$.
	If the map $T$ is \emph{covering} and the weight function is \emph{contracting} in the sense of \cite{liverani_conformal_1998} (or similarly if $\sup\log g<P(T)$ as in \cite{denker_existence_1991,denker_uniqueness_1990} or if $\sup\log g-\inf\log g<h_{top}(T)$ as in \cite{hofbauer_equilibrium_1982,bruin_equilibrium_2008}), then one obtains the existence of an equilibrium state $\mu$, with associated conformal measure $\nu$, with the topological pressure $P(T)$ and the density of equilibrium state $d\mu/d\nu$ given by the logarithm of the leading (positive) eigenvalue $\lambda$ of $\mathcal{L}$ and the corresponding positive eigenfunction $h$, respectively. The map $T$ exhibits an exponential decay of correlations with respect to $\nu$ and $\mu$.
	
	Keller and Liverani \cite{KL99} showed that the leading eigenvalue $\lambda$ and eigenfunction $h$ of $\mathcal{L}$ vary continuously with respect to certain small perturbations of $\mathcal{L}$.
	One example of such a perturbation is the introduction of a small hole $H\subset X$.
	When trajectories of $T$ enter $H$, the trajectory ends;  this generates an \emph{open} dynamical system.
	The set of initial conditions of trajectories that never land in $H$ is the survivor set $X_\infty$.
	For small holes, specialising to Lasota-Yorke maps of the interval, Liverani and Maume-Dechamps \cite{liverani_lasotayorke_2003} apply the perturbation theory of \cite{KL99} to obtain the existence of a unique conformal measure $\nu$ and an absolutely continuous \emph{conditionally} invariant measure $\mu$, with density $h\in\BV([0,1])$. 
	The leading eigenvalue $\lambda$ is interpretable as an \emph{escape rate}, and the open system displays an exponential decay of correlations with respect to $\nu$ and $\mu$.
	
	To obtain finer information on the behaviour of $\lambda$ with respect to perturbation size, particularly in the situation where the perturbation is not smooth, such as perturbations arising from the introduction of a hole, one requires some additional control on the perturbation.
	Keller and Liverani \cite{keller_rare_2009} develop abstract conditions on $\mathcal{L}$ and its perturbations $\mathcal{L}_\epsilon$ to ensure good first-order behaviour with respect to the perturbation size. 
	Following \cite{keller_rare_2009}, several authors \cite{ferguson_escape_2012,FFT015,pollicott_open_2017,BDT18} have used the Keller-Liverani \cite{KL99} perturbation theory to obtain similar first-order behaviour of the escape rate with respect to the perturbation size for open systems in various settings.
	
	This ``linear response'' of $\lambda$ is exploited in Keller \cite{keller_rare_2012} to develop an elegant spectral approach to deriving an exponential extreme value law to describe likelihoods of observing extreme values from evaluating an observation function $h:X\to\mathbb{R}$ along orbits of $T$.
	In particular, the $N\to\infty$ limiting law of
	\begin{equation}\label{K12}
		\nu\left(\left\{x\in X: h(T^j(x))\le z_{N}, j=0, \dots, N-1\right\}\right),
	\end{equation}
	where the thresholds $z_N$ are chosen so that $\lim_{N\to\infty} N\mu(\{x:h(x)>z_N\})\to t$ for some $t>0$.
	The spectral approach of \cite{keller_rare_2012} also provides a relatively explicit expression for the limit of (\ref{K12}), namely
	$\lim_{N\to\infty}\nu\left(\left\{x\in X: h(T^j(x))\le z_{N}, j=0, \dots, N-1\right\}\right)=\exp(-t\theta_0),$
	where $\theta_0$ is the extremal index.
	
	\subsection{Quenched perturbations of abstract operator cocycles and sequential/random dynamical systems}
	In the present work we begin with sequential compositions of linear operators $\mathcal{L}_{\omega,0}^n:=\mathcal{L}_{\sigma^{n-1}\omega,0}\circ\cdots\circ \mathcal{L}_{\sigma\omega,0}\circ \mathcal{L}_{\omega,0}$, 
	where $\sigma:\Omega\to\Omega$ is an invertible map on a configuration set $\Omega$.
	The driving $\sigma$ could also be an ergodic map on a probability space $(\Omega,\mathcal{F},m)$.
	We then consider a family of perturbed cocycles $\mathcal{L}_{\omega,\ep}^n:=\mathcal{L}_{\sigma^{n-1}\omega,\ep}\circ\cdots\circ \mathcal{L}_{\sigma\omega,\ep}\circ \mathcal{L}_{\omega,\ep}$, where the size of the perturbation $\mathcal{L}_{\omega,0}-\mathcal{L}_{\omega,\ep}$ is quantified by the value $\Delta_{\omega,\epsilon}$ (Definition \ref{def: DL_om}).
	Our first main result is an abstract quenched formula (Theorem \ref{thm: GRPT}) for the Lyapunov multipliers $\lambda_{\omega,0}$ up to first order in the size of the perturbation $\Delta_{\omega,\epsilon}$ of the operators $\mathcal{L}_{\omega,0}$.
	This quenched random formula generalises the main abstract first-order formula in \cite{keller_rare_2009} stated in the case of a single deterministic operator $\mathcal{L}_0$.
	\begin{quote}
		\textbf{Theorem A:}
		Suppose that assumptions \eqref{P1}-\eqref{P9} hold (see Section \ref{Sec: Gen Perturb Setup}).
		If $\Delta_{\omega,\epsilon}>0$ for all $\epsilon>0$ then for $m$-a.e. $\omega\in\Omega$:
		\begin{align*}
			\lim_{\ep\to 0}\frac{\lm_{\om,0}-\lm_{\om,\ep}}{\Dl_{\om,\ep}}=1-\sum_{k=0}^{\infty}\hat{q}_{\om,0}^{(k)}
			=:\theta_{\omega,0}.
		\end{align*}
	\end{quote}
	
	We then introduce random compositions of maps $T_\omega:X\to X$, drawn from a collection $\{T_\omega\}_{\omega\in\Omega}$.
	A driving map $\sigma:\Omega\to\Omega$ on a probability space $(\Omega,\cF,m)$ creates a map cocycle $T_\omega^n:=T_{\sigma^{n-1}\omega}\circ\cdots\circ T_{\sigma\omega}\circ T_\omega$.
	This map cocycle generates a transfer operator cocycle $\mathcal{L}_{\omega,0}^n:=\mathcal{L}_{\sigma^{n-1}\omega,0}\circ\cdots\circ \mathcal{L}_{\sigma\omega,0}\circ \mathcal{L}_{\omega,0}$, where $\mathcal{L}_{\omega,0}$ is the transfer operator for the map $T_\omega$.
	At each iterate, a random hole $H_\omega\subset X$ is introduced.

	The existence of a random quenched equilibrium state, conformal measure, escape rates, and exponential decay of correlations is established in \cite{AFGTV21} for relatively large holes, generalising the large-hole constructions of \cite{liverani_lasotayorke_2003} for a single deterministic map $T$ to the random setting with general driving.
	In contrast, our focus here is establishing a random quenched analogue of the results of \cite{liverani_lasotayorke_2003}, \cite{keller_rare_2009}, and \cite{keller_rare_2012} discussed above.
	To this end, 
	we let $\mathcal{L}_{\omega,\epsilon}$ be the transfer operator for the open map $T_\omega$ with a hole $H_{\omega,\epsilon}$ introduced in $X$, namely $\mathcal{L}_{\omega,\epsilon}(f)=\mathcal{L}_\omega(\ind_{X\setminus H_{\omega,\epsilon}}f)$.
	Our second collection of main results is a quenched formula for the derivative of $\lambda_{\omega,\epsilon}$ with respect to the sample invariant probability measure of the hole $\mu_{\omega,0}(H_{\omega,\epsilon})$ (Theorem \ref{thm: dynamics perturb thm}),
	as well as a quenched formula for the derivative of the escape rate $R_\ep(\mu_{\om,0})$ with respect to the sample invariant probability measure of the hole $\mu_{\omega,0}(H_{\omega,\epsilon})$ (Corollary \ref{esc rat cor}). 
	
	\begin{quote}
		\textbf{Corollary A:}
		If $(\mathlist{\bcomma}{\Om, m, \sg, \cJ_0, T, \cB, \cL_0, \nu_0, \phi_0, H_\ep})$ is a  random open  system (see Section \ref{sec: open systems}) with
		$\mu_{\om,0}(H_{\om,\ep})>0$ for all $\ep>0$ and \eqref{C1}--\eqref{C8} hold (see Section \ref{sec:goodrandom}), then for $m$-a.e.\ $\omega\in\Omega$:
		\begin{align}\label{cor A eq}
			\lim_{\ep\to 0}
			\frac{1-\lm_{\om,\ep}/\lm_{\om,0}}{\mu_{\om,0}(H_{\om,\ep})}
			=
			1-\sum_{k=0}^{\infty}\hat{q}_{\om,0}^{(k)}
		\end{align}
		and 
		\begin{align*}
			R_\ep(\mu_{\om,0})=\int_\Om \log \lm_{\om,0}-\log\lm_{\om,\ep}\, dm(\om).
		\end{align*}
		In addition, if \eqref{DCT cond 1*} holds and the $\mu_{\om,0}$-measures of the random holes scale with $\epsilon$ according to \eqref{EVT style cond}, then for $m$-a.e.\ $\omega\in\Omega$ 
		\begin{equation*}
			\label{escaperatederiv}
			\lim_{\ep\to 0}
			\frac{R_\ep(\mu_{\om,0})}{\mu_{\om,0}(H_{\om,\ep})}
			=
			\int_\Om \left( 1-\sum_{k=0}^{\infty}\hat{q}_{\om,0}^{(k)} \right)\, dm(\om).
		\end{equation*}
	\end{quote}

	\subsection{Quenched extreme value theory general sequential/random dynamics}
	To generalize to the random setting the rescaled distribution of the maxima given by (\ref{K12}), we now consider the (random) real-valued random observables $h_{\omega}$ defined on the phase space $X$
	and construct a process $h_{\sigma^j\omega}\circ T^j_{\omega}$.
	We are interested in determining the limiting law of
	\begin{equation}\label{HHU}
		\mu_{\omega,0}\left(\left\{x\in X: h_{\sigma^j\omega}(T^j_{\omega}(x))\le z_{\sigma^j\omega, N}, j=0, \dots, N-1\right\}\right),
	\end{equation}
	where $\{z_{\sigma^j\omega, N}\}_{0\le j\le N-1}$ is a collection of real-valued  thresholds.
	The sample probability measure $\mu_{\omega,0}$ enjoys the equivariance property $T^*_{\omega}\mu_{\omega,0}=\mu_{\sigma\omega,0},$ however the process $h_{\sigma^j\omega}\circ T^j_{\omega}$ is not stationary in the probabilistic sense, which makes the theory slightly more difficult.
	
	The first approach to non-stationary extreme value theory (EVT) was given 
	under convenient assumptions, by H\"usler in \cite{H83,H86}. He was able to recover the usual extremal behaviour seen for
	i.i.d. or stationary sequences under Leadbetter’s conditions \cite{LED}, namely (i)
	guaranteed mixing properties for the probability
	measure governing the process 
	and (ii) 
	that the exceedances should appear scattered through the time period under consideration.
	H\"usler's results can not be applied in the dynamical systems setting because his uniform bounds on the control of the exceedances are not  satisfied for deterministic, random, or sequential compositions of maps. 
	The first contribution dealing explicitly with extreme value theory for random and sequential systems is the paper \cite{FFV017};  see also \cite{FFMV016} for an  application to point processes. 
	These works were an adaptation of  Leadbetter's conditions and Hüsler’s approach: let us call them the probabilistic approach to extreme value theory, to distinguish it from the spectral and perturbative approach used in the current paper.


	As in the deterministic case, in order to avoid a degenerate limit distribution, one should conveniently choose the thresholds $z_{\sigma^j\omega,N}.$  H\"usler proved convergence to the Gumbel's law if for some $0<t<\infty$ we have convergence of the sum 	\begin{equation}\label{BLH}
		\sum_{j=0}^{N-1}\mu_{\omega,0}\left(h_{\sigma^j\omega}(T^j_{\omega}(x))> z_{\sigma^j\omega, N}\right)\rightarrow t.
	\end{equation} for $m$-a.e.\ $\omega.$
	In our current framework we will additionally allow the positive number $t$ to be any positive random variable in $L^\infty(m)$.
	The nonstationary theory developed in \cite{FFV017} for quenched random processes, has the further restrictions that the observation function is fixed ($\omega$-independent), and the thresholds $z_N$ (like the scaling $t$) are just real numbers, and requires the obvious restricted equivalent of (\ref{BLH}).
	In our framework the observation function $h_\omega$, the scaling $t_\omega$, and the thresholds $z_{\omega,N}$ may all be random (but need not be).
	We generalise and simplify the requirement (\ref{BLH}) to
	\begin{equation}\label{os}
		N\mu_{\omega,0}\left(h_{\omega}(x)> z_{\omega, N}\right)=t_{\omega}+\xi_{\omega, N},
	\end{equation}
	where the scaling $t$ may be a random variable $t\in L^\infty(m)$ and the ``errors'' $\xi_{\omega,N}$ satisfy (i) $\lim_{N\rightarrow \infty}\xi_{\omega,N}=0$ a.e., and (ii) $|\xi_{\omega,N}|\le W<\infty$ for a.e.\ $\omega$ and all sufficiently large $N$. 
 
	We provide a more detailed discussion of the relationship between the conditions (\ref{BLH}) and (\ref{os}) at the end of section \ref{EEVV}.

	In summary, we derive a spectral approach for a quenched random extreme value law, where the dynamics $T_\omega$ is random, the observation functions $h_\omega$ can be random, the thresholds controlling what is an extreme value can be random, and the scalings of the likelihoods of observing extreme values can be random, all controlled by general invertible ergodic driving.
	Moreover, we obtain a formula for the explicit form of Gumbel law for the extreme value distribution.
	This leads to our main extreme value theory result (stated in detail later as Theorem \ref{evtthm}):
	\begin{quote}
		\textbf{Theorem B:}
		For a random open system  $(\mathlist{\bcomma}{\Om, m, \sg, \cJ_0, T, \cB, \cL_0, \nu_0, \phi_0, H_\ep})$ (see Section \ref{sec: open systems}),  assuming \eqref{C1'}, \eqref{C2}, \eqref{C3}, \eqref{C4'}, \eqref{C5'}, \eqref{C7'}, \eqref{C8},  and \eqref{xibound} (see Sections~\ref{sec:goodrandom} and \ref{EEVV}),  for $m$-a.e.\ $\omega\in\Omega$ one has
		\begin{align*}
		&\lim_{N\to\infty} \nu_{\omega,0}\left(x\in X: h_{\sigma^j\omega}(T^j_{\omega}(x))\le z_{\sigma^j\omega, N}, j=0, \dots, N-1\right)\\
		&\qquad =\lefteqn{\lim_{N\to\infty} \mu_{\omega,0}\left(x\in X: h_{\sigma^j\omega}(T^j_{\omega}(x))\le z_{\sigma^j\omega, N}, j=0, \dots, N-1\right)}\\
		&\qquad =\exp\left(-\int_\Omega t_\omega\theta_{\omega,0}\ dm(\omega)\right),
		\end{align*}
		where $\nu_{\omega,0}$ and $\mu_{\omega,0}$ are the random conformal measure and the random invariant measure, respectively, for our random dynamics, $t_\omega$ is a random scaling function, and $\theta_{\omega,0}$ is an $\omega$-local extremal index.
	\end{quote}
	This result generalises the spectral approach to extreme value theory in \cite{keller_rare_2012}, for a single map $T$, single observation function $h$, single scaling, and single sequence of thresholds.
	
Given a family of random holes $\mathcal{H}_{\om,N}:=\{H_{\sigma^j\om, \epsilon_N}\}_{j\ge 0},$ one can define
the first (random) hitting time to a hole, starting at initial condition $x$ and random configuration $\omega$:
$$
\tau_{\om, \mathcal{H}_{\om,N}}(x):=\inf\{k\ge 1, T^k_{\om}(x)\in H_{\sigma^k\om, \epsilon_N}\}.$$
When this family of holes shrink with increasing $N$ according to Condition (\ref{xibound}) (see Section \ref{sec:hts}),
Theorem B provides a description of the statistics of random hitting times, scaled by the measure of the holes (see Theorem \ref{hve}).
	\begin{quote}
	    \textbf{Corollary B:}
	For a random open system  $(\mathlist{\bcomma}{\Om, m, \sg, \cJ_0, T, \cB, \cL_0, \nu_0, \phi_0, H_\ep})$ (see Section \ref{sec: open systems}),  assume \eqref{C1'}, \eqref{C2}, \eqref{C3}, \eqref{C4'}, \eqref{C5'}, \eqref{C7'}, \eqref{C8},  and \eqref{xibound} (see Sections~\ref{sec:goodrandom} and \ref{EEVV}).
For $m$-a.e.\ $\omega\in\Omega$ one has
	\begin{equation}
		\label{eqhit0}
		\lim_{N\to\infty}\mu_{\om,0}\left(\tau_{\om, \mathcal{H}_{\om,N}} \mu_{\om,0}(H_{\om, \epsilon_N})>t_{\om}\right) = \exp\left(-{\int_\Om t_{\om} \theta_{\om,0}\ dm(\omega)}\right).
	\end{equation}
	\end{quote}

	\subsection{Thermodynamic formalism for quenched random open systems}
	By assuming some additional uniformity in $\omega$ on the maps $T_\omega$ we use a recent random perturbative result \cite{crimmins_stability_2019} to obtain a complete quenched thermodynamic formalism. 
	The following existence result extends Theorem C of \cite{liverani_lasotayorke_2003}, which concerned inserting a single small hole into the phase space of a single deterministic map $T$, to the situation of random map cocycles with small random holes with the random process controlled by general invertible ergodic driving $\sigma$.
	
	\begin{quote}
		\textbf{Theorem C:}
		Suppose 
		that (\ref{E1})--(\ref{E9})
		(see Section~\ref{sec: existence})
		hold for the random open system $(\mathlist{\bcomma}{\Om, m, \sg, [0,1], T, \BV, \cL_0, \nu_0, \phi_0, H_\ep})$ (see Section \ref{sec: open systems}).
		Then for each $\ep>0$ sufficiently small there exists a unique random $T$-invariant probability measure $\mu_\ep=\set{\mu_{\om,\ep}}_{\om\in\Om}$ with $\supp(\mu_{\om,\ep})\sub X_{\om,\infty,\ep}$. 
		Furthermore, $\mu_\ep$ is the unique relative equilibrium state for the random open system 
		and satisfies a forward and backward exponential decay of correlations. In addition, there exists a random absolutely continuous (with respect to $\set{\nu_{\om,0}}_{\om\in\Om}$) conditionally invariant probability measure $\varrho_\ep=\set{\varrho_{\om,\ep}}_{\om\in\Om}$ with $\supp(\varrho_{\om,\ep})\sub[0,1]\bs H_{\om,\ep}$ and density function $h_{\om,\ep}\in \BV([0,1])$. 
	\end{quote}
	For a more explicit statement of Theorem C as well as the relevant assumptions and definitions see Section~\ref{sec: existence} and Theorem~\ref{existence theorem}.\\
	
	\subsection{Quenched limit theorems for random dynamics with contracting potentials}
	In Section \ref{sec: limit theorems} we prove some quenched limit theorems for closed random dynamics. 
	These limit results are new for more general potentials and their associated equilibrium states. We use two approaches. 
	The first is based on the perturbative technique developed in \cite{dragicevic_spectral_2018}, which  generalises the Nagaev-Guivarc'h method to random cocycles. 
	This technique establishes a relation between a suitable twisted operator cocycle and  the distribution of the random Birkhoff sums. 
	As a consequence, it is possible to get quenched versions of the large deviation principle, the central limit theorem, and the local central limit theorem. 
	The second approach invokes the martingale  techniques previously used in the quenched random setting in \cite{DFGTV18a}. 
	We obtain the almost sure invariance principle (ASIP) for the equivariant measure $\mu_{\om,0},$ which also implies the central limit theorem and the law of iterated logarithms, a general bound for large deviations and a dynamical Borel-Cantelli lemma.

	 \subsection{Examples} We conclude with several explicit examples. 
	We start in Example \hyperref[example1]{1} with the weight $1/T'_\omega$ for a family of random maps, random scalings, and random observations $h_\omega$ with a common extremum location in phase space, which is a common fixed point of the $T_\omega$.
	The special cases of a fixed map $T$ on the one hand, and a fixed scaling $t$ on the other, are also considered.
	The same calculations can be extended to observation functions with common extrema on a periodic orbit common to all $T_\omega$.
	Next in Example \hyperref[example2]{2} we consider the more difficult case where orbits are distributed according to equilibrium states of a general geometric weight $|DT_\omega|^{-r}$,
	using random $\beta$-maps and random observation functions $h_\omega$ with a common extremum at $x=0$.
	Example \hyperref[example3]{3} investigates a fixed map $T$ with random observation functions $h_\omega$ with extrema in a shrinking neighbourhood of a fixed point of $T$, where the neighbourhood lacks the symmetry of Example \hyperref[example1]{1}.
	In the last example, Example \hyperref[example4]{4}, we again consider random maps $T_\om$  
	with random observations $h_\omega$, but now the maxima of the observations are not related to fixed or periodic points of $T_\om$.

	Though we apply our results to the setting of random interval maps our results apply equally well to other random settings including random subshifts 
	\cite{Bogenschutz_RuelleTransferOperator_1995a,mayer_countable_2015}, random distance expanding maps \cite{mayer_distance_2011}, random polynomial systems \cite{Bruck_Generalizediteration_2003}, random transcendental maps \cite{mayer_random_2018}. In addition, Theorem A (as well as \eqref{cor A eq} of Corollary A)
	applies to sequential and semi-group settings including \cite{atnip_dynamics_2020,stadlbauer_quenched_2020}. See Remark~\ref{rem seq exist} for a sequential version of Theorem C.

	\section{Sequential Perturbation Theorem}\label{Sec: Gen Perturb Setup}
	In this section we prove a general perturbation result for sequential operators acting on sequential Banach spaces. In particular, we will not require any measurability or notion of randomness in this section.  
	
	Suppose that $\Om$ is a set and that the map $\sg:\Om\to\Om$ is invertible.
	Furthermore, we suppose that there is a family of (fiberwise) normed vector spaces  $\set{\cB_\om, \norm{\spot}_{\cB_\om}}_{\om\in\Om}$ and dual spaces $\set{\cB_\om^*,\norm{\spot}_{\cB_\om^*}}_{\om\in\Om}$ such that for each $\om\in\Om$ and each $0\leq \ep\leq \ep_0$ there are operators $\cL_{\om,\ep}:\cB_\om\to\cB_{\sg\om}$ such that the following hold.
	\begin{enumerate}[align=left,leftmargin=*,labelsep=\parindent]
		\item[\mylabel{P1}{P1}]There exists a
		function $C_1:\Om\to\RR_+$ such that for $f\in\cB_\om$  we have
		\begin{align*}
			\sup_{\ep\geq 0}\norm{\cL_{\om,\ep}(f)}_{\cB_{\sg\om}}\leq C_1(\om)\norm{f}_{\cB_\om}.
		\end{align*}
		\item[\mylabel{P2}{P2}] For each $\om\in\Om$ and $\ep\geq 0$ there is a functional $\nu_{\om,\ep}\in\cB_\om^*$, the dual space of $\cB_\om$, $\lm_{\om,\ep}\in\CC\bs\set{0}$, and $\phi_{\om,\ep}\in\cB_\om$ such that
		\begin{align*}
			\cL_{\om,\ep}(\phi_{\om,\ep})=\lm_{\om,\ep}\phi_{\sg\om,\ep}
			\quad\text{ and }\quad
			\nu_{\sg\om,\ep}(\cL_{\om,\ep}(f))=\lm_{\om,\ep}\nu_{\om,\ep}(f)
		\end{align*}
		for all $f\in\cB_\om$. Furthermore we assume that
		$$
		\nu_{\om,0}(\phi_{\om,\ep})=1.
		$$
		\item[\mylabel{P3}{P3}] There is an operator $Q_{\om,\ep}:\cB_\om\to\cB_{\sg\om}$ such that for each $f\in\cB_\om$ we have
		\begin{align*}
			\lm_{\om,\ep}^{-1}\cL_{\om,\ep}(f)=\nu_{\om,\ep}(f)\cdot\phi_{\sg\om,\ep}+Q_{\om,\ep}(f).
		\end{align*}
		Furthermore, we have
		\begin{align*}
			Q_{\om,\ep}(\phi_{\om,\ep})=0
			\quad\text{ and }\quad
			\nu_{\sg\om,\ep}(Q_{\om,\ep}(f))=0.
		\end{align*}
		Note that assumptions \eqref{P2} and \eqref{P3} together imply that
		\begin{align*}
			\nu_{\om,\ep}(\phi_{\om,\ep})=1.
		\end{align*}
		\item[\mylabel{P4}{P4}] There exists a function $C_2:\Om\to\RR_+$ such that
		\begin{align*}
			\sup_{\ep\geq 0}\norm{\phi_{\om,\ep}}_{\cB_{\om}}=C_2(\om)<\infty.
		\end{align*}
		
		For each $\om\in\Om$ and $\ep\geq 0$  we define the quantities
		\begin{align}\label{def: DL_om}
			\Dl_{\om,\ep}:=\nu_{\sg\om,0}\left((\cL_{\om,0}-\cL_{\om,\ep})(\phi_{\om,0})\right)
		\end{align}
		and
		\begin{align}\label{def: eta_om}
			\eta_{\om,\ep}:=\norm{\nu_{\sg\om,0}(\cL_{\om,0}-\cL_{\om,\ep})}_{\cB_\om^*}.
		\end{align}
		\item[\mylabel{P5}{P5}] For each $\om\in\Om$ we have
		\begin{align*}
			\lim_{\ep\to 0}\eta_{\om,\ep}=0.
		\end{align*}
		\item[\mylabel{P6}{P6}] For each $\om\in\Om$ such that $\Dl_{\om,\ep}>0$ for every $\ep>0$, we have that there exists a  function $C_3:\Om\to\RR_+$ such that
		\begin{align*}
			\limsup_{\ep\to 0}\frac{\eta_{\om,\ep}}{\Dl_{\om,\ep}}:=C_3(\om) <\infty.
		\end{align*}
		Given $\om\in\Om$, if there is $\ep_0>0$ such that for each $\ep\leq \ep_0$ we have that $\Dl_{\om,\ep}=0$ then we also have that $\eta_{\om,\ep}=0$ for each $\ep\leq \ep_0$.
		\item[\mylabel{P7}{P7}]  For each $\om\in\Om$ we have
		\begin{align*}
			\lim_{\ep \to 0}\nu_{\om,\ep}(\phi_{\om,0})=1.
		\end{align*}	
		\item[\mylabel{P8}{P8}] For each $\om\in\Om$ with $\Dl_{\om,\ep}>0$ for all $\ep>0$ we have
		\begin{flalign*}
			\lim_{n\to\infty}\limsup_{\ep \to 0} \Dl_{\om,\ep}^{-1}\nu_{\sg\om,0}\lt(\lt(\cL_{\om,0}-\cL_{\om,\ep}\rt)\lt(Q_{\sg^{-n}\om,\ep}^n\phi_{\sg^{-n}\om,0}\rt)\rt)=0.
		\end{flalign*}
		\item[\mylabel{P9}{P9}] For each $\om\in\Om$ with $\Dl_{\om,\ep}>0$ for all $\ep>0$ we have the limit
		\begin{align*}
			q_{\om,0}^{(k)}:=\lim_{\ep\to 0} \frac{\nu_{\sg\om,0}\left((\cL_{\om,0}-\cL_{\om,\ep})(\cL_{\sg^{-k}\om,\ep}^k)(\cL_{\sg^{-(k+1)}\om,0}-\cL_{\sg^{-(k+1)}\om,\ep})(\phi_{\sg^{-(k+1)}\om,0})\right)}{\nu_{\sg\om,0}\left((\cL_{\om,0}-\cL_{\om,\ep})(\phi_{\om,0})\right)}
		\end{align*}
		exists for each $k\geq 0$.
	\end{enumerate}
	
	Consider the identity
	\begin{align}
		\lm_{\om,0}-\lm_{\om,\ep}&=\lm_{\om,0}\nu_{\om,0}(\phi_{\om,\ep})-\nu_{\sg\om,0}(\lm_{\om,\ep}\phi_{\sg\om,\ep})\nonumber\\
		&=\nu_{\sg\om,0}(\cL_{\om,0}(\phi_{\om,\ep}))-\nu_{\sg\om,0}(\cL_{\om,\ep}(\phi_{\om,\ep}))\nonumber\\
		&=\nu_{\sg\om,0}\left((\cL_{\om,0}-\cL_{\om,\ep})(\phi_{\om,\ep})\right).\label{diff eigenvalues identity unif}
	\end{align}
	It then follows from \eqref{diff eigenvalues identity unif}, together with \eqref{def: eta_om} and assumption \eqref{P4}, that
	\begin{align}\label{conv eigenvalues unif}
		\absval{\lm_{\om,0}-\lm_{\om,\ep}}\leq C_2(\om)\eta_{\om,\ep}.
	\end{align}
	In particular, given assumption \eqref{P5}, \eqref{conv eigenvalues unif} implies
	\begin{align}\label{limit of eigenvalues}
		\lim_{\ep\to 0}\lm_{\om,\ep}=\lm_{\om,0}
	\end{align}
	for each $\om\in\Om$.
	\begin{remark}
		Note that \eqref{P6} and \eqref{conv eigenvalues unif} imply that
		\begin{align*}
			\limsup_{\ep\to 0}\frac{\absval{\lm_{\om,0}-\lm_{\om,\ep}}}{\Dl_{\om,\ep}}\leq C_2(\om)C_3(\om)<\infty.
		\end{align*}
	\end{remark}
	For $n\geq 1$ we define the normalized operator $\~\cL_{\om,\ep}^n:\cB_\om\to\cB_{\sg^n\om}$ by
	\begin{align*}
		\~\cL_{\om,\ep}^n:=(\lm_{\om,\ep}^n)^{-1}\cL_{\om,\ep}^n
	\end{align*}
	where
	\begin{align*}
		\lm_{\om,\ep}^n:=\lm_{\om,\ep}\cdot\dots\cdot\lm_{\sg^{n-1}\om,\ep}.
	\end{align*}
	In view of assumption \eqref{P3}, induction gives
	\begin{align*}
		\~\cL_{\om,\ep}^n(f)=\nu_{\om,\ep}(f)\cdot \phi_{\sg^n\om,\ep}+Q_{\om,\ep}^n(f)
	\end{align*}
	for each $n\geq 1$ and all $f\in\cB_\om$.
	Similarly with \eqref{diff eigenvalues identity unif}, we have that
	\begin{align*}
		\lm_{\om,0}^n-\lm_{\om,\ep}^n=\nu_{\sg^n\om,0}\left(\left(\cL_{\om,0}^n-\cL_{\om,\ep}^n\right)(\phi_{\om,\ep})\right).
	\end{align*}
	We now arrive at the main result of this section.
	We prove a differentiability result for the perturbed quantities $\lm_{\om,\ep}$ as $\ep\to 0$ in the spirit of Keller and Liverani \cite{keller_rare_2009}.

	\begin{theorem}\label{thm: GRPT}
		Suppose that assumptions \eqref{P1}-\eqref{P8} hold.
		If there is some $\ep_0>0$ such that $\Dl_{\om,\ep}=0$ for $\ep\leq \ep_0$ then
		\begin{align*}
			\lm_{\om,0}=\lm_{\om,\ep},
		\end{align*}
		or if \eqref{P9} holds then
		\begin{align*}
			\lim_{\ep\to 0}\frac{\lm_{\om,0}-\lm_{\om,\ep}}{\Dl_{\om,\ep}}=1-\sum_{k=0}^{\infty}(\lm_{\sg^{-(k+1)}\om,0}^{k+1})^{-1}q_{\om,0}^{(k)}.
		\end{align*}
	\end{theorem}
	\begin{proof}

		Fix $\om\in\Om$.
		First, we note that  assumption \eqref{P6} implies that if there is some $\ep_0$ such that $\Dl_{\om,\ep}=0$ for all $\ep\leq \ep_0$ (which implies, by assumption, that $\eta_{\om,\ep}=0$), then \eqref{conv eigenvalues unif} immediately implies that
		\begin{align*}
			\lm_{\om,0}=\lm_{\om,\ep}.
		\end{align*}
		Now we suppose that $\Dl_{\om,\ep}>0$ for all $\ep>0$. 
		Using \eqref{P2}, \eqref{P3}, and \eqref{diff eigenvalues identity unif}, for each $n\geq 1$ and all $\om\in\Om$ we have
		\begin{align}
			&\nu_{\sg^{-n}\om,\ep}(\phi_{\sg^{-n}\om,0})(\lm_{\om,0}-\lm_{\om,\ep})=\nu_{\sg^{-n}\om,\ep}(\phi_{\sg^{-n}\om,0})\nu_{\sg\om,0}\left((\cL_{\om,0}-\cL_{\om,\ep})(\phi_{\om,\ep})\right)
			\nonumber\\
			&\quad
			=\nu_{\sg\om,0}\left((\cL_{\om,0}-\cL_{\om,\ep})(\nu_{\sg^{-n}\om,\ep}(\phi_{\sg^{-n}\om,0})\cdot\phi_{\om,\ep})\right)
			\nonumber\\
			&\quad
			=\nu_{\sg\om,0}\left((\cL_{\om,0}-\cL_{\om,\ep})(\nu_{\sg^{-n}\om,\ep}(\phi_{\sg^{-n}\om,0})\cdot\phi_{\om,\ep}+Q_{\sg^{-n}\om,\ep}^n(\phi_{\sg^{-n}\om,0})-Q_{\sg^{-n}\om,\ep}^n(\phi_{\sg^{-n}\om,0}))\right)
			\nonumber\\
			&\quad
			=\nu_{\sg\om,0}\left((\cL_{\om,0}-\cL_{\om,\ep})(\~\cL_{\sg^{-n}\om,\ep}^n-Q_{\sg^{-n}\om,\ep}^n)(\phi_{\sg^{-n}\om,0})\right)
			\nonumber\\
			&\quad
			=\nu_{\sg\om,0}\left((\cL_{\om,0}-\cL_{\om,\ep})(\~\cL_{\sg^{-n}\om,0}^n-\~\cL_{\sg^{-n}\om,0}^n+\~\cL_{\sg^{-n}\om,\ep}^n-Q_{\sg^{  -n}\om,\ep}^n)(\phi_{\sg^{-n}\om,0})\right)
			\nonumber\\
			&\quad
			=\nu_{\sg\om,0}\left((\cL_{\om,0}-\cL_{\om,\ep})(\phi_{\om,0})\right)
			\nonumber\\
			&\qquad
			-\nu_{\sg\om,0}\left((\cL_{\om,0}-\cL_{\om,\ep})(\~\cL_{\sg^{-n}\om,0}^n-\~\cL_{\sg^{-n}\om,\ep}^n)(\phi_{\sg^{-n}\om,0})\right)
			\nonumber\\
			&\qquad\qquad
			-\nu_{\sg\om,0}\left((\cL_{\om,0}-\cL_{\om,\ep})(Q_{\sg^{-n}\om,\ep}^n(\phi_{\sg^{-n}\om,0}))\right)
			\nonumber\\
			&\quad
			=\Dl_{\om,\ep}-\nu_{\sg\om,0}\left((\cL_{\om,0}-\cL_{\om,\ep})(\~\cL_{\sg^{-n}\om,0}^n-\~\cL_{\sg^{-n}\om,\ep}^n)(\phi_{\sg^{-n}\om,0})\right)
			\label{D_2 estimate unif}\\
			&\qquad
			-\nu_{\sg\om,0}\left((\cL_{\om,0}-\cL_{\om,\ep})(Q_{\sg^{-n}\om,\ep}^n(\phi_{\sg^{-n}\om,0}))\right).
			\label{Big O estimate unif}
		\end{align}
		Since
		\begin{align*}
			(\~\cL_{\sg^{-n}\om,0}^n-\~\cL_{\sg^{-n}\om,\ep}^n)(\phi_{\sg^{-n}\om,0})
			=
			\phi_{\om,0} - \~\cL_{\sg^{-n}\om,\ep}^n(\phi_{\sg^{-n}\om,0}),
		\end{align*}
		using a telescoping argument, the second term of \eqref{D_2 estimate unif}
		$$
		D_2:=-\nu_{\sg\om,0}\left((\cL_{\om,0}-\cL_{\om,\ep})(\~\cL_{\sg^{-n}\om,0}^n-\~\cL_{\sg^{-n}\om,\ep}^n)(\phi_{\sg^{-n}\om,0})\right)
		$$
		can be rewritten as
		\begin{align}
			D_2&=-\sum_{k=0}^{n-1}\nu_{\sg\om,0}\left((\cL_{\om,0}-\cL_{\om,\ep})(\~\cL_{\sg^{-k}\om,\ep}^k)(\~\cL_{\sg^{-(k+1)}\om,0}-\~\cL_{\sg^{-(k+1)}\om,\ep})(\phi_{\sg^{-(k+1)}\om,0})\right)
			\nonumber\\
			&
			=-\sum_{k=0}^{n-1}\nu_{\sg\om,0}\left((\cL_{\om,0}-\cL_{\om,\ep})(\~\cL_{\sg^{-k}\om,\ep}^k)(\~\cL_{\sg^{-(k+1)}\om,0}\right.
			\nonumber\\
			&\qquad
			\left.-\lm_{\sg^{-(k+1)}\om,0}^{-1}\cL_{\sg^{-(k+1)}\om,\ep}+\lm_{\sg^{-(k+1)}\om,0}^{-1}\cL_{\sg^{-(k+1)}\om,\ep}-\~\cL_{\sg^{-(k+1)}\om,\ep})(\phi_{\sg^{-(k+1)}\om,0})\right)
			\nonumber\\
			&
			=-\sum_{k=0}^{n-1}\nu_{\sg\om,0}\left((\cL_{\om,0}-\cL_{\om,\ep})(\~\cL_{\sg^{-k}\om,\ep}^k)(\~\cL_{\sg^{-(k+1)}\om,0}\right.
			\label{first summand in D2}	\\
			&\qquad
			\left.-\lm_{\sg^{-(k+1)}\om,0}^{-1}\cL_{\sg^{-(k+1)}\om,\ep}+\lm_{\sg^{-(k+1)}\om,0}^{-1}\cL_{\sg^{-(k+1)}\om,\ep})(\phi_{\sg^{-(k+1)}\om,0})\right)
			\nonumber\\
			&\qquad\qquad
			+\sum_{k=0}^{n-1}\nu_{\sg\om,0}\left((\cL_{\om,0}-\cL_{\om,\ep})(\~\cL_{\sg^{-k}\om,\ep}^k)(\~\cL_{\sg^{-(k+1)}\om,\ep})(\phi_{\sg^{-(k+1)}\om,0})\right).
			\label{second summand in D2}
		\end{align}
		Reindexing, the second summand of the above calculation, and multiplying by $1$, \eqref{second summand in D2}, can be rewritten as
		\begin{align}
			\sum_{k=0}^{n-1}&\nu_{\sg\om,0}\left((\cL_{\om,0}-\cL_{\om,\ep})(\~\cL_{\sg^{-(k+1)}\om,\ep}^{k+1})(\phi_{\sg^{-(k+1)}\om,0})\right)
			\nonumber\\
			&\quad
			=\sum_{k=1}^{n}\nu_{\sg\om,0}\left((\cL_{\om,0}-\cL_{\om,\ep})(\~\cL_{\sg^{-k}\om,\ep}^{k})(\phi_{\sg^{-k}\om,0})\right)
			\nonumber\\
			&\quad
			=\sum_{k=1}^{n}\lm_{\sg^{-k}\om,0}^{-1}\lm_{\sg^{-k}\om,0}\nu_{\sg\om,0}\left((\cL_{\om,0}-\cL_{\om,\ep})(\~\cL_{\sg^{-k}\om,\ep}^{k})(\phi_{\sg^{-k}\om,0})\right).
			\label{est1 of D2}
		\end{align}
		And \eqref{first summand in D2} from above can again be broken into two parts and then rewritten as
		\begin{align}
			&-\sum_{k=0}^{n-1}\lm_{\sg^{-(k+1)}\om,0}^{-1}\nu_{\sg\om,0}\left((\cL_{\om,0}-\cL_{\om,\ep})(\~\cL_{\sg^{-k}\om,\ep}^k)(\cL_{\sg^{-(k+1)}\om,0}-\cL_{\sg^{-(k+1)}\om,\ep})(\phi_{\sg^{-(k+1)}\om,0})\right)
			\nonumber\\
			&\quad
			-\sum_{k=1}^{n}\lm_{\sg^{-k}\om,0}^{-1}\lm_{\sg^{-k}\om,\ep}\nu_{\sg\om,0}\left((\cL_{\om,0}-\cL_{\om,\ep})(\~\cL_{\sg^{-k}\om,\ep}^{k})(\phi_{\sg^{-k}\om,0})\right),
			\label{est2 of D2}
		\end{align}
		where in the second sum we have used the fact that $\cL_{\sg^{-k}\om,\ep}=\lm_{\sg^{-k}\om,\ep}\~\cL_{\sg^{-k}\om,\ep}$.
		Altogether using \eqref{est1 of D2} and \eqref{est2 of D2}, $D_2$ we can be written as
		\begin{align}
			D_2&=
			-\sum_{k=0}^{n-1}\lm_{\sg^{-(k+1)}\om,0}^{-1}\nu_{\sg\om,0}\left((\cL_{\om,0}-\cL_{\om,\ep})(\~\cL_{\sg^{-k}\om,\ep}^k)(\cL_{\sg^{-(k+1)}\om,0}-\cL_{\sg^{-(k+1)}\om,\ep})(\phi_{\sg^{-(k+1)}\om,0})\right)
			\nonumber\\
			&\qquad
			-\sum_{k=1}^{n}\lm_{\sg^{-k}\om,0}^{-1}\lm_{\sg^{-k}\om,\ep}\nu_{\sg\om,0}\left((\cL_{\om,0}-\cL_{\om,\ep})(\~\cL_{\sg^{-k}\om,\ep}^{k})(\phi_{\sg^{-k}\om,0})\right)
			\nonumber\\
			&\qquad\quad
			+\sum_{k=1}^{n}\lm_{\sg^{-k}\om,0}^{-1}\lm_{\sg^{-k}\om,0}\nu_{\sg\om,0}\left((\cL_{\om,0}-\cL_{\om,\ep})(\~\cL_{\sg^{-k}\om,\ep}^{k})(\phi_{\sg^{-k}\om,0})\right)
			\nonumber\\
			&=
			-\sum_{k=0}^{n-1}\lm_{\sg^{-(k+1)}\om,0}^{-1}\nu_{\sg\om,0}\left((\cL_{\om,0}-\cL_{\om,\ep})(\~\cL_{\sg^{-k}\om,\ep}^k)(\cL_{\sg^{-(k+1)}\om,0}-\cL_{\sg^{-(k+1)}\om,\ep})(\phi_{\sg^{-(k+1)}\om,0})\right)
			\nonumber\\
			&\quad
			=+\sum_{k=1}^{n}\lm_{\sg^{-k}\om,0}^{-1}\left(\lm_{\sg^{-k}\om,0}-\lm_{\sg^{-k}\om,\ep}\right)\nu_{\sg\om,0}\left((\cL_{\om,0}-\cL_{\om,\ep})(\~\cL_{\sg^{-k}\om,\ep}^{k})(\phi_{\sg^{-k}\om,0})\right).\label{final est D2}
		\end{align}
		Now for each $k\geq 0$ we let
		\begin{align}\label{eq: def of q_ep}
			q_{\om,\ep}^{(k)}:=\frac{\nu_{\sg\om,0}\left((\cL_{\om,0}-\cL_{\om,\ep})(\cL_{\sg^{-k}\om,\ep}^k)(\cL_{\sg^{-(k+1)}\om,0}-\cL_{\sg^{-(k+1)}\om,\ep})(\phi_{\sg^{-(k+1)}\om,0})\right)}{\nu_{\sg\om,0}\left((\cL_{\om,0}-\cL_{\om,\ep})(\phi_{\om,0})\right)}.
		\end{align}
		Using \eqref{final est D2} we can continue our rephrasing of $\nu_{\sg^{-n}\om,\ep}(\phi_{\sg^{-n}\om,0})(\lm_{\om,0}-\lm_{\om,\ep})$ from \eqref{D_2 estimate unif} to get
		\begin{align}
			&\nu_{\sg^{-n}\om,\ep}(\phi_{\sg^{-n}\om,0})(\lm_{\om,0}-\lm_{\om,\ep})\nonumber
			\\
			&\quad=\Dl_{\om,\ep}-\sum_{k=0}^{n-1}\lm_{\sg^{-(k+1)}\om,0}^{-1}\nu_{\sg\om,0}\left((\cL_{\om,0}-\cL_{\om,\ep})(\~\cL_{\sg^{-k}\om,\ep}^k)(\cL_{\sg^{-(k+1)}\om,0}-\cL_{\sg^{-(k+1)}\om,\ep})(\phi_{\sg^{-(k+1)}\om,0})\right)
			\nonumber\\
			&\qquad
			+\sum_{k=1}^{n}\lm_{\sg^{-k}\om,0}^{-1}\left(\lm_{\sg^{-k}\om,0}-\lm_{\sg^{-k}\om,\ep}\right)\nu_{\sg\om,0}\left((\cL_{\om,0}-\cL_{\om,\ep})(\~\cL_{\sg^{-k}\om,\ep}^{k})(\phi_{\sg^{-k}\om,0})\right)
			\nonumber\\
			&\quad\qquad -\nu_{\sg\om,0}\left((\cL_{\om,0}-\cL_{\om,\ep})(Q_{\sg^{-n}\om,\ep}^n(\phi_{\sg^{-n}\om,0}))\right)
			\nonumber\\
			&\quad=\Dl_{\om,\ep}\left(1-\sum_{k=0}^{n-1}\lm_{\sg^{-(k+1)}\om,0}^{-1}(\lm_{\sg^{-k}\om,\ep}^k)^{-1}q_{\om,\ep}^{(k)}\right)
			\nonumber\\
			&\qquad
			+\sum_{k=1}^{n}\lm_{\sg^{-k}\om,0}^{-1}\left(\lm_{\sg^{-k}\om,0}-\lm_{\sg^{-k}\om,\ep}\right)\nu_{\sg\om,0}\left((\cL_{\om,0}-\cL_{\om,\ep})(\~\cL_{\sg^{-k}\om,\ep}^{k})(\phi_{\sg^{-k}\om,0})\right)
			\nonumber\\
			&\quad\qquad 	-\nu_{\sg\om,0}\left((\cL_{\om,0}-\cL_{\om,\ep})(Q_{\sg^{-n}\om,\ep}^n(\phi_{\sg^{-n}\om,0}))\right).\label{main thm long calc}
		\end{align}
		Dividing the calculation culminating in \eqref{main thm long calc} by $\Dl_{\om,\ep}$ on both sides gives
		\begin{align}
			\label{maineq}
			&\nu_{\sg^{-n}\om,\ep}(\phi_{\sg^{-n}\om,0})\frac{\lm_{\om,0}-\lm_{\om,\ep}}{\Dl_{\om,\ep}}
			\\
			&\quad=1-\sum_{k=0}^{n-1}\lm_{\sg^{-(k+1)}\om,0}^{-1}(\lm_{\sg^{-k}\om,\ep}^k)^{-1}q_{\om,\ep}^{(k)}
			\nonumber\\
			&\quad
			+\Dl_{\om,\ep}^{-1}\sum_{k=1}^{n}\lm_{\sg^{-k}\om,0}^{-1}\left(\lm_{\sg^{-k}\om,0}-\lm_{\sg^{-k}\om,\ep}\right)\nu_{\sg\om,0}\left((\cL_{\om,0}-\cL_{\om,\ep})(\~\cL_{\sg^{-k}\om,\ep}^{k})(\phi_{\sg^{-k}\om,0})\right)
			\label{final calc second summand}\\
			&\qquad -	\Dl_{\om,\ep}^{-1}\nu_{\sg\om,0}\left((\cL_{\om,0}-\cL_{\om,\ep})(Q_{\sg^{-n}\om,\ep}^n(\phi_{\sg^{-n}\om,0}))\right).
			\label{final calc third summand}
		\end{align}
		Assumption \eqref{P8} ensures that \eqref{final calc third summand} goes to zero as $\ep\to 0$ and $n\to\infty$. Now, using \eqref{def: eta_om}, \eqref{conv eigenvalues unif}, \eqref{P1}, and \eqref{P4} we bound \eqref{final calc second summand} by
		\begin{align}
			&\Dl_{\om,\ep}^{-1}\sum_{k=1}^n\absval{\lm_{\sg^{-k}\om,0}}^{-1}\absval{\lm_{\sg^{-k}\om,0}-\lm_{\sg^{-k}\om,\ep}}\eta_{\om,\ep}\norm{\~\cL_{\sg^{-k}\om,\ep}^k(\phi_{\sg^{-k}\om,0})}_{\cB_{\om}}
			\nonumber\\
			&\qquad\qquad
			\leq
			\frac{\eta_{\om,\ep}}{\Dl_{\om,\ep}}
			\sum_{k=1}^nC_2(\sg^{-k}\om)\absval{\lm_{\sg^{-k}\om,0}}^{-1}\eta_{\sg^{-k}\om,\ep}\absval{\lm_{\sg^{-k}\om,\ep}^k}^{-1}\norm{\cL_{\sg^{-k}\om,\ep}^k(\phi_{\sg^{-k}\om,0})}_{\cB_{\om}}
			\nonumber\\
			&\qquad\qquad
			\leq
			\frac{\eta_{\om,\ep}}{\Dl_{\om,\ep}}
			\sum_{k=1}^n(C_2(\sg^{-k}\om))^2C_1^k(\sg^{-k}\om)\absval{\lm_{\sg^{-k}\om,0}}^{-1}\eta_{\sg^{-k}\om,\ep}\absval{\lm_{\sg^{-k}\om,\ep}^k}^{-1}.
			\label{theorem sum est}
		\end{align}
		In view of \eqref{P5}, \eqref{P6}, and \eqref{limit of eigenvalues}, for fixed $n$, we may continue from \eqref{theorem sum est} and let $\ep\to 0$ to see that
		\begin{align}\label{second summand goes to 0 for final calc}
			\lim_{\ep \to 0}
			\frac{\eta_{\om,\ep}}{\Dl_{\om,\ep}}
			\sum_{k=1}^n(C_2(\sg^{-k}\om))^2C_1^k(\sg^{-k}\om)\absval{\lm_{\sg^{-k}\om,0}}^{-1}\eta_{\sg^{-k}\om,\ep}\absval{\lm_{\sg^{-k}\om,\ep}^k}^{-1}=0.
		\end{align}
		In light of \eqref{P6}, \eqref{P7}, 
		\eqref{second summand goes to 0 for final calc}, and \eqref{maineq}--\eqref{final calc third summand} together with \eqref{P8} and \eqref{P9}, we see that first letting $\ep\to 0$ and then $n\to\infty$ gives
		\begin{align*}
			\lim_{\ep\to 0}\frac{\lm_{\om,0}-\lm_{\om,\ep}}{\Dl_{\om,\ep}}
			= 1-\sum_{k=0}^{\infty}(\lm_{\sg^{-(k+1)}\om,0}^{k+1})^{-1}q_{\om,0}^{(k)}
		\end{align*}
		as desired.
	\end{proof}
	In the sequel we will refer to the quantity on the right hand side of the last equation in the proof of the previous theorem by $\ta_{\om,0}$, i.e. we set
	\begin{align}\label{eq: def of theta_0}
		\ta_{\om,0}:=1-\sum_{k=0}^{\infty}(\lm_{\sg^{-(k+1)}\om,0}^{k+1})^{-1}q_{\om,0}^{(k)}.
	\end{align}

	\section{Random Open Systems}\label{sec:ROS}
	In this section we introduce the general setup of random open systems to which we can apply our perturbation theory from Section~\ref{Sec: Gen Perturb Setup}.
	
	We begin with a probability space $(\Om,\sF,m)$ and an ergodic, invertible map $\sg:\Om\to\Om$ which preserves the measure $m$, i.e.
	\begin{align*}
		m\circ\sg^{-1}=m.
	\end{align*}
	For each $\om\in \Om$ let $\cJ_{\om,0}$ be a closed subset of a complete metrisable space $X$ such that the map
	\begin{align*}
		\Om\ni \om\longmapsto\cJ_{\om,0}
	\end{align*}
	is a closed random set, i.e. $\cJ_{\om,0}\sub X$ is closed for each $\om\in\Om$ and the map $\om\mapsto\cJ_{\om,0}$ is measurable  
	(see \cite{crauel_random_2002}), and we consider the maps
	\begin{align*}
		T_\om:\cJ_{\om,0}\to\cJ_{\sg\om,0}.
	\end{align*}
	By $T_\om^N:\cJ_{\om,0}\to\cJ_{\sg^N\om,0}$ we mean the $N$-fold composition 
	\begin{align*}
		T_{\sg^N\om}\circ\dots\circ T_\om:\cJ_{\om,0}\to\cJ_{\sg^N\om,0}.
	\end{align*}
	Given a set $A\sub\cJ_{\sg^N\om,0}$ we let
	\begin{align*}
		T_\om^{-N}(A):=\set{x\in\cJ_{\om,0}:T_\om^N(x)\in A}
	\end{align*}
	denote the inverse image of $A$ under the map $T_\om^N$ for each $\om\in\Om$ and $N\geq 1$.
	Now let
	\begin{align*}
		\cJ_0:=\bigcup_{\om\in\Om}\set{\om}\times\cJ_{\om,0}\sub \Om\times X,
	\end{align*}
	and define the induced skew-product map $T:\cJ_0\to\cJ_0$ by
	\begin{align*}
		T(\om,x)=(\sg\om,T_\om(x)).
	\end{align*}
	Let $\sB$ denote the Borel $\sg$-algebra of $X$ and let $\sF\otimes\sB$ be the product $\sg$-algebra on $\Om\times X$. We suppose the following:
	\begin{enumerate}[align=left,leftmargin=*,labelsep=\parindent]
		\item[\mylabel{M1}{M1}] the map $T:\cJ_0\to\cJ_0$ is measurable with respect to $\sF\otimes\sB$.
	\end{enumerate} 
	\begin{definition}\label{def: random prob measures}	
		A measure $\mu$ on $\Om\times X$ with respect to the product $\sg$-algebra $\sF\otimes\sB$ is said to be \textit{random measure} relative to $m$ if it has marginal $m$, i.e. if
		$$
		\mu\circ\pi^{-1}_1=m.
		$$ 
		The disintegrations $\set{\mu_\om}_{\om\in\Om}$ of $\mu$ with respect to the partition $\lt(\{\om\}\times X\rt)_{\om\in\Om}$ satisfy the following properties:
		\begin{enumerate}
			\item For every $B\in\sB$, the map $\Om\ni\om\longmapsto\mu_\om(B)\in X$ is measurable, 
			\item For $m$-a.e. $\om\in\Om$, the map $\sB\ni B\longmapsto\mu_\om(B)\in X$ is a Borel measure.
		\end{enumerate}
		We say that the random measure $\mu=\set{\mu_\om}_{\om\in\Om}$ is a \textit{random probability measure} if for $m$-a.e. $\om\in\Om$ the fiber measure $\mu_\om$ is a probability measure. Given a set $Y=\cup_{\om\in\Om}\set{\om}\times Y_\om\sub\Om\times X$, we say that the random measure $\mu=\set{\mu_\om}_{\om\in\Om}$ is supported in $Y$ if $\supp(\mu)\sub Y$ and consequently $\supp(\mu_\om)\sub Y_\om$ for $m$-a.e. $\om\in\Om$. We let $\cP_\Om(Y)$ denote the set of all random probability measures supported in $Y$. We will frequently denote a random measure $\mu$ by $\set{\mu_\om}_{\om\in\Om}$.
	\end{definition}
	The following proposition from Crauel \cite{crauel_random_2002}, shows a random probability measure $\set{\mu_\om}_{\om\in\Om}$ on $\cJ_0$ uniquely identifies a probability measure on $\cJ_0$.
	\begin{proposition}[\cite{crauel_random_2002}, Propositions 3.3]\label{prop: random measure equiv}
		If $\set{\mu_\om}_{\om\in\Om}\in\cP_\Om(\cJ_0)$ is a random probability measure on $\cJ_0$, then for every bounded measurable function $f:\cJ_0\to\RR$, the function 
		$$
		\Om\ni\om\longmapsto \int_{\cJ_{\om,0}} f(\om,x) \, d\mu_\om(x)
		$$ 
		is measurable and 
		$$
		\sF\otimes\sB\ni A\longmapsto\int_\Om \int_{\cJ_{\om,0}} \ind_A(\om,x) \, d\mu_\om(x)\, dm(\om)
		$$
		defines a probability measure on $\cJ_0$.
	\end{proposition}
	
	We assume there exists a family of Banach spaces $\set{\cB_\om,\norm{\spot}_{\cB_\om}}_{\om\in\Om}$ of real-valued functions on each $\cJ_{\om,0}$ and a family of potentials $\set{g_{\om,0}}_{\om\in\Om}$ with $g_{\om,0}\in\cB_\om$ such that the (fiberwise) transfer operator $\cL_{\om,0}:\cB_\om\to\cB_{\sg\om}$ given by
	\begin{align*}
		\cL_{\om,0}(f)(x):=\sum_{y\in T_\om^{-1}(x)}f(y)g_{\om,0}(y), \quad f\in\cB_\om, \, x\in\cJ_{\sg\om,0}
	\end{align*}
	is well defined.
	Using induction we see that iterates $\cL_{\om,0}^N:\cB_{\om}\to\cB_{\sg^N\om}$ of the transfer operator are given by
	\begin{align*}
		\cL_{\om,0}^N(f)(x):=\sum_{y\in T_\om^{-N}(x)}f(y)g_{\om,0}^{(N)}(y), \quad f\in\cB_\om, \, x\in\cJ_{\sg^N\om,0},
	\end{align*}
	where 
	$$
	g_{\om,0}^{(N)}:=\prod_{j=0}^{N-1}g_{\sg^j\om,0}\circ T_\om^j.
	$$
	We let $\cB$ denote the space of functions $f:\cJ_0\to\RR$ such that $f_\om\in\cB_\om$ for each $\om\in\Om$ and we define the global transfer operator $\cL_0:
	\cB\to \cB$ by $(\cL_0 f)_\om(x):=\cL_{\sg^{-1}\om,0}f_{\sg^{-1}\om}(x)$ for $f\in\cB$ and $x\in\cJ_{\om,0}$.
	We assume the following measurability assumption:
	\begin{enumerate}[align=left,leftmargin=*,labelsep=\parindent]
		\item[\mylabel{M2}{M2}] For every measurable function $h \in \cB$, the map 
		$(\om, x) \mapsto (\cL_0 h)_\om(x)$ is measurable.
	\end{enumerate}
	We suppose the following condition on the existence of a closed conformal measure. 
	\begin{enumerate}[align=left,leftmargin=*,labelsep=\parindent]
		\item[\mylabel{CCM}{CCM}] There exists a random probability  measure $\nu_0=\set{\nu_{\om,0}}_{\om\in\Om}\in \cP_\Om(\cJ_0)$ and measurable functions $\lm_0:\Om \to\RR\bs\set{0}$ and $\phi_0:\cJ_0\to (0,\infty)$ such that 
		\begin{align*}
			\cL_{\om,0}(\phi_{\om,0})=\lm_{\om,0}\phi_{\sg\om,0}
			\quad\text{ and }\quad
			\nu_{\sg\om,0}(\cL_{\om,0}(f))=\lm_{\om,0}\nu_{\om,0}(f)
		\end{align*}
		for all $f\in\cB_\om$ where $\phi_{\om,0}(\spot):=\phi_0(\om,\spot)$. Furthermore, we suppose that the fiber measures $\nu_{\om,0}$ are non-atomic and that $\lm_{\om,0}:=\nu_{\sg\om,0}(\cL_{\om,0}\ind)$ with $\log\lm_{\om,0}\in L^1(m)$. 
		We then define the random probability measure
		$\mu_0$ on $\cJ_0$ by
		\begin{align}\label{eq: def of mu_om,0}
			\mu_{\om,0}(f):=\int_{\cJ_{\om,0}} f\phi_{\om,0} \ d\nu_{\om,0},  \qquad f\in L^1(\cJ_{\om,0}, \nu_{\om,0}).
		\end{align}
	\end{enumerate}
	From the definition, one can easily show that $\mu_0$ is $T$-invariant, that is,  
	\begin{align}\label{eq: mu_om,0 T invar}
		\int_{\cJ_{\om,0}} f\circ T_\om \ d\mu_{\om,0}
		=
		\int_{\cJ_{\sg\om,0}} f \ d\mu_{\sg\om,0}, \qquad f\in L^1(\cJ_{\sg\om,0}, \mu_{\sg\om,0}).
	\end{align}
	\begin{remark}
		Our Assumption \eqref{CCM} has been shown to hold in several random settings: random interval maps \cite{AFGTV20,AFGTV-IVC}, random subshifts \cite{Bogenschutz_RuelleTransferOperator_1995a,mayer_countable_2015}, random distance expanding maps \cite{mayer_distance_2011}, random polynomial systems \cite{Bruck_Generalizediteration_2003}, and random transcendental maps \cite{mayer_random_2018}.
	\end{remark}
	We will also assume that the fiberwise Banach spaces $\cB_\om\sub L^\infty(\cJ_{\om,0}, \nu_{\om,0})$ and that there exists a measurable $m$-a.e. finite function $K:\Om\to[1,\infty)$ such that
	\begin{enumerate}[align=left,leftmargin=*,labelsep=\parindent]
		\item[\mylabel{B}{B}] $\norm{f}_{\infty,\om}\leq K_\om\norm{f}_{\cB_\om}$ 
	\end{enumerate}	
	for all $f\in\cB_\om$ and each $\om\in\Om$, where $\|\cdot\|_{\infty,\om}$ denotes the supremum norm with respect to $\nu_{\om,0}$. 
	We are now ready to introduce holes into the closed systems. 
	\subsection{Open Systems}\label{sec: open systems}
	
	For each $\ep>0$ we let $H_\ep\sub \cJ_0$ be measurable with respect to the product $\sg$-algebra $\sF\otimes\sB$ on $\cJ_0$ such that 
	\begin{enumerate}[align=left,leftmargin=*,labelsep=\parindent]
		\item[\mylabel{A}{A}]
		$H_\ep'\sub H_\ep$ for each $0<\ep'\leq \ep$.
	\end{enumerate}
	For each $\om\in\Om$ the sets $H_{\om,\ep}\sub \cJ_{\om,0}$ are uniquely determined by the condition that 
	\begin{align*}
		\set{\om}\times H_{\om,\ep}=H_\ep\cap\lt(\set{\om}\times \cJ_{\om,0}\rt).
	\end{align*}
	Equivalently we have 
	\begin{align*}
		H_{\om,\ep}=\pi_2(H_\ep\cap(\set{\om}\times \cJ_{\om,0})),
	\end{align*}
	where $\pi_2:\cJ_0\to \cJ_{\om,0}$ is the projection onto the second coordinate. By definition we have that the sets $H_{\om,\ep}$ are $\nu_{\om,0}$-measurable, and \eqref{A} implies that 
	\begin{enumerate}[align=left,leftmargin=*,labelsep=\parindent]
		\item[\mylabel{A'}{A'}] $H_{\om,\ep'}\sub H_{\om,\ep}$ for each $\ep'\leq \ep$ and each $\om\in\Om$.
	\end{enumerate}
	For each $\ep>0$ we set 
	$$
	    \Om_{+,\ep}:=\set{\om\in\Om:\mu_{\om,0}(H_{\om,\ep})>0}
	$$ 
	and then define
	\begin{align}
	    \label{def Om+}
	    \Om_+:=\bigcap_{\ep>0}\Om_{+,\ep}.
	\end{align}
	\begin{remark}
	    Note that the set $\Om_+$ is measurable as it is the intersection of a decreasing family of measurable sets and it is not necessarily $\sg$-invariant.
	\end{remark}
	Now define
	\begin{align*}
		\cJ_{\om,\ep}:=\cJ_{\om,0}\bs H_{\om,\ep},
	\end{align*}
	and let
	\begin{align*}
		\cJ_\ep:=\bigcup_{\om\in\Om}\set{\om}\times\cJ_{\om,\ep}.
	\end{align*}
	For each $\om\in\Om$, $N\geq 0$, and $\ep>0$ we define
	\begin{align*}
		X_{\om,N,\ep}:&=\set{x\in\cJ_{\om,0}: T_\om^j(x)\notin H_{\sg^j\om,\ep} \text{ for all } 0\leq j\leq N }
		=\bigcap_{j=0}^N T_\om^{-j}\left(\cJ_{\sg^j\om,\ep}\right)
	\end{align*}
	to be the set of points in $\cJ_{\om,\ep}$ which survive, i.e. those points whose trajectories do not land in the holes, for $N$ iterates. We then naturally define
	\begin{align*}
		X_{\om,\infty,\ep}:=\bigcap_{N=0}^\infty X_{\om,N,\ep}=\bigcap_{N=0}^\infty T_\om^{-N}(\cJ_{\sg^N\om,\ep})
	\end{align*}
	to be the set of points which will never land in a hole under iteration of the maps $T_\om^N$ for any $N\geq 0$. We call $X_{\om,\infty,\ep}$ the \textit{$(\om,\ep)$-surviving set}. 
	Note that the sets $X_{\om,N,\ep}$ and $X_{\om,\infty,\ep}$ are forward invariant satisfying the properties
	\begin{align}
	    T_\om(X_{\om,N,\ep})\sub X_{\sg\om,N-1,\ep}
	    \qquad\text{ and }\qquad
	    T_\om(X_{\om,\infty,\ep})\sub X_{\sg\om,\infty,\ep}.
	\label{surv set forw inv}
	\end{align}
	Now for any $0\leq \al\leq \infty$ we set
	\begin{align*}
		\hat{X}_{\om,\al,\ep}:=\ind_{X_{\om,\al,\ep}}
		= 
		\prod_{j=0}^{\al}\ind_{\cJ_{\sg^j\om,\ep}}\circ T_\om^j.
	\end{align*}
	The global surviving set is defined as
	\begin{align*}
		\cX_{\al,\ep}:=\bigcup_{\om\in\Om}\set{\om}\times X_{\om,\al,\ep}
	\end{align*}
	for each $0\leq \al\leq \infty$.
	Then $\cX_{\infty,\ep}\sub\cJ_\ep$ is precisely the set of points that survive under forward iteration of the skew-product map $T$. We will assume that the fiberwise survivor sets are nonempty:
		\begin{enumerate}[align=left,leftmargin=*,labelsep=\parindent]
		\item[\mylabel{X}{cond X}]
		For each $\ep>0$ sufficiently small we have $X_{\om,\infty,\ep}\neq\emptyset$ for $m$-a.e. $\om\in\Om$.
	\end{enumerate}
	As an immediate consequence of \eqref{cond X} we have that $\cX_{\infty,\ep}\neq\emptyset$.  
	\begin{remark}\label{rem check X}
	    Note that since \eqref{A} implies that $X_{\om,\infty,\ep'}\bus X_{\om,\infty,\ep}$ for all $\ep'<\ep$, \eqref{cond X} holds if there exists $\ep>0$ such that $X_{\om,\infty,\ep}\neq\emptyset$ for $m$-a.e. $\om\in\Om$. Furthermore, since $T_\om(X_{\om,\infty,\ep})\sub X_{\sg\om,\infty,\ep}$, if $X_{\om,\infty,\ep}\neq\emptyset$ then $X_{\sg^N\om,\infty,\ep'}\neq\emptyset$ for each $N\geq 1$ and $\ep'\leq \ep$. As $\cX_{\infty,\ep}$ is forward invariant we have that $X_{\om,\infty,\ep}\neq\emptyset$ not only implies that $\cX_{\infty,\ep}\neq\emptyset$, but also that $\cX_{\infty,\ep}$ is infinite. 
	\end{remark}
	The following proposition presents a setting in which the survivor set is nonempty for random piecewise continuous open dynamics. 
	\begin{proposition}\label{prop surv nonemp}
	    Suppose that for $m$-a.e. $\om\in\Om$ 
	    there exist $V_\om,U_{\om,1}, \dots, U_{\om,k_\om}\sub\cJ_{\om,\ep}$ nonempty compact subsets such that for $m$-a.e. $\om\in\Om$
	    \begin{enumerate}
	        \item $T_\om\rvert_{U_\om,j}$ is continuous for each $1\leq j\leq k_\om$,
	        \item $T_\om(U_\om)\bus V_{\sg\om}$, where $U_\om:=\cup_{j=1}^{k_\om}U_{\om,j}\sub V_\om$ for each $\om$.
	    \end{enumerate}
	    Then $X_{\om,\infty,\ep}\neq\emptyset$ for $m$-a.e. $\om\in\Om$ and consequently $\cX_{\infty,\ep}\neq\emptyset$.
	    Furthermore, if 
	    $$
	        m(\set{\om\in\Om: k_\om>1})>0
	    $$ 
	    then for $m$-a.e. $\om\in\Om$ the survivor set $X_{\om,\infty,\ep}$ is uncountable. 
	\end{proposition}
	\begin{proof}
	    For each $1\leq j\leq k_\om$ let $T_{\om,j}:U_{\om,j}\to\cJ_{\sg\om,0}$ denote the continuous map $T_\om\rvert_{U_{\om,j}}$, and let $T_{\om,U}:U_\om\to\cJ_{\sg\om,0}$ denote the map $T_\om\rvert_{U_\om}$. Since $V_{\sg\om}$ is compact,  for each $1\leq j\leq k_\om$ we have that $T_{\om,j}^{-1}(V_{\sg\om})$ is a nonempty compact subset of $U_{\om,j}$. 
	    Given $N\geq 1$ let $\gm=\gm_0\gm_1\dots\gm_{N-1}$ be an $N$-length word with $1\leq \gm_j\leq k_{\sg^j\om}$ for each $0\leq j\leq N-1$. Let $\Gm_{\om,N}$ denote the finite collection of all such words of length $N$. Then for each $N\geq 1$ and each $\gm\in\Gm_{\om,N}$
	    $$
	        T_{\om,\gm}^{-N}(V_{\sg^N\om}):=
	        T_{\om,\gm_0}^{-1}\circ\dots\circ T_{\sg^{N-1}\om,\gm_{N-1}}^{-1}(V_{\sg^N\om})
	        \sub U_{\om,\gm_0}
	    $$
	    is compact. 
	    Furthermore, $T_{\om,\gm}^{-N}(V_{\sg^N\om})$ forms a decreasing sequence in $U_{\om,\gm_0}$.  Hence, we see that 
	    $$
	        T_{\om,U}^{-N}(V_{\sg^N\om})=\bigcup_{\gm\in\Gm_{\om,N}}T_{\om,\gm}^{-N}(V_{\sg^N\om})
	    $$
	    forms a decreasing sequence of compact subsets of $U_\om$. 
        Thus,
	    $$
	        X_{\om,\infty,\ep}=\bigcap_{N=0}^\infty T_\om^{-N}(V_{\sg^N\om})
	        \bus \bigcap_{N=0}^\infty T_{\om,U}^{-N}(V_{\sg^N\om})\neq\emptyset
	    $$ 
	    as desired. The final claim follows from the ergodicity of $\sg$, which ensures that for a.e. $\om$ there are infinitely many $j\in\NN$ such that $k_{\sg^j\om}>1$, and the usual bijection between a point in $X_{\om,\infty,\ep}$ and an infinite word $\gm$ in the fiberwise sequence space $\Sg_\om:=\set{\gm=\gm_1\gm_2\dots: 1\leq \gm_j\leq k_{\sg^j\om}}$.
	\end{proof}
	Now we define the perturbed operator $\cL_{\om,\ep}:\cB_\om\to\cB_{\sg\om}$ by
	\begin{align*}
		\cL_{\om,\ep}(f):=\cL_{\om,0}\left(f\cdot\ind_{\cJ_{\om,\ep}}\right)
		=\sum_{y\in T_\om^{-1}(x)}f(y)\ind_{\cJ_{\om,\ep}}g_{\om,0}(y)
		=\sum_{y\in T_\om^{-1}(x)}f(y)g_{\om,\ep}(y)
		, \quad f\in\cB_\om,
	\end{align*}
	where for each $\om\in\Om$ and $\ep\geq 0$ we define $g_{\om,\ep}:=g_{\om,0}\ind_{\cJ_{\om,\ep}}$, and, similarly, for each $N\in\NN$, 
	$$
	g_{\om,\ep}^{(N)}:=\prod_{j=0}^{N-1}g_{\sg^j\om,\ep}\circ T_\om^j.
	$$
	Note that measurability of $H_\ep\sub \cJ_0$ and condition \eqref{M2} imply  that for every $\ep>0$ and $f\in\cB$ the map $(\om,x)\mapsto (\cL_{\ep} f)_\om (x)$ is also measurable.
	Iterates of the perturbed operator $\cL_{\om,\ep}^N:\cB_{\om}\to\cB_{\sg^N\om}$ are given by
	\begin{align*}
		\cL_{\om,\ep}^N:=\cL_{\sg^{N-1}\om,\ep}\circ\dots\circ\cL_{\om,\ep},
	\end{align*}
	which, using induction, we may write as
	\begin{align}\label{eq: iterate pert tr op}
		\cL_{\om,\ep}^N(f)=\cL_{\om,0}^N\left(f\cdot\hat{X}_{\om,N-1,\ep}\right), \qquad f\in\cB_\om.
	\end{align}
For every $\ep\geq 0$ we let 
$$
    \~\cL_{\om,\ep}:=\lm_{\om,\ep}^{-1}\cL_{\om,\ep}.
$$ 
\begin{definition}\label{def ROS}
	We will call a closed random dynamical system that satisfies the measurability conditions \eqref{M1}, \eqref{M2}, \eqref{CCM}, and \eqref{B} a \textit{random open system} if 
	assumptions \eqref{A} and \eqref{cond X} are also satisfied. 
	We let $(\mathlist{\bcomma}{\Om, m, \sg, \cJ_0, T, \cB, \cL_0, \nu_0, \phi_0, H_\ep})$
	denote the random open system generated by the random maps $T_\om:\cJ_{\om,0}\to\cJ_{\sg\om,0}$ and random holes $H_{\om,\ep}\sub\cJ_{\om,0}$.
\end{definition}

\subsection{Some of the Terms from Sequential Perturbation Theorem}

In this short section we develop a more thorough understanding of some of the vital terms in the general sequential perturbation theorem of Section \ref{Sec: Gen Perturb Setup} in the setting of random open systems.
We begin by calculating the quantities $\eta_{\om,\ep}$ and $\Dl_{\om,\ep}$ from Section~\ref{Sec: Gen Perturb Setup}. In particular, we have that 
\begin{align}
	\Dl_{\om,\ep}&:=\nu_{\sg\om,0}\left((\cL_{\om,0}-\cL_{\om,\ep})(\phi_{\om,0})\right)
	\nonumber\\
	&=\nu_{\sg\om,0}\left(\cL_{\om,0}(\phi_{\om,0}\cdot\ind_{H_{\om,\ep}})\right)
	\nonumber\\
	&=\lm_{\om,0}\cdot\nu_{\om,0}(\phi_{\om,0}\cdot\ind_{H_{\om,\ep}})
	\nonumber\\
	&=\lm_{\om,0}\cdot\mu_{\om,0}(H_{\om,\ep})
	\label{eq: Dl = mu H}
\end{align}
and
\begin{align}
	\eta_{\om,\ep}:&=\norm{\nu_{\sg\om,0}\left(\cL_{\om,0}-\cL_{\om,\ep}\right)}_{\cB_\om}
	\nonumber
	\\
	&=\sup_{\norm{\psi}_{\cB_\om}\leq 1}\nu_{\sg\om,0}\left(\cL_{\om,0}(\psi\cdot\ind_{H_{\om,\ep}})\right)
	\nonumber
	\\
	&=\lm_{\om,0}\cdot\sup_{\norm{\psi}_{\cB_\om}\leq 1}\nu_{\om,0}\left(\psi\cdot\ind_{H_{\om,\ep}}\right)
	\nonumber
	\\
	&\leq K_\om\lm_{\om,0}\cdot\nu_{\om,0}\left(H_{\om,\ep}\right),
	\label{etaineq}
\end{align}
where the last line follows from \eqref{B}.

For $\ep>0$ and $\mu_{\om,0}(H_{\om,\ep})>0$, calculating $q_{\om,\ep}^{(k)}$ in this setting gives
\begin{align*}
	&q_{\om,\ep}^{(k)}
	:=\frac
	{\nu_{\sg\om,0}\left((\cL_{\om,0}-\cL_{\om,\ep})(\cL_{\sg^{-k}\om,\ep}^k)(\cL_{\sg^{-(k+1)}\om,0}-\cL_{\sg^{-(k+1)}\om,\ep})(\phi_{\sg^{-(k+1)}\om,0})\right)}
	{\nu_{\sg\om,0}\left((\cL_{\om,0}-\cL_{\om,\ep})(\phi_{\om,0})\right)}
	\\
	&=\frac{\lm_{\sg^{-(k+1)}\om,0}^{k+1}\cdot\mu_{\sg^{-(k+1)}\om,0}\left(
		T_{\sg^{-(k+1)}\om}^{-(k+1)}(H_{\om,\ep})
		\cap\left(\bigcap_{j=1}^k T_{\sg^{-(k+1)}\om}^{-(k+1)+j} (H_{\sg^{-j}\om,\ep}^c)\right)
		\cap H_{\sg^{-(k+1)}\om,\ep}
		\right)
	}
	{\mu_{\om,0}(H_{\om,\ep})}
	\\
	&=\frac{\lm_{\sg^{-(k+1)}\om,0}^{k+1}\cdot\mu_{\sg^{-(k+1)}\om,0}\left(
		T_{\sg^{-(k+1)}\om}^{-(k+1)}(H_{\om,\ep})
		\cap\left(\bigcap_{j=1}^k T_{\sg^{-(k+1)}\om}^{-(k+1)+j} (H_{\sg^{-j}\om,\ep}^c)\right)
		\cap H_{\sg^{-(k+1)}\om,\ep}
		\right)
	}
	{\mu_{\sg^{-(k+1)}\om,0}\left(
		T_{\sg^{-(k+1)}\om}^{-(k+1)}(H_{\om,\ep})\right)}.		
\end{align*}
For notational convenience we define the quantity $\hat q_{\om,\ep}^{(k)}$ by 
\begin{align}\label{def of hat q}
	\hat q_{\om,\ep}^{(k)}:=
	\frac{\mu_{\sg^{-(k+1)}\om,0}\left(
		T_{\sg^{-(k+1)}\om}^{-(k+1)}(H_{\om,\ep})
		\cap\left(\bigcap_{j=1}^k T_{\sg^{-(k+1)}\om}^{-(k+1)+j} (H_{\sg^{-j}\om,\ep}^c)\right)
		\cap H_{\sg^{-(k+1)}\om,\ep}
		\right)
	}
	{\mu_{\sg^{-(k+1)}\om,0}\left(
		T_{\sg^{-(k+1)}\om}^{-(k+1)}(H_{\om,\ep})\right)},
\end{align}
and thus we have that 
\begin{align*}
	\hat q_{\om,\ep}^{(k)}=\left(\lm_{\sg^{-(k+1)}\om,0}^{k+1}\right)^{-1} q_{\om,\ep}^{(k)}.
\end{align*}
In light of \eqref{def of hat q}, one can think of $\hat q_{\om,\ep}^{(k)}$ as the conditional probability (on the fiber $\sg^{-(k+1)}\om$) of a point starting in the hole $H_{\sg^{-(k+1)}\om,\ep}$, leaving and avoiding holes for $k$ steps, and finally landing in the hole $H_{\om,\ep}$ after exactly $k+1$ steps conditioned on the trajectory of the point landing in $H_{\om,\ep}$.
Similarly, for $\om\in\Om_+$, we set 
\begin{align*}
	\hat q_{\om,0}^{(k)}:=\left(\lm_{\sg^{-(k+1)}\om,0}^{k+1}\right)^{-1} q_{\om,0}^{(k)}.
\end{align*}

\section{Quenched  perturbation theorem and escape rate asymptotics for random open systems}
\label{sec:goodrandom}
In this section we introduce versions of the assumptions \eqref{P1}--\eqref{P9} tailored to random open systems. Under these assumptions we then prove a derivative result akin to Theorem~\ref{thm: GRPT} for random open systems as well as a similar derivative result for the escape rate.

Suppose $(\mathlist{\bcomma}{\Om, m, \sg, \cJ_0, T, \cB, \cL_0, \nu_0, \phi_0, H_\ep})$ is a random open system. 
We assume the following conditions hold:
\begin{enumerate}[align=left,leftmargin=*,labelsep=\parindent]
	\item[\mylabel{C1}{C1}]There exists a measurable $m$-a.e. finite function $C_1:\Om\to\RR_+$ such that for $f\in\cB_\om$ we have
	\begin{align*}
		\sup_{\ep\geq 0}\norm{\cL_{\om,\ep}(f)}_{\cB_{\sg\om}}\leq C_1(\om)\norm{f}_{\cB_\om}.
	\end{align*}
	\item[\mylabel{C2}{C2}] For each $\ep\geq 0$ there is a random measure  $\set{\nu_{\om,\ep}}_{\om\in\Om}$ supported in $\cJ_{0}$
	and measurable functions $\lm_{\ep}:\Om\to(0,\infty)$ 
	with $\log\lm_{\om,\ep}\in L^1(m)$ and $\phi_{\ep}:\cJ_0\to \RR$ such that 
	\begin{align*}
		\cL_{\om,\ep}(\phi_{\om,\ep})=\lm_{\om,\ep}\phi_{\sg\om,\ep}
		\quad\text{ and }\quad
		\nu_{\sg\om,\ep}(\cL_{\om,\ep}(f))=\lm_{\om,\ep}\nu_{\om,\ep}(f)
	\end{align*}
	for all $f\in\cB_\om$. Furthermore we assume that for $m$-a.e. $\om\in\Om$
	$$
	\nu_{\om,0}(\phi_{\om,\ep})=1
	\quad\text{ and } \quad
	\nu_{\om,0}(\ind)=1.
	$$
	\item[\mylabel{C3}{C3}] There is an operator $Q_{\om,\ep}:\cB_\om\to\cB_{\sg\om}$ such that for $m$-a.e. $\om\in\Om$ and each $f\in\cB_\om$ we have
	\begin{align*}
		\lm_{\om,\ep}^{-1}\cL_{\om,\ep}(f)=\nu_{\om,\ep}(f)\cdot\phi_{\sg\om,\ep}+Q_{\om,\ep}(f).
	\end{align*}
	Furthermore, for $m$-a.e. $\om\in\Om$ we have
	\begin{align*}
		Q_{\om,\ep}(\phi_{\om,\ep})=0
		\quad\text{ and }\quad
		\nu_{\sg\om,\ep}(Q_{\om,\ep}(f))=0.
	\end{align*}
	\item[\mylabel{C4}{C4}] 
	For each $f\in\cB$ there exist measurable functions $C_f:\Om\to(0,\infty)$ and $\al:\Om\times \NN\to(0,\infty)$ with $\al_\om(N)\to 0$ as $N\to\infty$ such that for $m$-a.e. $\om\in\Om$ and all $N\in\NN$
	\begin{align*}
		\sup_{\ep\geq 0}\norm{Q_{\om,\ep}^N f_{\om}}_{\infty,\sg^N\om}
		&\leq 
		C_f(\om)\al_\om(N)\norm{f_{\om}}_{\cB_{\om}}, 
		\\
		\sup_{\ep\geq 0}\norm{Q_{\sg^{-N}\om,\ep}^Nf_{\sg^{-N}\om}}_{\infty,\om}
		&\leq 
		C_f(\om)\al_\om(N)\norm{f_{\sg^{-N}\om}}_{\cB_{\sg^{-N}\om}}, 
	\end{align*} 
	and $\norm{\phi_{\sg^{N}\om,0}}_{\infty,\sg^N\om}\al_\om(N)\to 0$, $\norm{\phi_{\sg^{-N}\om,0}}_{\infty,\sg^{-N}\om}\al_\om(N)\to 0$ as $N\to\infty$.
	\item[\mylabel{C5}{C5}] 
	There exists a measurable $m$-a.e. finite function $C_2:\Om\to[1,\infty)$ such that 
	\begin{align*}
		\sup_{\ep\geq 0}\norm{\phi_{\om,\ep}}_{\infty,\om}\leq C_2(\om) 
		\quad\text{ and }\quad
		\norm{\phi_{\om,0}}_{\cB_\om}\leq C_2(\om).
	\end{align*}

	\item[\mylabel{C6}{C6}] 
	For $m$-a.e. $\om\in\Om$ we have
	\begin{align*}
		\lim_{\ep\to 0}\nu_{\om,0}(H_{\om,\ep})=0. 
	\end{align*}
	
	\item[\mylabel{C7}{C7}] 
	There exists a measurable $m$-a.e. finite function $C_3:\Om\to[1,\infty)$ such that  for all $\ep>0$ sufficiently small we have
	\begin{align*}
		\inf\phi_{\om,0}\geq C_3^{-1}(\om)>0
		\qquad\text{ and }\qquad
		\essinf_\om\inf\phi_{\om,\ep}\geq 0.
	\end{align*}

	\item[\mylabel{C8}{C8}] 
	For $m$-a.e. $\om\in\Om_+$ 
	we have that the limit $\hat{q}_{\om,0}^{(k)}:=\lim_{\ep\to 0} \hat{q}_{\om,\ep}^{(k)}$ exists for each $k\geq 0$, where $\hat{q}_{\om,\ep}^{(k)}$ is as in \eqref{def of hat q}.
\end{enumerate}

\begin{remark}\label{rem:scaling}
	To obtain the scaling required in \eqref{C2} and \eqref{C3}, in particular to obtain the assumption that $\nu_{\om,0}(\phi_{\om,\ep})=1$, suppose that $(\mathlist{\bcomma}{\Om, m, \sg, \cJ_0, T, \cB, \cL_0, \nu_0, \phi_0, H_\ep})$ is a random open system satisfying the following properties:
	\begin{enumerate}[align=left,leftmargin=*,labelsep=\parindent]
		\item[\mylabel{O1}{O1}] For each $\ep\geq 0$ there is a random measure $\set{\nu_{\om,\ep}'}_{\om\in\Om}$ with $\nu'_{\om,\ep}\in\cB_\om^*$, the dual space of $\cB_\om$, $\lm'_{\om,\ep}\in\CC\bs\set{0}$, and $\phi'_{\om,\ep}\in\cB_\om$ such that
		\begin{align*}
			\cL_{\om,\ep}(\phi'_{\om,\ep})=\lm'_{\om,\ep}\phi'_{\sg\om,\ep}
			\quad\text{ and }\quad
			\nu_{\sg\om,\ep}'(\cL_{\om,\ep}(f))=\lm'_{\om,\ep}\nu'_{\om,\ep}(f)
		\end{align*}
		for all $f\in\cB_\om$. 
		\item[\mylabel{O2}{O2}] There is an operator $Q_{\om,\ep}':\cB_\om\to\cB_{\sg\om}$ such that for $m$-a.e. $\om\in\Om$ and each $f\in\cB_\om$ we have
		\begin{align*}
			(\lm'_{\om,\ep})^{-1}\cL_{\om,\ep}(f)=\nu'_{\om,\ep}(f)\cdot \phi'_{\sg\om,\ep}+Q'_{\om,\ep}(f).
		\end{align*}
		Furthermore, for $m$-a.e. $\om\in\Om$ we have
		\begin{align*}
			Q'_{\om,\ep}(\phi'_{\om,\ep})=0
			\quad\text{ and }\quad
			\nu_{\sg\om,\ep}'(Q'_{\om,\ep}(f))=0.
		\end{align*}
	\end{enumerate}
	Then, for every $\ep>0$ and each $f\in\cB_\om$, by setting 
	\begin{align*}
		\phi_{\om,0}&:=\phi'_{\om,0},
		&
		\phi_{\om,\ep}&:=\frac{1}{\nu_{\om,0}(\phi'_{\om,\ep})}\cdot \phi'_{\om,\ep},
		\\
		\nu_{\om,0}(f)&:=\nu'_{\om,0}(f),
		&
		\nu_{\om,\ep}(f)&:=\nu_{\om,0}(\phi'_{\om,\ep})\nu'_{\om,\ep}(f),
		\\
		\lm_{\om,0}&:=\lm'_{\om,0},
		&
		\lm_{\om,\ep}&:=\frac{\nu_{\sg\om,0}(\phi'_{\sg\om,\ep})}{\nu_{\om,0}(\phi'_{\om,\ep})}\lm'_{\om,\ep},
		\\
		Q_{\om,0}(f)&:=Q'_{\om,0}(f),
		&
		Q_{\om,\ep}(f)&:=\frac{\nu_{\om,0}(\phi'_{\om,\ep})}{\nu_{\sg\om,0}(\phi'_{\sg\om,\ep})}Q'_{\om,\ep}(f)
	\end{align*}
	we see that \eqref{C2} and \eqref{C3} hold, and in particular, $\nu_{\om,0}(\phi_{\om,\ep})=1$. Furthermore, \eqref{O1} and \eqref{O2} together imply that $\nu'_{\om,\ep}(\phi'_{\om,\ep})=\nu_{\om,\ep}(\phi_{\om,\ep})=1$ for $m$-a.e. $\om\in\Om$. Note that the measures $\nu_{\om,\ep}$ (for $\ep>0$) described in \eqref{C2} are not necessarily probability measures and should not be confused with the conformal measures for the open systems.  
\end{remark}

Now given $N\in\NN$ and $\om\in\Om$, for $\psi_{\om}\in\cB_{\om}$ such that $\nu_{\om,0}(\psi_{\om})=1$ we have
\begin{align*}
	&\int_{X_{\om,N-1,\ep}}
	\psi_{\om}\, d\nu_{\om,0}
	=
	\int_{\cJ_{\om,0}} \psi_{\om}\cdot
	\prod_{j=0}^{N-1}\ind_{\cJ_{\sg^{j-N}\om,\ep}}\circ T_{\om}^j \, d\nu_{\om,0}\\
	&\qquad=
	\left(\lm_{\om,0}^N\right)^{-1}
	\int_{\cJ_{\sg^N\om,0}}\cL_{\om,0}^N\left(\psi_{\om}\cdot\prod_{j=0}^{N-1}\ind_{\cJ_{\sg^{j-N}\om,\ep}}\circ T_{\om}^j \right)\, d\nu_{\sg^N\om,0}\\
	&\qquad=
	\frac{\lm_{\om,\ep}^N}{\lm_{\om,0}^N}
	\int_{\cJ_{\sg^N\om,0}}\~\cL_{\om,\ep}^N\left( \psi_{\om} \right)\, d\nu_{\sg^N\om,0}\\
	&\qquad=
	\frac{\lm_{\om,\ep}^N}{\lm_{\om,0}^N}\int_{\cJ_{\sg^N\om,0}}
	\nu_{\om,\ep}(\psi_{\om})\cdot\phi_{\sg^N\om,\ep}\, d\nu_{\sg^N\om,0}
	+
	\frac{\lm_{\om,\ep}^N}{\lm_{\om,0}^N}
	\int_{\cJ_{\sg^N\om,0}}Q_{\om,\ep}^N\left( \psi_{\om} \right)\, d\nu_{\sg^N\om,0}.	
\end{align*}
Thus if $\psi_{\om}=\ind$ we have
\begin{equation}
	\label{NulamN}
	\nu_{\om,0}(X_{\om,N-1,\ep})= \frac{\lm_{\om,\ep}^N}{\lm_{\om,0}^N}\nu_{\om, \ep}(\ind)
	+\frac{\lm_{\om,\ep}^N}{\lm_{\om,0}^N}\int_{\cJ_{\sg^N\om,0}}Q_{\om,\ep}^N\left( \ind \right)\, d\nu_{\sg^N\om,0},
\end{equation}
and if $\psi_{\om}=\phi_{\om,0}$, then we have
\begin{align}
	\label{eq: mu0 of n survivor}
	\mu_{\om,0}(X_{\om,N-1,\ep})= \frac{\lm_{\om,\ep}^N}{\lm_{\om,0}^N}
	\nu_{\om,\ep}(\phi_{\om,0})
	+\frac{\lm_{\om,\ep}^N}{\lm_{\om,0}^N}\int_{\cJ_{\sg^N\om,0}}Q_{\om,\ep}^N\left( \phi_{\om,0} \right)\, d\nu_{\sg^N\om,0}.
\end{align}
\begin{remark}\label{rem lm_0>lm_ep}
    Using \eqref{diff eigenvalues identity unif} and \eqref{C7} we see that 
    \begin{align*}
        \lm_{\om,0}-\lm_{\om,\ep}
		&=\nu_{\sg\om,0}\left((\cL_{\om,0}-\cL_{\om,\ep})(\phi_{\om,\ep})\right)
        =\nu_{\sg\om,0}\left((\cL_{\om,0}(\ind_{H_{\om,\ep}}\phi_{\om,\ep})\right)
		\\
		&=\lm_{\om,0}\nu_{\om,0}(\ind_{H_{\om,\ep}}\phi_{\om,\ep})
		\geq \lm_{\om,0}\nu_{\om,0}(H_{\om,\ep})\inf\phi_{\om,\ep}\geq 0.
    \end{align*}
    Thus, we have that $\lm_{\om,0}\geq \lm_{\om,\ep}$. Note that if 
    $\nu_{\om,0}(\ind_{H_{\om,\ep}}\phi_{\om,\ep})>0$ then we have that $\lm_{\om,0}>\lm_{\om,\ep}$. Consequently, if
    $\nu_{\om,0}(H_{\om,\ep})=0$ then we have that  $\lm_{\om,0}=\lm_{\om,\ep}$.
\end{remark}
\begin{definition}\label{def: escape rate}
	Given a random probability measure $\zt$ on $\cJ_{0}$, for each $\ep>0$, we define the lower and upper escape rates on the fiber $\om\in\Om$ respectively by the following:
	\begin{align*}
		\Ul{R}_\ep(\zt_\om):=-\limsup_{N\to\infty}\frac{1}{N}\log \zt_\om(X_{\om,N,\ep})
		\quad \text{ and } \quad
		\ol{R}_\ep(\zt_\om):=-\liminf_{N\to\infty}\frac{1}{N}\log \zt_\om(X_{\om,N,\ep}).
	\end{align*}
	If $\Ul{R}_\ep(\zt_\om)=\ol{R}_\ep(\zt_\om)$ we say the escape rate exists and denote the common value by $R_\ep(\zt_\om)$. 
\end{definition}
As an immediate consequence of \eqref{NulamN}, \eqref{eq: mu0 of n survivor}, and assumptions \eqref{C4}, \eqref{C5}, and \eqref{C7} 
for each (fixed) $\ep>0$ we have that 
\begin{align}\label{cons of NulamN}
	\lim_{N\to\infty}\frac{1}{N}\log \nu_{\om,0}(X_{\om,N,\ep})=
	\lim_{N\to\infty}\frac{1}{N}\log \mu_{\om,0}(X_{\om,N,\ep})=
	\lim_{N\to\infty}\frac{1}{N}\log \lm_{\om,\ep}^N-\lim_{N\to\infty}\frac{1}{N}\log \lm_{\om,0}^N.
\end{align}
Since $\log\lm_{\om,\ep}\in L^1(m)$ for all $\ep\geq 0$ by \eqref{C2}, the following proposition follows directly from Birkhoff's Ergodic Theorem.
\begin{proposition}\label{prop: escape rates}
	Given a random open system $(\mathlist{\bcomma}{\Om, m, \sg, \cJ_0, T, \cB, \cL_0, \nu_0, \phi_0, H_\ep})$ satisfying conditions \eqref{C1}-\eqref{C7}, for $m$-a.e. $\om\in\Om$ we have that 
	\begin{align}\label{eq: lem escape rate}
		R_\ep(\nu_{\om,0})=R_\ep(\mu_{\om,0})=\int_\Om \log\lm_{\om,0} \, dm(\om) - \int_\Om\log\lm_{\om,\ep}\, dm(\om).
	\end{align}
\end{proposition}

We now begin to work toward an application of Theorem \ref{thm: GRPT} to random open systems.
The following implications are immediate: \eqref{C1}$\implies$\eqref{P1}, \eqref{C2}$\implies$\eqref{P2}, \eqref{C3}$\implies$\eqref{P3}, \eqref{C5}$\implies$\eqref{P4}, and in light of \eqref{eq: Dl = mu H} and \eqref{etaineq} we also have that \eqref{C6}$\implies$\eqref{P5} and \eqref{C7}$\implies$\eqref{P6}.
Thus, in order to ensure that Theorem~\ref{thm: GRPT} applies for the random open dynamical setting we need only to check assumptions \eqref{P7} and \eqref{P8}. We now prove the following lemma showing that \eqref{P7} and \eqref{P8} follow from assumptions \eqref{C1}-\eqref{C7}.

Recall that the set $\Om_+$, defined in \eqref{def Om+}, is the set of all fibers $\om$ such that $\mu_{\om,0}(H_{\om,\ep})>0$ for all $\ep>0$.
\begin{lemma}\label{lem: checking P7 and P8}
	Given a random open system $(\mathlist{\bcomma}{\Om, m, \sg, \cJ_0, T, \cB, \cL_0, \nu_0, \phi_0, H_\ep})$ satisfying conditions \eqref{C1}-\eqref{C7}, for $m$-a.e. $\om\in\Om_+$  we have that 
	\begin{align}\label{eq: lem check P7}
		\lim_{\ep \to 0}\nu_{\om,\ep}(\phi_{\om,0})=1
	\end{align}
	and 
	\begin{align}\label{eq: lem check P8}
		\lim_{N\to\infty}\limsup_{\ep \to 0}\Dl_{\om,\ep}^{-1} \nu_{\sg\om,0}\lt((\cL_{\om,0}-\cL_{\om,\ep})(Q_{\sg^{-N}\om,\ep}^N\phi_{\sg^{-N}\om,0})\rt)=0.
	\end{align}
\end{lemma}
\begin{proof}
	First, using \eqref{eq: Dl = mu H}, we note that if $\mu_{\om,0}(H_{\om,\ep})>0$ then so is $\Dl_{\om,\ep}$.
	To prove \eqref{eq: lem check P7}, we note that for fixed $N\in\NN$ we have 
	\begin{align}\label{eq: lim ratio times measure is 1}
		\lim_{\ep \to 0}\frac{\lm_{\om,0}^N}{\lm_{\om,\ep}^N}\mu_{\om,0}(X_{\om,N-1,\ep})=1
	\end{align}
	since $\lm_{\om,0}^N/\lm_{\om,\ep}^N\to 1$ (by \eqref{conv eigenvalues unif}, \eqref{etaineq}, \eqref{C5}, and \eqref{C6}) and non-atomicity of $\nu_{\om,0}$ \eqref{CCM} together with \eqref{C3} imply that $\mu_{\om,0}(X_{\om,N-1,\ep})\to 1$ as $\ep\to 0$. 
	Using  \eqref{eq: mu0 of n survivor} we can write 
	\begin{align}\label{eq: solve for nu eps}
		\nu_{\om,\ep}(\phi_{\om,0})
		= \frac{\lm_{\om,0}^N}{\lm_{\om,\ep}^N}\mu_{\om,0}(X_{\om,N-1,\ep})
		- 
		\nu_{\sg^N\om,0}(Q_{\om,\ep}^N(\phi_{\om,0})),
	\end{align}
	and thus using \eqref{eq: lim ratio times measure is 1} and \eqref{C4}, 
	for each $\om\in\Om$ and each $N\in\NN$ we can write 
	\begin{align*}
		\lim_{\ep\to 0}\absval{1-\nu_{\om,\ep}(\phi_{\om,0})}
		&\leq
		\lim_{\ep\to 0}
		\absval{ 1-\frac{\lm_{\om,0}^N}{\lm_{\om,\ep}^N}\mu_{\om,0}(X_{\om,N-1,\ep})}
		+
		\norm{Q_{\om,\ep}^N(\phi_{\om,0})}_{\infty,\sg^N\om}
		\\
		&\leq
		C_{\phi_0}( \om)\al_\om(N)\norm{\phi_{\om,0}}_{\cB_\om}.
	\end{align*}
	As this holds for each $N\in\NN$ and as the right-hand side of the previous equation goes to zero as $N\to\infty$, we must in fact have that  
	\begin{align*}
		\lim_{\ep\to 0}\absval{1-\nu_{\om,\ep}(\phi_{\om,0})}=0,
	\end{align*}
	which yields the first claim.
	
	Now, for the second claim, using \eqref{eq: Dl = mu H}, we note that \eqref{C7} implies 
	\begin{align*}
		\Dl_{\om,\ep}^{-1} \nu_{\sg\om,0}\lt((\cL_{\om,0}-\cL_{\om,\ep})(Q_{\sg^{-N}\om,\ep}^N\phi_{\sg^{-N}\om})\rt)
		&=
		\frac{\nu_{\om,0}\lt(\ind_{H_{\om,\ep}}\cdot Q_{\sg^{-N}\om,\ep}^N(\phi_{\sg^{-N}\om,0}) \rt) }{\mu_{\om,0}(H_{\om,\ep})}
		\\
		&\leq
		\frac{\nu_{\om,0}(H_{\om,\ep})}{\mu_{\om,0}(H_{\om,\ep})}\norm{Q_{\sg^{-N}\om,\ep}^N(\phi_{\sg^{-N}\om,0})}_{\infty,\om}
		\\
		&\leq C_3(\om)\norm{Q_{\sg^{-N}\om,\ep}^N(\phi_{\sg^{-N}\om,0})}_{\infty,\om}.
	\end{align*}
	Thus, letting $\ep\to 0$ first and then $N\to\infty$, the second claim follows from \eqref{C4}.
\end{proof}

Now recall from \eqref{eq: def of theta_0} that if $\mu_{\om,0}(H_{\om,\ep})>0$ for each $\ep>0$, then $\ta_{\om,0}$ is given by 
\begin{align*}
	\ta_{\om,0}:=
	1-\sum_{k=0}^{\infty}(\lm_{\sg^{-(k+1)}\om,0}^{k+1})^{-1}q_{\om,0}^{(k)}
	=
	1-\sum_{k=0}^{\infty}\hat q_{\om,0}^{(k)}.
\end{align*}
In light of Lemma~\ref{lem: checking P7 and P8}, we see that \eqref{C1}-\eqref{C8} imply \eqref{P1}-\eqref{P9}, and thus we have the following Theorem and first main result of this section.

\begin{theorem}\label{thm: dynamics perturb thm}
	Suppose that \eqref{C1}-\eqref{C7} hold for a random open system $(\mathlist{\bcomma}{\Om, m, \sg, \cJ_0, T, \cB, \cL_0, \nu_0, \phi_0, H_\ep})$. 
	For $m$-a.e. $\om\in\Om$ if there is some $\ep_0>0$ such that $\mu_{\om,0}(H_{\om,\ep})=0$ for $\ep\leq \ep_0$ then
	\begin{align*}
		\lm_{\om,0}=\lm_{\om,\ep},
	\end{align*}
	or if \eqref{C8} holds, then 
	\begin{align*}
		\lim_{\ep\to 0}
		\frac{\lm_{\om,0}-\lm_{\om,\ep}}{\lm_{\om,0}\mu_{\om,0}(H_{\om,\ep})}
		=\ta_{\om,0}.
	\end{align*}
	Furthermore, the map $\Om_+\ni\om\mapsto\ta_{\om,0}$ is measurable. 
\end{theorem}
\begin{proof}
	All statements follow directly from Theorem \ref{thm: GRPT}, except measurability of $\theta_{\omega,0}$, which follows from its construction as a limit of measurable objects.
\end{proof}
\begin{remark}
	As Theorem \ref{thm: GRPT} does not require any measurability, one could restate
	Theorem~\ref{thm: dynamics perturb thm} to hold for sequential open systems satisfying the sequential versions of hypotheses \eqref{C1}-\eqref{C8}.
\end{remark}
The following lemma will be useful for bounding $\ta_{\om,0}$.

\begin{lemma}
	\label{qbound2}
	Suppose $m(\Om\bs\Om_+)=0$. 
	Then for each $\ep>0$ and $m$-a.e. $\om\in\Om$ we have
	\begin{align*}
		\sum_{k=0}^\infty (\lm_{\sg^{-(k+1)}\om,0}^{k+1})^{-1}q_{\om,\ep}^{(k)}=\sum_{k=0}^\infty \hat q^{(k)}_{\om,\ep}= 1.	
	\end{align*}
\end{lemma}
\begin{proof}
	Recall from Section~\ref{sec: open systems} that the hole $H_\ep\sub\cJ_0$ is given by  
	$$
	H_\ep:=\bigcup_{\om\in \Om}\set{\om}\times H_{\om,\ep},
	$$
	and the skew map $T:\cJ_0\to\cJ_0$ is given by
	$$
	T(\om, x)=(\sg\om, T_{\om}(x)).
	$$
	Let $\tau_{H_\ep}(\om,x)$ denote the first return time of the point $(\om, x)\in H_\ep$ into $H_\ep.$  For each $k\geq 0$ let $B_k:=\set{(\om,x)\in H_\ep: \tau_{H_\ep}(\om,x)=k+1}$ be the set of points $(\om,x)$ that remain outside of the hole $H_\ep$ for exactly $k$ iterates, and for each $\om\in\Om$ we set $B_{k,\om}:=\pi_2(B_k\cap\set{\om}\times\cJ_{\om,0})$. Then $B_{k,\om}$ is precisely the set 
	$$
	B_{k,\om}=\set{x\in H_{\om,\ep}: T_\om^{k+1}(x)\in H_{\sg^{k+1}\om,\ep} \text{ and } T_\om^j(x)\not\in H_{\sg^j\om,\ep} \text{ for all } 1\leq j\leq k}.
	$$
	For each $k\geq 0$ we can disintegrate $\mu_0$ as
	\begin{align*}
	    \mu_0\left(B_k\right)
	    &=\int_\Om\mu_{\om,0}(B_{k,\om})\ dm(\om),
	\end{align*}
	where the last equality follows from the $\sg$-invariance $m$.
	In particular, for fixed $\om\in\Om_+$ and each $k\geq 0$ we have
	\begin{align}\label{qsum1}
		&\mu_{\sg^{-(k+1)}\om,0}(B_{k,\sg^{-(k+1)}\om})
		=
		\mu_{\sg^{-(k+1)}\om,0}\left(
		T_{\sg^{-(k+1)}\om}^{-(k+1)}(H_{\om,\ep})
		\cap\left(\bigcap_{j=1}^k T_{\sg^{-(k+1)}\om}^{-(k+1)+j} (H_{\sg^{-j}\om,\ep}^c)\right)
		\cap H_{\sg^{-(k+1)}\om,\ep}
		\right).
	\end{align}
	Recall from \eqref{def of hat q} that $\hat q_{\om,\ep}^{(k)}$ is given by
	\begin{align*}
		\hat q_{\om,\ep}^{(k)}:=
		\frac{\mu_{\sg^{-(k+1)}\om,0}\left(
			T_{\sg^{-(k+1)}\om}^{-(k+1)}(H_{\om,\ep})
			\cap\left(\bigcap_{j=1}^k T_{\sg^{-(k+1)}\om}^{-(k+1)+j} (H_{\sg^{-j}\om,\ep}^c)\right)
			\cap H_{\sg^{-(k+1)}\om,\ep}
			\right)
		}
		{\mu_{\sg^{-(k+1)}\om,0}\left(
			T_{\sg^{-(k+1)}\om}^{-(k+1)}(H_{\om,\ep})\right)},
	\end{align*}
	which is well defined since $\om\in\Om_+$ by assumption. Note that since $\sum_{k=0}^\infty\mu_{\sg^{-(k+1)}\om,0}(B_{k,\sg^{-(k+1)}\om})\leq \mu_{\om,0}(H_{\om,\ep})$ we must have that $\sum_{k=0}^{\infty}\hat{q}^{(k)}_{\om,\ep}\leq1$. 
	Now, as the measure $\mu_0$ is $T$-invariant, the right-hand side of \eqref{qsum1} is equal to 
	$\hat{q}^{(k)}_{\om,\ep}\mu_{\om,0}(H_{\om,\ep})$, and therefore 
	\begin{align}\label{B_k measure}
	    \int_\Om \hat{q}^{(k)}_{\om,\ep}\ \mu_{\om,0}(H_{\om,\ep})\ dm(\om)
	    =\int_\Om \mu_{\om,0}(B_{k,\om})\ dm(\om)
	    =\mu_0(B_k).
	\end{align}
	By the Poincar\'e recurrence theorem we have
	$$
	\sum_{k=0}^{\infty}\mu_0(B_k)=\mu_0(H_\ep)=\int_\Om\mu_{\om,0}(H_{\om,\ep})\, dm(\om).
	$$
	By interchanging the sum with the integral above (possible by Tonelli's Theorem) and using \eqref{B_k measure}, we have
	$$
	\int_\Om \left(\sum_{k=0}^{\infty}\hat{q}^{(k)}_{\om,\ep}\ \mu_{\om,0}(H_{\om,\ep})\right)\ dm(\om)
	=
	\int_\Om \mu_{\om,0}(H_{\om,\ep})\ dm(\om),
	$$
	which implies
	$$
	\int_\Om \lt(\mu_{\om,0}(H_{\om,\ep})\left(\sum_{k=0}^{\infty}\hat{q}^{(k)}_{\om,\ep}-1\right)\rt) dm(\om)=0.
	$$
	Since we already have that $\sum_{k=0}^{\infty}\hat{q}^{(k)}_{\om,\ep}\leq1$ for $m$-a.e. $\om\in\Om$, we must in fact have $\sum_{k=0}^{\infty}\hat{q}^{(k)}_{\om,\ep}=1$, which completes the proof.
\end{proof}
By definition, we have that $\hat q_{\om,\ep}^{(k)}\in[0,1]$ and thus \eqref{C8} implies that $\hat q_{\om,0}^{(k)}\in[0,1]$ for each $k\geq 0$ as well. In light of the previous lemma, dominated convergence implies that  
\begin{align}\label{theta in [0,1]}
	\ta_{\om,0}\in[0,1].
\end{align}

\subsection{Escape Rate Asymptotics}


If we have some additional $\om$-control on the size of the holes we obtain the following corollary of Theorem~\ref{thm: dynamics perturb thm}, which provides a formula for the escape rate asymptotics for small random holes.
The scaling of the holes (\ref{EVT style cond}) takes a similar form to 
the scaling we will shortly use in the next section for our quenched extreme value theory.

\begin{corollary}\label{esc rat cor}
	Suppose that \eqref{C1}--\eqref{C8} hold for a random open system $(\mathlist{\bcomma}{\Om, m, \sg, \cJ_0, T, \cB, \cL_0, \nu_0, \phi_0, H_\ep})$ and there exists $A\in L^1(m)$ such that for $m$-a.e. $\om\in\Om$ and all $\ep>0$ sufficiently small we have 
	\begin{align}\label{DCT cond 1*}
		\frac{\log\lm_{\om,0}-\log\lm_{\om,\ep}}{\mu_{\om,0}(H_{\om,\ep})}\leq A(\om).
	\end{align}
	Further suppose that there is some $\kappa(\ep)$ with $\kp(\ep)\to\infty$ as $\ep\to 0$, $t\in L^\infty(m)$ with $t_\om>0$, and $\xi_{\om,\ep}\in L^\infty(m)$ with $\xi_{\om,\ep}\to 0$ as $\ep\to 0$ for $m$-a.e. $\om\in\Om$ such that
	\begin{align}\label{EVT style cond}
		\mu_{\om,0}(H_{\om,\ep})=\frac{t_\om+\xi_{\om,\ep}}{\kp(\ep)}
	\end{align}
	and $|\xi_{\om,\ep}|<Ct_\om$ for all $\ep>0$ sufficiently small for some $C>0$.
	Then for $m$-a.e. $\om\in\Om$ we have
	\begin{align*}
		\lim_{\ep\to 0}\frac{R_\ep(\mu_{\om,0})}{\mu_{\om,0}(H_{\om,\ep})}
		=
		\frac{\int_\Om \ta_{\om,0}t_\om\, dm(\om)}{t_\om}.
	\end{align*}
	In particular, if $t_\om$ is constant $m$-a.e. then 
	\begin{align*}
		\lim_{\ep\to 0}\frac{R_\ep(\mu_{\om,0})}{\mu_{\om,0}(H_{\om,\ep})}
		=
		\int_\Om \ta_{\om,0}\, dm(\om).
	\end{align*}
\end{corollary}
\begin{proof}
    First, we note that \eqref{EVT style cond} implies that $\mu_{\om,0}(H_{\om,\ep})>0$ for $m$-a.e. $\om\in\Om$ and all $\ep>0$ sufficiently small, i.e. $m(\Om\bs\Om_+)=0$. 
	Using Theorem \ref{thm: dynamics perturb thm} we have that
	\begin{align}\label{fiber esc rat deriv}
		\lim_{\ep\to 0}\frac{\log\lm_{\om,0}-\log\lm_{\om,\ep}}{\mu_{\om,0}(H_{\om,\ep})}=\ta_{\om,0}
	\end{align}
	since
	\begin{align*}
		\frac{\log\lm_{\om,0}-\log\lm_{\om,\ep}}{\mu_{\om,0}(H_{\om,\ep})}
		&=
		\frac{\log\lm_{\om,0}-\log\lm_{\om,\ep}}{\lm_{\om,0}-\lm_{\om,\ep}}
		\cdot
		\frac{\lm_{\om,0}-\lm_{\om,\ep}}{\mu_{\om,0}(H_{\om,\ep})}
		\lra
		\frac{1}{\lm_{\om,0}}\cdot\lm_{\om,0}\ta_{\om,0}=\ta_{\om,0}
	\end{align*}
	as $\ep\to 0$. 
	In light of \eqref{cons of NulamN} and (\ref{eq: lem escape rate}), we use \eqref{EVT style cond} to get that 
	\begin{align*}
		\lim_{\ep\to 0}\frac{R_\ep(\mu_{\om,0})}{\mu_{\om,0}(H_{\om,\ep})}
		&=
		\lim_{\ep\to 0}\lim_{N\to\infty}\frac{1}{N}\frac{\log\lm_{\om,0}^N-\log\lm_{\om,\ep}^N}{\mu_{\om,0}(H_{\om,\ep})}
		\\
		&=
		\lim_{\ep\to 0}\lim_{N\to\infty}\frac{1}{N}\sum_{j=0}^{N-1}\frac{\log\lm_{\sg^j\om,0}-\log\lm_{\sg^j\om,\ep}}{\mu_{\sg^j\om,0}(H_{\sg^j\om,\ep})}\cdot\frac{\mu_{\sg^j\om,0}(H_{\sg^j\om,\ep})}{\mu_{\om,0}(H_{\om,\ep})}
		\\
		&=
		\lim_{\ep\to 0}\lim_{N\to\infty}\frac{1}{N}\sum_{j=0}^{N-1}\frac{\log\lm_{\sg^j\om,0}-\log\lm_{\sg^j\om,\ep}}{\mu_{\sg^j\om,0}(H_{\sg^j\om,\ep})}\cdot\frac{t_{\sg^j\om}+\xi_{\sg^j\om,\ep}}{t_{\om}+\xi_{\om,\ep}}
		\\
		&=
		\lim_{\ep\to 0}\frac{1}{t_\om+\xi_{\om,\ep}}\lim_{N\to\infty}\frac{1}{N}\sum_{j=0}^{N-1}\frac{\log\lm_{\sg^j\om,0}-\log\lm_{\sg^j\om,\ep}}{\mu_{\sg^j\om,0}(H_{\sg^j\om,\ep})}\cdot (t_{\sg^j\om}+\xi_{\sg^j\om,\ep})
		\\
		&=
		\lim_{\ep\to 0}\frac{1}{t_\om+\xi_{\om,\ep}}\int_\Om\frac{\log\lm_{\om,0}-\log\lm_{\om,\ep}}{\mu_{\om,0}(H_{\om,\ep})}\cdot (t_{\om}+\xi_{\om,\ep})\, dm(\om),
	\end{align*}
	where the last line follows from Birkhoff (which is applicable thanks to \eqref{DCT cond 1*} and the fact that $t_\om,\xi_{\om,\ep}\in L^\infty(m)$).
	As $\xi_{\om,\ep}\to 0$ by assumption, \eqref{DCT cond 1*} allows us to apply Dominated Convergence which, in view of \eqref{fiber esc rat deriv}, implies
	\begin{align*}
		\lim_{\ep\to 0}\frac{R_\ep(\mu_{\om,0})}{\mu_{\om,0}(H_{\om,\ep})}
		=
		\frac{1}{t_\om}\lim_{\ep\to 0}\int_\Om \frac{\log\lm_{\om,0}-\log\lm_{\om,\ep}}{\mu_{\om,0}(H_{\om,\ep})}\cdot (t_{\om}+\xi_{\om,\ep})\, dm(\om)
		=
		\frac{\int_\Om\ta_{\om,0}t_\om\, dm(\om)}{t_\om},
	\end{align*}
	completing the proof.
\end{proof}
In the next section we will present easily checkable assumptions which will imply the hypotheses of Corollary~\ref{esc rat cor}.

\section{Quenched extreme value law}\label{EEVV}

\subsection{Gumbel's law in the quenched regime}
Suppose that $h_\om:\cJ_{\om,0}\to\RR$ is a continuous function for each $\om\in\Om$.
In our quenched random extreme value theory, $h_\omega$ is an observation function that is allowed to depend on $\omega$.
For each $\om\in\Om$ let $\ol{z}_\om$ be the essential supremum of $h_\om$ with respect to $\nu_{\om,0}$, that is
\begin{align*}
	\ol{z}_\om:=\sup\set{z\in\RR\; :\nu_{\om,0}(\set{x\in\cJ_{\om,0}:h_\om(x)\geq z})>0}.
\end{align*}
Similarly we define $\Ul z_\om$ to be the essential infimum of $h_\om$ with respect to $\nu_{\om,0}$. 
Suppose $\Ul{z}_\om<\ol{z}_\om$ for $m$-a.e. $\om\in\Om$, 
and for each $z\in[\Ul{z}_\om,\ol{z}_\om]$ we define the set
\begin{align*}
	V_{\om,z_\omega}:=\set{x\in\cJ_{\om,0}:h_\om(x)-z_\omega>0},
\end{align*}
which represents points $x$ in our phase space where the observation $h_\omega$ exceeds a random threshold $z_\omega$ at base configuration $\omega$.
Our theory allows for random thresholds $z_\omega$ so that we may consider both ``anomalous'' and ``absolute'' exceedances.
For example in a real-world application, $h_\omega(x)$ may represent the surface ocean temperature at a spatial location $x$ for an ocean system configuration $\omega$.
Random temperature exceedances above $z_\omega$ allow one to describe extreme value statistics for temperature anomalies, e.g.\ those above a climatological seasonal mean (on average the surface ocean is warmer in summer and colder in winter). 
On the other hand, non-random absolute temperature exceedances above $z$ are more relevant for marine life.
We suppose that
\begin{align*}
	\nu_{\om,0}(V_{\om,\ol{z}_\om})=0.
\end{align*}

As is standard in extreme value theory, to develop an exponential law we will consider an increasing sequence of thresholds $z_{\omega,0}<z_{\omega,1}<\cdots$.
For each $N\geq 0$ we take  $z_{\om,N}\in[\Ul{z}_\om,\ol{z}_\om]$ and the choice will be made explicit in a moment. 
For each $k,N\in\NN$ define $G_{\om,N}^{(k)}:\cJ_{\om,0}\to\RR$ by
\begin{align*}
	G_{\om,N}^{(k)}(x):=h_{\sg^N\om}(T_\om^N(x))-z_{\sg^N\om,k}.
\end{align*}
If $G_{\om,N}^{(k)}>0$ then $h_{\sg^N\om}(T_\om^N(x))>z_{\sg^N\om,k}$.
Our extreme value law concerns the large $N$ limit of the likelihood of continued threshold non-exceedances:
\begin{equation}
	\label{evtexpression}
	\nu_{\om,0}\left(\set{x\in\cJ_{\om,0}:\max\left(G_{\om,0}^{(N)}(x), \dots,G_{\om,N-1}^{(N)}(x)\right)\leq 0}\right).
\end{equation}
We may easily transform (\ref{evtexpression}) into the language of random open systems:
one immediately has that $G_{\om,N}^{(k)}>0$ is equivalent to $T_\om^N(x)\in V_{\sg^N\om,z_{\sg^N\om,k}}$.
Therefore, for each $N\geq 0$ we have
\begin{equation}
	\label{evtexpression2}
	(\ref{evtexpression})=
	\nu_{\om,0}\left(\set{x\in\cJ_{\om,0}:T_\om^j(x)\notin V_{\sg^j\om,z_{\sg^j\om,N}} \text{ for }j=0,\dots, N-1}\right)\\
	=\nu_{\om,0}\left(X_{\om,N-1,\ep_N}\right),
\end{equation}
where to obtain the second equality we 
identify the sets $V_{\om,z_{\om,N}}$ with holes $H_{\om,\ep_N}\subset \cJ_{\omega,0}$ for each $N\in\NN$
\footnote{
	In Section~\ref{EEVV} we consider a decreasing sequence of holes, whereas in previous sections we considered a decreasing family of holes $H_{\om,\ep}$ parameterised by $\ep>0$. For the sake of notational continuity, in this section, and in the sequel, we denote a decreasing sequence of holes by $H_{\om,\ep_N}$. 
	Note that $\ep_N$ here is just a parameter and should not be thought of as a real number, but instead as an index and that $H_{\om,\ep_N}$ serves as an alternative notation to $H_{\om,N}$. Furthermore, note that the measure of the hole depends on the fiber $\om$ with $\mu_{\om,0}(H_{\om,\ep_N})\to 0$ as $N\to\infty$. }.
%
Now using (\ref{NulamN}) and (\ref{evtexpression2}) we may convert  (\ref{evtexpression}) into the spectral expression:
\begin{equation}
	\label{evtexpression3}
	(\ref{evtexpression})=\frac{\lm_{\om,\ep_N}^N}{\lm_{\om,0}^N}
	\lt(\nu_{\om,\ep_N}(\ind)+\nu_{\sg^N\om,0}\lt(Q_{\om,\ep_N}^N( \ind)\rt)\rt).
\end{equation}
Similarly, using \eqref{eq: mu0 of n survivor}, we can write 
\begin{align}
	&\mu_{\om,0}\left(\set{x\in\cJ_{\om,0}:T_\om^j(x)\notin V_{\sg^j\om,z_{\sg^j\om,N}} \text{ for }j=0,\dots, N-1}\right)
	=
	\mu_{\om,0}\left(X_{\om,N-1,\ep_N}\right)
	\nonumber\\
	&\qquad=\frac{\lm_{\om,\ep_N}^N}{\lm_{\om,0}^N}
	\lt(\nu_{\om,\ep_N}(\phi_{\om,0})+\nu_{\sg^N\om,0}\lt(Q_{\om,\ep_N}^N( \phi_{\om,0})\rt)\rt).
	\label{evtexpression2mu}
\end{align}
Before we state our main result concerning the $N\to\infty$ limit, in addition to \eqref{C2}, \eqref{C3}, and \eqref{C8}, 
we make the following uniform adjustments to some of the assumptions in Section \ref{sec:goodrandom} as well as an assumption 
on the choice of the sequence of thresholds, condition \eqref{xibound}. 
At the end of this section, we will compare it with the H\"usler condition, which is the usual prescription for non stationary processes, as we anticipated in the Introduction.
\begin{enumerate}[align=left,leftmargin=*,labelsep=\parindent]
	\item[\mylabel{S}{xibound}]
For any fixed random scaling function $t\in L^\infty(m)$ with $t>0$, we may find sequences of functions $z_{N},\xi_N\in L^\infty(m)$  and a constant $W<\infty$ satisfying 
	$$\mu_{\omega,0}(\{h_\omega(x)-z_{\omega,N}>0\})=(t_\omega+\xi_{\omega,N})/N, \mbox{ for a.e.\ $\omega$ and each $N\ge 1$}
	$$
where:\\
(i)
 $\lim_{N\to\infty} \xi_{\omega,N}=0$ for a.e.\ $\omega$ and \\
 (ii) $|\xi_{\omega,N}|\le W$ for a.e.\ $\omega$ and all $N\ge 1$.
\end{enumerate}
\begin{enumerate}[align=left,leftmargin=*,labelsep=\parindent]
	\item[\mylabel{C1'}{C1'}]There exists $C_1\geq 1$ such that for $m$-e.a. $\om\in\Om$ we have 
	\begin{align*}
		C_1^{-1}\leq \cL_{\om,0}\ind\leq C_1.
	\end{align*}
	
	\item[\mylabel{C4'}{C4'}] 
	For each $f\in\cB$ and each $N\in\NN$ there exists $C_f>0$ and $\al(N)>0$ (independent of $\om$) with $\al:=\sum_{N=1}^\infty\al(N)<\infty$ such that for $m$-a.e. $\om\in\Om$, all $N\in\NN$
	\begin{align*}
		\sup_{\ep\geq 0}\norm{Q_{\om,\ep}^N f_{\om}}_{\infty,\sg^N\om}\leq C_f\al(N)\norm{f_{\om}}_{\cB_{\om}}.
	\end{align*}
	\item[\mylabel{C5'}{C5'}] There exists $C_2\geq 1$ such that
	\begin{align*}
		\sup_{\ep\geq 0}\norm{\phi_{\om,\ep}}_{\infty,\om}\leq C_2
		\quad\text{ and }\quad
		\norm{\phi_{\om,0}}_{\cB_\om}\leq C_2
	\end{align*}
	for $m$-a.e. $\om\in\Om$.
	
	\item[\mylabel{C7'}{C7'}] 
	There exists $C_3\geq 1$ such that for all $\ep>0$ sufficiently small we have
	\begin{align*}
		\essinf_\om\inf\phi_{\om,0}\geq C_3^{-1} >0
		\qquad\text{ and }\qquad 
		\essinf_\om\inf\phi_{\om,\ep}\geq 0.
	\end{align*}

\end{enumerate}
\begin{remark}\label{zero hole}
    Note that 
    \eqref{xibound} implies that $\mu_{\om,0}(H_{\om,\ep_N})>0$ for each $N\in\NN$ since $t_\om>0$, and in particular we have $m(\Om\bs\Om_+)=0$. 
\end{remark}
\begin{remark}
	\label{rem61}
	Note that since $\lm_{\om,0}=\nu_{\sg\om,0}(\cL_{\om,0}\ind)$, \eqref{C1'} implies that 
	\begin{align*}
		C_1^{-1}\leq \lm_{\om,0}\leq C_1
	\end{align*}
	for $m$-a.e. $\om$.
\end{remark}
\begin{remark}\label{rem checking esc cor cond}
	Note that conditions \eqref{xibound}, \eqref{C5'}, and \eqref{C7'} together imply \eqref{C6}, thus Theorem~\ref{thm: dynamics perturb thm} applies. 
	Furthermore, these same conditions along with Remark~\ref{rem61} imply that there exists $A\geq 1$ such that 
	\begin{align*}
		\frac{\log\lm_{\om,0}-\log\lm_{\om,\ep}}{\mu_{\om,0}(H_{\om,\ep})}\leq A
	\end{align*}
	for $m$-a.e. $\om\in\Om$ and all $\ep>0$ sufficiently small, and thus \eqref{DCT cond 1*} and \eqref{EVT style cond} hold, meaning that  Corollary~\ref{esc rat cor} applies as well. 
\end{remark}

The following lemma shows that $\nu_{\om,\ep_N}(\phi_{\om,0})$ converges to $1$ uniformly in $\om$ under our assumptions \eqref{C1'}, \eqref{C4'}, \eqref{C5'}, \eqref{C7'}, and \eqref{xibound}.
\begin{lemma}\label{lemC9'}
	For each $N\in\NN$ there exists $C_{\ep_N}\geq 1$ with $C_{\ep_N}\to 1$ as $N\to\infty$ such that
	\begin{align}\label{C9'ineq}
		C_{\ep_N}^{-1}\leq \nu_{\om,\ep_N}(\ind),\, \nu_{\om, \ep_N}(\phi_{\om,0}) \leq C_{\ep_N}
	\end{align}
	for $m$-a.e.\ $\omega\in\Omega$.
\end{lemma}
\begin{proof}
	Note that 
	\begin{equation}
		\label{opdiff}
		\nu_{\sg\om,0}\lt((\cL_{\om,0}-\cL_{\om,\ep_N})(\~\cL_{\sg^{-k}\om,\ep_N}^k \phi_{\sg^{-k}\om,0})	\rt)
		=
		\lm_{\om,0}\nu_{\om,0}\lt(\ind_{H_{\om,\ep_N}}\~\cL_{\sg^{-k}\om,\ep_N}^k\phi_{\sg^{-k}\om,0}\rt).
	\end{equation}
	Further, by (\ref{C7'}) we have
	\begin{equation}
		\label{EVTeta2}
		\nu_{\om,0}(H_{\om,\ep_N})
		\leq
		C_3
		\mu_{\om,0}(H_{\om,\ep_N})
		=
		\frac{C_3
			(t_\om+\xi_{\om,N})}{N}
	\end{equation}
	for $m$-a.e. $\om\in\Om$.
	Following the same derivation of \eqref{diff eigenvalues identity unif}, and using \eqref{C2} gives that 
	\begin{align}
		\lm_{\om,0}-\lm_{\om,\ep_N}
		&=\nu_{\sg\om,0}\left((\cL_{\om,0}-\cL_{\om,\ep_N})(\phi_{\om,\ep_N})\right)
		=\lm_{\om,0}\nu_{\om,0}(\ind_{H_{\om,\ep_N}}\phi_{\om,\ep_N}).
		\label{diff eigenvalues identity unif unif}
	\end{align}
	Thus it follows from Remark~\ref{rem lm_0>lm_ep}, \eqref{diff eigenvalues identity unif unif}, \eqref{C5'}, and \eqref{EVTeta2} that
	\begin{equation}
		\label{EVTlamdiff2}
		0\leq \lm_{\om,0}-\lm_{\om,\ep_N}
		\leq \frac{C_2C_3\lm_{\om,0}(t_\om+\xi_{\om,N})}{N}.
	\end{equation}    
	Using Remark~\ref{rem61} and \eqref{EVTlamdiff2}, for all $N$ sufficiently large, we see that 
	\begin{align}\label{LmUniErr}
		1\leq\frac{\lm_{\om,0}}{\lm_{\om,\ep_N}}
		=
		\frac{\lm_{\om,0}}{\lm_{\om,0}-\lm_{\om,0}\nu_{\om,0}(\ind_{H_{\om,\ep_N}}\phi_{\om,\ep_N})}
		\leq 
		\frac{1}{1-C_2\nu_{\om,0}(H_{\om,\ep_N})}
		\leq
		\frac{1}{1-E_N},
	\end{align}
	where
	\begin{align*}
		E_N:=\frac{C_2C_3(|t|_\infty +W)}{N}\to 0
	\end{align*}
	as $N\to\infty$.

	\begin{claim}\label{claim Lemma 6.1.a}
		For every $n$, $\ep_N$, and $m$-a.e. $\om\in\Om$ we have
		\begin{align*}
			\absval{1-\nu_{\om,\ep_N}(\phi_{\om,0})}
			\leq
			n\left(\frac{1}{1-E_N}\right)^n\cdot
			\left(C_1E_N+\frac{C_1C_2E_N}{1-E_N}\right)
			+
			C_2C_{\phi_0}\al(n).
		\end{align*}
	\end{claim}
	\begin{subproof}
		Using \eqref{C2} and a telescoping argument, we can write
		\begin{align}
			\absval{1-\nu_{\om,\ep_N}(\phi_{\om,0})}
			&=\absval{\nu_{\sg^n\om,0}(\phi_{\sg^n\om,0})-\nu_{\om,\ep_N}(\phi_{\om,0})}
			\nonumber\\
			&
			=\absval{\nu_{\sg^n\om,0}(\phi_{\sg^n\om,0})-\nu_{\sg^n\om,0}\big(\nu_{\om,\ep_N}(\phi_{\om,0})\cdot \phi_{\sg^n\om,\ep_N}\big)}
			\nonumber\\	
			&
			=\absval{\nu_{\sg^n\om,0}\left(\phi_{\sg^n\om,0}-\~\cL_{\om,\ep_N}^n(\phi_{\om,0})+Q_{\om,\ep_N}^n(\phi_{\om,0})\right)}
			\nonumber\\
			&
			\leq\sum_{k=0}^{n-1}\absval{\nu_{\sg^n\om,0}\left(\~\cL_{\sg^{n-k}\om,\ep_N}^k(\phi_{\sg^{n-k}\om,0})-\~\cL_{\sg^{n-(k+1)}\om,\ep_N}^{k+1}(\phi_{\sg^{n-(k+1)}\om,0})\right)}
			\nonumber\\ 	
			&\qquad
			+\absval{\nu_{\sg^n\om,0}\left(Q_{\om,\ep_N}^n(\phi_{\om,0})\right)}
			\nonumber\\
			& 
			=\sum_{k=0}^{n-1}\absval{\nu_{\sg^n\om,0}\left(\left(\~\cL_{\sg^{n-k}\om,\ep_N}^{k}\right)\left(\~\cL_{\sg^{n-(k+1)}\om,0}-\~\cL_{\sg^{n-(k+1)}\om,\ep_N}\right)(\phi_{\sg^{n-(k+1)}\om,0})\right)}
			\label{lemma a sum}\\
			&\qquad
			+\absval{\nu_{\sg^n\om,0}\left(Q_{\om,\ep_N}^n(\phi_{\om,0})\right)}.
			\label{lemma a error term}
		\end{align}
		First we note that we can estimate \eqref{lemma a error term} as
		\begin{align}
			\absval{\nu_{\sg^n\om,0}\left(Q_{\om,\ep_N}^n(\phi_{\om,0})\right)}\leq \norm{Q_{\om,\ep_N}^n(\phi_{\om,0})}_{\infty,\sg^n\om}
			\leq 
			C_2C_{\phi_0}\al(n)
			.
			\label{unifeq0}
		\end{align}	
		Now recall that
		\begin{align*}
			\hat{X}_{\sg^{-k}\om,k,\ep_N}:=\ind_{\bigcap_{j=0}^k T_{\sg^{-k}\om}^{-j}(\cJ_{\sg^{j-k}\om,\ep_N})}.
		\end{align*}
		Using \eqref{C2}, we can write \eqref{lemma a sum} as 
		\begin{align}
			&\sum_{k=0}^{n-1}\absval{\nu_{\sg^n\om,0}\left(\left(\~\cL_{\sg^{n-k}\om,\ep_N}^{k}\right)\left(\~\cL_{\sg^{n-(k+1)}\om,0}-\~\cL_{\sg^{n-(k+1)}\om,\ep_N}\right)(\phi_{\sg^{n-(k+1)}\om,0})\right)}
			\nonumber\\
			&\quad
			=\sum_{k=0}^{n-1}\absval{\nu_{\sg^n\om,0}\left(\left(\lm_{\sg^{n-k}\om,\ep_N}^k\right)^{-1}
				\cL_{\sg^{n-k}\om,0}^k\left(\hat{X}_{\sg^{n-k}\om,k,\ep_N}\cdot\left(\~\cL_{\sg^{n-(k+1)}\om,0}-\~\cL_{\sg^{n-(k+1)}\om,\ep_N}\right)(\phi_{\sg^{n-(k+1)}\om,0})\right)\right)}
			\nonumber\\	
			&\quad
			=\sum_{k=0}^{n-1}\frac{\lm_{\sg^{n-k}\om,0}^k}{\lm_{\sg^{n-k}\om,\ep_N}^k}\absval{\nu_{\sg^{n-k}\om,0}\left(\hat{X}_{\sg^{n-k}\om,k,\ep_N}\cdot\left(\~\cL_{\sg^{n-(k+1)}\om,0}-\~\cL_{\sg^{n-(k+1)}\om,\ep_N}\right)(\phi_{\sg^{n-(k+1)}\om,0})\right)}.
			\label{unifeq1.0}
		\end{align}	
		Now, since $\sfrac{\lm_{\om,0}}{\lm_{\om,\ep_N}}\geq1$ (by Remark~\ref{rem lm_0>lm_ep}) for $m$ a.e. $\om$ and $\ep\geq 0$, and thus $\sfrac{\lm_{\sg^{n-k}\om,0}^k}{\lm_{\sg^{n-k}\om,\ep_N}^k}\geq1$ for each $k\geq 1$, we can write
		\begin{align}
			\eqref{unifeq1.0}&\leq
			\frac{\lm_{\om,0}^n}{\lm_{\om,\ep_N}^n}\sum_{k=0}^{n-1}\absval{\nu_{\sg^{n-k}\om,0}\left(\hat{X}_{\sg^{n-k}\om,k,\ep_N}\cdot\left(\~\cL_{\sg^{n-(k+1)}\om,0}-\~\cL_{\sg^{n-(k+1)}\om,\ep_N}\right)(\phi_{\sg^{n-(k+1)}\om,0})\right)}
			\nonumber\\
			&
			\leq 
			\left(\frac{1}{1-E_N}\right)^n\cdot
			\sum_{k=0}^{n-1}\absval{\nu_{\sg^{n-k}\om,0}\left(\hat{X}_{\sg^{n-k}\om,k,\ep_N}\cdot\left(\~\cL_{\sg^{n-(k+1)}\om,0}-\~\cL_{\sg^{n-(k+1)}\om,\ep_N}\right)(\phi_{\sg^{n-(k+1)}\om,0})\right)},
			\label{unifeq1} 		
		\end{align}
		Using 
		\begin{align*}
			\~\cL_{\om,0}-\~\cL_{\om,\ep}&=\~\cL_{\om,0}-\lm_{\om,0}^{-1}\cL_{\om,\ep}+\lm_{\om,0}^{-1}\cL_{\om,\ep}-\~\cL_{\om,\ep}\nonumber\\
			&=\lm_{\om,0}^{-1}\left(\cL_{\om,0}-\cL_{\om,\ep}\right)+(\lm_{\om,0}^{-1}-\lm_{\om,\ep}^{-1})\cL_{\om,\ep}\nonumber\\
			&=\lm_{\om,0}^{-1}\left(\left(\cL_{\om,0}-\cL_{\om,\ep}\right)+\lm_{\om,0}\lm_{\om,\ep}(\lm_{\om,0}^{-1}-\lm_{\om,\ep}^{-1})\~\cL_{\om,\ep}\right)\nonumber\\
			&=\lm_{\om,0}^{-1}\left(\left(\cL_{\om,0}-\cL_{\om,\ep}\right)+(\lm_{\om,\ep}-\lm_{\om,0})\~\cL_{\om,\ep}\right),
		\end{align*}
		the fact that
		\begin{align*}
			\cL_{\om,\ep_N}\left((f_{\sg\om}\circ T_\om)\cdot h_{\om}\right)=f_{\sg\om}\cdot\cL_{\om,\ep_N}(h_{\om}),
		\end{align*}
		for all $\ep_N\geq 0$ and $\om\in\Om$, and Remark~\ref{rem61}, we may estimate the sum in \eqref{unifeq1} by 
		\begin{align}
			&\sum_{k=0}^{n-1}\absval{\nu_{\sg^{n-k}\om,0}\left(\hat{X}_{\sg^{n-k}\om,k,\ep_N}\cdot\left(\~\cL_{\sg^{n-(k+1)}\om,0}-\~\cL_{\sg^{n-(k+1)}\om,\ep_N}\right)(\phi_{\sg^{n-(k+1)}\om,0})\right)}
			\nonumber\\
			&\quad
			\leq\sum_{k=0}^{n-1}\lm_{\sg^{n-(k+1)}\om,0}^{-1}\absval{\nu_{\sg^{n-k}\om,0}\left(\hat{X}_{\sg^{n-k}\om,k,\ep_N}\cdot\left(\cL_{\sg^{n-(k+1)}\om,0}-\cL_{\sg^{n-(k+1)}\om,\ep_N}\right)(\phi_{\sg^{n-(k+1)}\om,0})\right)}
			\nonumber\\
			&\qquad
			+\sum_{k=0}^{n-1}\lm_{\sg^{n-(k+1)}\om,0}^{-1}\absval{(\lm_{\sg^{n-(k+1)}\om,\ep_N}-\lm_{\sg^{n-(k+1)}\om,0})\cdot \nu_{\sg^{n-k}\om,0}\left(\hat{X}_{\sg^{n-k}\om,k,\ep_N}\cdot\~\cL_{\sg^{n-(k+1)}\om,\ep_N}(\phi_{\sg^{n-(k+1)}\om,0})\right)} 
			\nonumber\\
			&\quad
			= \sum_{k=0}^{n-1}\lm_{\sg^{n-(k+1)}\om,0}^{-1}
			\absval{\nu_{\sg^{n-k}\om,0}\left(\left(\cL_{\sg^{n-(k+1)}\om,0}-\cL_{\sg^{n-(k+1)}\om,\ep_N}\right)\left(\left(\hat{X}_{\sg^{n-k}\om,k,\ep_N}\circ T_{\sg^{n-(k+1)}\om}\right)\cdot\phi_{\sg^{n-(k+1)}\om,0}\right)\right)}
			\label{unifeq2.1}\\
			&\qquad
			+\sum_{k=0}^{n-1}\bigg(\lm_{\sg^{n-(k+1)}\om,0}^{-1}\absval{1-\frac{\lm_{\sg^{n-(k+1)}\om,0}}{\lm_{\sg^{n-(k+1)}\om,\ep_N}}}
			\nonumber\\
			&\quad\qquad
			\cdot\absval{ \nu_{\sg^{n-k}\om,0}\left(\cL_{\sg^{n-(k+1)}\om,\ep_N}\left(\left(\hat{X}_{\sg^{n-k}\om,k,\ep_N}\circ T_{\sg^{n-(k+1)}\om}\right)\cdot\phi_{\sg^{n-(k+1)}\om,0}\right)\right)}\bigg).
			\label{unifeq2.2}
		\end{align}
		Using \eqref{diff eigenvalues identity unif unif}, \eqref{EVTlamdiff2}, and Remark~\ref{rem61} we can estimate 
		\eqref{unifeq2.1} to get 
		\begin{align}
			&\sum_{k=0}^{n-1}\lm_{\sg^{n-(k+1)}\om,0}^{-1}\absval{\nu_{\sg^{n-k}\om,0}\left(\left(\cL_{\sg^{n-(k+1)}\om,0}-\cL_{\sg^{n-(k+1)}\om,\ep_N}\right)\left(\left(\hat{X}_{\sg^{n-k}\om,k,\ep_N}\circ T_{\sg^{n-(k+1)}\om}\right)\cdot\phi_{\sg^{n-(k+1)}\om,0}\right)\right)}
			\nonumber\\
			&\qquad
			=\sum_{k=0}^{n-1}\absval{\nu_{\sg^{n-(k+1)}\om,0}\left(\ind_{H_{\sg^{n-(k+1)}\om,\ep_N}}\cdot\left(\left(\hat{X}_{\sg^{n-k}\om,k,\ep_N}\circ T_{\sg^{n-(k+1)}\om}\right)\cdot\phi_{\sg^{n-(k+1)}\om,0}\right)\right)}
			\nonumber\\
			&\qquad
			\leq \sum_{k=0}^{n-1}\nu_{\sg^{n-(k+1)}\om,0}\left(\ind_{H_{\sg^{n-(k+1)}\om,\ep_N}}\right)\norm{\phi_{\sg^{n-(k+1)}\om,0}}_{\infty,\sg^{n-(k+1)}\om}
			\leq
			\frac{C_2C_3(t_\om+\xi_{\om,N})\cdot n}{N}\leq nE_N.
			\label{unifeq3}
		\end{align}
		Since $\cL_{\om,\ep}(f)=\cL_{\om,0}(\hat X_{\om,0,\ep}f)$ we can rewrite the second product in the sum in \eqref{unifeq2.2} so that we have 
		\begin{align}
			&\nu_{\sg^{n-k}\om,0}\left(\cL_{\sg^{n-(k+1)}\om,\ep_N}\left(\left(\hat{X}_{\sg^{n-k}\om,k,\ep_N}\circ T_{\sg^{n-(k+1)}\om}\right)\cdot\phi_{\sg^{n-(k+1)}\om,0}\right)\right)
			\nonumber\\
			&\qquad=
			\nu_{\sg^{n-k}\om,0}\left(\cL_{\sg^{n-(k+1)}\om,0}\left(\hat{X}_{\sg^{n-(k+1)}\om,0,\ep_N}\left(\hat{X}_{\sg^{n-k}\om,k,\ep_N}\circ T_{\sg^{n-(k+1)}\om}\right)\cdot\phi_{\sg^{n-(k+1)}\om,0}\right)\right)
			\nonumber\\
			&\qquad
			\leq \norm{\cL_{\sg^{n-(k+1)}\om,0}\ind}_{\infty,\sg^{n-k}\om}\norm{\phi_{\sg^{n-(k+1)}\om,0}}_{\infty,\sg^{n-(k+1)}\om}
			\leq C_1C_2.
			\label{unifeq4}
		\end{align}
		Inserting \eqref{unifeq4} into \eqref{unifeq2.2} 
		and using \eqref{LmUniErr} yields 
		\begin{align}
			&
			\sum_{k=0}^{n-1}\absval{1-\frac{\lm_{\sg^{n-(k+1)}\om,0}}{\lm_{\sg^{n-(k+1)}\om,\ep_N}}}\cdot\absval{ \nu_{\sg^{n-k}\om,0}\left(\cL_{\sg^{n-(k+1)}\om,\ep_N}\left(\left(\hat{X}_{\sg^{n-k}\om,k,\ep_N}\circ T_{\sg^{n-(k+1)}\om}\right)\cdot\phi_{\sg^{n-(k+1)}\om,0}\right)\right)}
			\nonumber\\
			&\qquad
			\leq 
			\sum_{k=0}^{n-1}\frac{C_1C_2E_N}{1-E_N} 
			= n\cdot \frac{C_1C_2E_N}{1-E_N}. 
			\label{unifeq5}
		\end{align}
		Thus, collecting the estimates \eqref{unifeq2.1}-\eqref{unifeq5} together with \eqref{unifeq0} and inserting into \eqref{unifeq1} yields
		\begin{align}
			\absval{1-\nu_{\om,\ep_N}(\phi_{\om,0})}
			\leq
			\left(\frac{1}{1-E_N}\right)^n\cdot
			\left(nE_N+\frac{C_1C_2E_N\cdot n}{1-E_N}\right)
			+
			C_2C_{\phi_0}\al(n),
			\label{nuepsphi0control}
		\end{align}
		which finishes the proof of the claim.
	\end{subproof}

	To finish the proof of Lemma~\ref{lemC9'}, we note that \eqref{nuepsphi0control} holds for $m$-a.e.\ $\omega$, every $N$ sufficiently large, and each $n\ge 1$.
	Given a $\delta>0$, choose and fix $n$ so that $C_2C_{\varphi_0}\alpha(n)<\delta/2$.
	Because $\lim_{N\to\infty}E_N=0$, 
	we may choose $N$ large enough so that first summand in (\ref{nuepsphi0control}) is also smaller than $\delta/2$.
	Thus, $\lim_{N\to \infty}\nu_{\omega,\epsilon_N}(\varphi_{\omega,0})=1$, uniformly in $\omega$.
	This proves \eqref{C9'ineq} for $\nu_{\om,\ep_N}(\varphi_{\omega,0})$;  we immediately obtain the other inequality using (\ref{C5'}) and (\ref{C7'}), and thus the proof of Lemma~\ref{lemC9'} is complete. 
	
\end{proof}

%

We now obtain a formula for the explicit form of Gumbel law for the extreme value distribution.
\begin{theorem}
	\label{evtthm}
	Given a random open system $(\mathlist{\bcomma}{\Om, m, \sg, \cJ_0, T, \cB, \cL_0, \nu_0, \phi_0, H_\ep})$ satisfying \eqref{C1'}, \eqref{C2}, \eqref{C3}, \eqref{C4'}, \eqref{C5'}, \eqref{C7'}, \eqref{C8},  and \eqref{xibound}, 
	for almost every $\omega\in\Omega$ one has
	\begin{equation}
		\label{evtthmeqn}
		\lim_{N\to\infty}\nu_{\om,0}\left(X_{\om,N-1,\ep_N}\right)
		=
		\lim_{N\to\infty}\mu_{\om,0}\left(X_{\om,N-1,\ep_N}\right)
		=
		\lim_{N\to\infty}\frac{\lm_{\om,\ep_N}^N}{\lm_{\om,0}^N}
		=
		\exp\left(-\int_\Om t_\om\ta_{\om,0}\, dm(\om)\right).
	\end{equation}
\end{theorem}

\begin{proof}
	\quad

	\textbf{Step 1: Estimating $\lambda_{\omega,\epsilon_N}/\lambda_{\omega,0}$.}
	To work towards constructing an estimate for $\lm^N_{\om,\ep_N}/\lm^N_{\om,0}$, we first estimate $\lm_{\om,\ep_N}/\lm_{\om,0}$.
	For brevity, in this step we drop the $N$ subscripts on $\ep$.
	Following the proof of Theorem~\ref{thm: GRPT} up to equation (\ref{maineq}), which uses assumptions \eqref{C2} and \eqref{C3}, we have 
	\begin{eqnarray}
		\lefteqn{\nu_{\sg^{-n}\om,\ep}(\phi_{\sg^{-n}\om,0})\frac{\lm_{\om,0}-\lm_{\om,\ep}}{\Dl_{\om,\ep}}}
		\nonumber\\
		\label{def:ta_ep,n}		
		\qquad&=&\underbrace{1-\sum_{k=0}^{n-1} \lm_{\sg^{-(k+1)}\om,0}^{-1}(\lm^k_{\sg^{-k}\om,\ep})^{-1}q^{(k)}_{\om,\ep}}_{=:\ta_{\om,\ep,n}}
		\\
		\label{def:ta'_ep,n}		
		\qquad&+&\Dl^{-1}_{\om,\ep}\underbrace{\sum_{k=1}^n \lm^{-1}_{\sg^{-k}\om,0}(\lm_{\sg^{-k}\om,0}-\lm_{\sg^{-k}\om,\ep})\nu_{\sg\om,0}((\cL_{\om,0}-\cL_{\om,\ep})(\~\cL_{\sg^{-k}\om,\ep}^k)(\phi_{\sg^{-k}\om,0}))}_{=:\ta'_{\om,\ep,n}}
		\\
		\label{def:ta''_ep,n}	
		\qquad&+&\Dl^{-1}_{\om,\ep}\underbrace{\nu_{\sg\om,0}(\cL_{\om,0}-\cL_{\om,\ep})(Q^n_{\sg^{-n}\om,\ep}(\phi_{\sg^{-n}\om,0}))}_{=:\ta''_{\om,\ep,n}}.
	\end{eqnarray}
	By first rearranging to solve for $\lm_{\om,\ep}$ we have
	\begin{eqnarray*}
		\lm_{\om,\ep}&=&\lm_{\om,0}-\frac{\ta_{\om,\ep,n}\Dl_{\om,\ep}+\ta'_{\om,\ep,n}+\ta''_{\om,\ep,n}}{\nu_{\sg^{-n}\om,\ep}(\phi_{\sg^{-n}\om,0})},
	\end{eqnarray*}
	and thus
	\begin{align}
		\frac{\lm_{\om,\ep}}{\lm_{\om,0}}=1-
		\underbrace{
			\frac{\ta_{\om,\ep,n}\Dl_{\om,\ep}}{\lm_{\om,0}\nu_{\sg^{-n}\om,\ep}(\phi_{\sg^{-n}\om,0})}}_{=:Y^{(1)}_{\om,\ep,n}}
			-\underbrace{\frac{\ta'_{\om,\ep,n}}{\lm_{\om,0}\nu_{\sg^{-n}\om,\ep}(\phi_{\sg^{-n}\om,0})}}_{Y^{(2)}_{\om,\ep,n}}
			-\underbrace{\frac{\ta''_{\om,\ep,n}}{\lm_{\om,0}\nu_{\sg^{-n}\om,\ep}(\phi_{\sg^{-n}\om,0})}}_{=:Y^{(3)}_{\om,\ep,n}}.
		\label{Y^ieqn}
	\end{align}
	Setting $Y_{\om,\ep,n}:=Y^{(1)}_{\om,\ep,n}+Y^{(2)}_{\om,\ep,n}+Y^{(3)}_{\om,\ep,n}$ applying Taylor to $\log(1-\cdot)$, we obtain
	\begin{equation}
		\label{Yeqn}
		\frac{\lm_{\om,\ep}}{\lm_{\om,0}}
		=\exp\left(-Y_{\om,\ep,n}-\frac{Y_{\om,\ep,n}^2}{2(1-y)^2}\right),
	\end{equation}
	where $0\le y\le Y_{\om,\ep,n}$.
	Setting $\ep=\ep_N$ in (\ref{Yeqn}) we obtain
	\begin{eqnarray}
		\label{82}
		\frac{\lm^N_{\om,\ep_N}}{\lm_{\om,0}^N}
		&=&
		\exp\left(
		-\sum_{i=0}^{N-1}Y_{\sg^i\om,\ep_N,n} -\sum_{i=0}^{N-1}\frac{Y_{\sg^i\om,\ep_N,n}^2}{2(1-y)^2}
		\right)\\
		\label{82a}	&=&	\exp\left(
		-\frac{1}{N}\sum_{i=0}^{N-1}\left(g^{(1)}_{N,n}(\sg^i\om)+g^{(2)}_{N,n}(\sg^i\om)+g^{(3)}_{N,n}(\sg^i\om)\right) -\sum_{i=0}^{N-1}\frac{Y_{\sg^i\om,\ep_N,n}^2}{2(1-y)^2}
		\right),
	\end{eqnarray}
	where $g^{(j)}_{N,n}(\omega)=NY^{(j)}_{\omega,\epsilon_N,n}$ for $j=1,2,3$ and $0\le y\le Y_{\omega,\epsilon_N,n}$.

	\textbf{Step 2: A non-standard ergodic lemma.}
	In preparation for estimating the products along orbits contained in $\lm^N_{\om,\ep_N}/\lm^N_{\om,0}$, we state and prove a non-standard ergodic lemma.
	\begin{lemma}
		\label{ptwiselemma2}
		For $N\ge 0$, let $g_N\in L^1(m)$.
		Suppose that as $N\to\infty$, $g_N\to g$ $m$-almost everywhere for some $g\in L^1(m)$ and that $\lim_{N\to\infty}\int_\Om |g_N-g|\,dm =0$.
		Then for $m$-a.e. $\om\in\Om$, $\lim_{N\to\infty}\frac{1}{N}\sum_{i=0}^{N-1} g_N(\sg^i\om)$ exists and equals $\mathbb{E}(g)=\int_\Om g\ dm$.
	\end{lemma}
	\begin{proof}
		We write
		$$
		\left|\frac{1}{N}\sum_{i=0}^{N-1} g_N(\sg^i\om) - \mathbb{E}(g)\right|
		\le
		\frac{1}{N}\sum_{i=0}^{N-1} |g_N(\sg^i\om) - g(\sg^i\om)|
		+\left|\frac{1}{N}\sum_{i=0}^{N-1} g(\sg^i\om) - \mathbb{E}(g)\right|.
		$$
		First, we note that the Birkhoff Ergodic Theorem implies that
		\begin{align*}
			\lim_{N\to\infty}\left|\frac{1}{N}\sum_{i=0}^{N-1} g(\sg^i\om) - \mathbb{E}(g)\right|=0.
		\end{align*}
		Now, to deal with the remaining term, since $g_N\to g$ almost everywhere, given $\dl>0$, we let $N_\dl\in\NN$ be sufficiently large such that
		\begin{align*}
			m\lt(\Om_\dl:=\set{\om\in\Om: |g(\om)-g_N(\om)|<\dl \text{ for all } N\geq N_\dl} \rt)>1-\dl.
		\end{align*}
		Birkhoff applied to $\ind_{\Om_\dl}$ then gives that for each $\dl>0$ and all $N\geq N_\dl$ sufficiently large we have that
		\begin{align}\label{eq: erg dens arg}
			\frac{1}{N}\#\set{0\leq k< N: \sg^k\om\in \Om_\dl}>1-\dl.
		\end{align}	
		Since  $\lim_{N\to\infty}\int_\Om|g_N-g|\, dm(\om)=0$, for each $\dl>0$ there exists $N_\dl'>0$ such that $\int_\Om |g_N-g|\, dm<\dl$ for all $N\geq N_\dl'$. 
		We apply the Birkhoff Ergodic Theorem to the functions $|g_N-g|\cdot\ind_{\Om_\dl}$ and  $|g_N-g|\cdot\ind_{\Om_\dl^c}$.
		Using \eqref{eq: erg dens arg} gives that for $m$-a.e. $\om\in\Om$ and all $N\geq \max(N_\dl,N_\dl')$ sufficiently large (so that the Birkhoff error is less than $\delta$, noting that $N$ depends on $\om$) we have
		\begin{align*}
			&\frac{1}{N}\sum_{i=0}^{N-1} |g_N(\sg^i\om) - g(\sg^i\om)|
			\\
			&\qquad
			=
			\frac{1}{N}\sum_{i=0}^{N-1} |g_N(\sg^i\om) - g(\sg^i\om)|\ind_{\Om_\dl}(\sg^i\om)
			+
			\frac{1}{N}\sum_{i=0}^{N-1} |g_N(\sg^i\om) - g(\sg^i\om)|\ind_{\Om_\dl^c}(\sg^i\om)
			\\
			&\qquad
			<
			\dl+\int_{\Om_\dl^c}|g_N-g|\, dm+\dl
			\leq 2\dl+\int_{\Om}|g_N-g|\, dm <3\dl.
		\end{align*}
		As this holds for every $\dl>0$, we must in fact have that
		\begin{align*}
			\lim_{N\to\infty}\left|\frac{1}{N}\sum_{i=0}^{N-1} g_N(\sg^i\om) - \mathbb{E}(g)\right|=0
		\end{align*}
		as desired.
	\end{proof}
	
	\textbf{Step 3: Estimating $g^{(1)}$.}
	
	In this step we construct estimates of $g^{(1)}$ that are required to apply Lemma \ref{ptwiselemma2}.
	
	In preparation for the first use of Lemma \ref{ptwiselemma2} we recall that 
	$$
	g_{N,n}^{(1)}(\omega):=\frac{N\theta_{\omega,\epsilon_N,n}\Delta_{\omega,\epsilon_N}}{\lambda_{\omega,0}\nu_{\sigma^{-n}\omega,\epsilon_N}(\varphi_{\sigma^{-n}\omega,0})}\\
	=\frac{N\theta_{\omega,\epsilon_N,n}\mu_{\omega,0}(H_{\omega,\epsilon_N})}{\nu_{\sigma^{-n}\omega,\epsilon_N}(\varphi_{\sigma^{-n}\omega,0})}
	=\frac{\theta_{\omega,\epsilon_N,n}(t_\omega+\xi_{\omega,N})}{\nu_{\sigma^{-n}\omega,\epsilon_N}(\varphi_{\sigma^{-n}\omega,0})},$$
	where  $\lim_{N\to\infty}\xi_{\omega,N}=0$ for a.e.\ $\omega$ and $|\xi_{\om,N}|\le W$ by (\ref{xibound}).
	We also set $g_n^{(1)}(\omega):=\theta_{\omega,0,n}t_\omega$, where 
	\begin{align}\label{theta_0,n def}
		\ta_{\om,0,n}:=1-\sum_{k=0}^{n-1}\left(\lm_{\sg^{-(k+1)}\om,0}^{k+1}\right)^{-1}q_{\om,0}^{(k)}
		=1-\sum_{k=0}^{n-1}\hat q_{\om,0}^{(k)}.
	\end{align}
	
	By Lemma \ref{qbound2} and \eqref{C8} we see that $0\le \hat q_{\om,0}^{(k)},\hat q_{\om,\ep_N}^{(k)}\le 1$ for each $k$, $N$, and $m$-a.e.\ $\omega$. Thus, \eqref{C8} and the Dominated Convergence Theorem imply that $\lim_{N\to\infty}\|\hat q_{\om,\ep_N}^{(k)}-\hat q_{\om,0}^{(k)}\|_1=0$. From \eqref{def:ta_ep,n} we have that 
	\begin{align}
		\ta_{\om,\ep_N,n}
		:=1-\sum_{k=0}^{n-1}\lm_{\sg^{-(k+1)}\om,0}^{-1}(\lm^k_{\sg^{-k}\om,\ep_N})^{-1}q_{\om,\ep_N}^{(k)}
		=1-\sum_{k=0}^{n-1}\frac{\lm_{\sg^{-k}\om,0}^k}{\lm_{\sg^{-k}\om,\ep_N}^k}\hat q_{\om,\ep_N}^{(k)},
	\end{align}
	and thus that $\ta_{\om,0,N}\in[0,1]$ and $\ta_{\om,\ep_N,n}\leq 1$ for each $n$ and $m$-a.e.\ $\omega$.
	Using \eqref{LmUniErr} and the fact that $\sfrac{\lm_{\om,0}}{\lm_{\om,\ep_N}}\geq 1$ (by Remark~\ref{rem lm_0>lm_ep}), we have that 
	\begin{align*}
		\ta_{\om,\ep_N,n}
		\geq 1-\frac{\lm_{\sg^{-(n-1)}\om,0}^{n-1}}{\lm_{\sg^{-(n-1)}\om,\ep_N}^{n-1}}\sum_{k=0}^{n-1}\hat q_{\om,\ep_N}^{(k)}
		\geq
		1-\left(1+\frac{C_1E_N}{C_1-E_N}\right)^n.
	\end{align*}
	For fixed $n$, again we apply Dominated Convergence to get that 
	\begin{align}\label{ta_ep to ta_0 in L1}
	    \lim_{N\to\infty}\|\ta_{\om,\ep_N,n}-\ta_{\om,0,n}\|_1=0.
	\end{align}
	Using Lemma~\ref{lemC9'} and \eqref{xibound} we see that $g_{N,n}^{(1)},g_n^{(1)}\in L^1(m)$ for each $N,n$.
Moreover using these same facts	
\begin{eqnarray*}
		\|g_{N,n}^{(1)}-g_n^{(1)}\|_1&\le& 
	\int \left(\left|\frac{\theta_{\omega,\epsilon_N,n}(t_\omega+\xi_{\om,N})}{\nu_{\sigma^{-n}\omega,\epsilon_N}(\varphi_{\sigma^{-n}\omega,0})}-\theta_{\om,0,n}t_\om\right| \right)\ dm.
	\end{eqnarray*}
	The integrand goes to zero $\om$ almost surely and it is bounded by $C_{\ep_N}(|t|_\infty+W)+|t|_\infty.$ We could therefor apply dominated convergence to prove that the $L^1(m)$ difference goes to zero.

	Referring to the first term in the $Y$-sum in (\ref{82}), we may now apply Lemma \ref{ptwiselemma2} to conclude that 
	$$\sum_{i=0}^{N-1}\frac{\theta_{\sigma^i\omega,\epsilon_N,n}\Delta_{\sigma^i\omega,\epsilon_N}}{\lambda_{\sigma^i\omega,0}\nu_{\sigma^{-n+i}\omega,\epsilon_N}(\varphi_{\sigma^{-n+i}\omega,0})}=\frac{1}{N}\sum_{i=0}^{N-1}g_{N,n}^{(1)}(\sigma^i\omega)\to \int_\Omega g_n^{(1)}(\omega)\ dm(\omega)=\int_\Omega \theta_{\omega,0,n}t_\omega\ dm(\omega),$$
	as $N\to\infty$ for each $n$ and $m$-a.e.\ $\omega$.
	
	

	\textbf{Step 4: Estimating $g^{(2)}$.} 
	We now perform a similar analysis to the previous step to control the terms in the sum (\ref{82}) corresponding to $\theta'$.
	Again in preparation for applying Lemma \ref{ptwiselemma2}, recall that
	$$g_{N,n}^{(2)}(\omega)=\frac{N\theta'_{\omega,\epsilon_N,n}}{\lambda_{\omega,0}\nu_{\sigma^{-n}\omega,\epsilon_N}(\varphi_{\sigma^{-n}\omega,0})}$$
	and set $g_n^{(2)}(\omega)\equiv 0$ for each $n$.
	Using \eqref{def:ta'_ep,n}, (\ref{opdiff}), (\ref{EVTlamdiff2}), and (\ref{C7'}), for sufficiently large $N$ we have  
	\begin{eqnarray*}
		\ta_{\om,\ep_N,n}'
		&:=&\sum_{k=1}^n \lm^{-1}_{\sg^{-k}\om,0}(\lm_{\sg^{-k}\om,0}-\lm_{\sg^{-k}\om,\ep})\nu_{\sg\om,0}((\cL_{\om,0}-\cL_{\om,\ep})(\~\cL_{\sg^{-k}\om,\ep}^k)(\phi_{\sg^{-k}\om,0}))
		\\
		&\leq& 
		\frac{C_2C_3\lambda_{\omega,0}}{N}\sum_{k=1}^n\lambda_{\sigma^{-k}\omega,0}^{-1}\left(\lambda_{\sigma^{-k}\omega,0}(t_{\sg^{-k}\om}+\xi_{\sg^{-k}\om,N})\right)
		\nu_{\om,0}\lt(\ind_{H_{\om,\ep_N}}\~\cL_{\sg^{-k}\om,\ep_N}^k\phi_{\sg^{-k}\om,0}\rt)
		\\
		&\leq&
		\frac{C_2C_3\lambda_{\omega,0}}{N}\sum_{k=1}^n(t_{\sg^{-k}\om}+\xi_{\sg^{-k}\om,N})
		\nu_{\om,0}(H_{\om,\ep_N})\norm{\~\cL_{\sg^{-k}\om,\ep_N}^k\phi_{\sg^{-k}\om,0}}_{\infty,\om}
		\\
		&\leq& 
		\frac{C_2C_3\lambda_{\omega,0}}{N}\cdot \frac{C_3(t_\om+\xi_{\om,N})}{N}
		\sum_{k=1}^n(t_{\sg^{-k}\om}+\xi_{\sg^{-k}\om,N})
		\norm{\~\cL_{\sg^{-k}\om,\ep_N}^k\phi_{\sg^{-k}\om,0}}_{\infty,\om}.
	\end{eqnarray*}
	Using Lemma~\ref{lemC9'}, (\ref{C5'}), and (\ref{C4'}) we note that
	\begin{eqnarray}
		\nonumber		\norm{\~\cL_{\sg^{-k}\om,\ep_N}^k\phi_{\sg^{-k}\om,0}}_{\infty,\om}
		&=&
		\norm{\nu_{\sg^{-k}\om,\ep_N}(\phi_{\sg^{-k}\om,0})\phi_{\om,\ep_N}+Q_{\sg^{-k}\om,\ep_N}^k\phi_{\sg^{-k}\om,0}}_{\infty,\om}
		\\
		\nonumber		&\leq &
		\norm{C_{\ep_N}C_2+C_{\phi_0}C_2\al(k)}_{\infty,\om}
		\\
		\label{Lpowerinfbound}		&\leq&
		C_2C_{\ep_N}+C_2C_{\phi_0}\al.
	\end{eqnarray}
	Finally, we have for $N$ sufficiently large and by (\ref{xibound}) 
	\begin{eqnarray*}
		g_{N,n}^{(2)}(\omega)&\le& \frac{NC_2C_3\lambda_{\omega,0}(C_2C_{\ep_N}+C_2C_{\phi_0}\al)}{N\lambda_{\omega,0}\nu_{\sg^{-n}\om,\ep_{N}}(\phi_{\sg^{-n}\om,0})}\cdot \frac{C_3(t_\om+\xi_{\om,N})}{N}
		\sum_{k=1}^n(t_{\sg^{-k}\om}+\xi_{\sg^{-k}\om,N}).
		\end{eqnarray*}
	
Integrability of $g_{N,n}^{(2)}$ follows from the the fact that $t\in L^\infty(m)$ and $|\xi_{\om,N}|\le W,$ for almost all $\om.$
	Moreover by  (\ref{xibound}), for each $n$ we have that $g_{N,n}^{(2)}\to 0$ almost everywhere as $N\to\infty$.
	By dominated convergence  we have again that $\lim_{N\to\infty}\|g_{N,n}^{(2)}-g_n^{(2)}\|_1\to 0$ for each $n$.
	
	Referring to $Y_{\om,\ep_N,n}^{(2)}$ in \eqref{Y^ieqn}, we may now apply Lemma \ref{ptwiselemma2} to conclude that 
	$$
	\sum_{i=0}^{N-1}\frac{\theta'_{\sigma^i\omega,\epsilon_N,n}}{\lambda_{\sigma^i\omega,0}\nu_{\sigma^{-n+i}\omega,\epsilon_N}(\varphi_{\sigma^{-n+i}\omega,0})}=\frac{1}{N}\sum_{i=0}^{N-1}g_{N,n}^{(2)}(\sigma^i\omega)\to \int_\Omega g_n^{(2)}(\omega)\ dm(\omega)=0,
	$$
	as $N\to\infty$ for each $n$ and a.e.\ $\omega$.
	
	
	\textbf{Step 5: Estimating $g^{(3)}$.}
	We repeat a similar analysis to control the terms in the sum (\ref{82}) corresponding to $\theta''$.
	Again in preparation for applying Lemma \ref{ptwiselemma2}, recall that
	$$g_{N,n}^{(3)}(\omega)=\frac{N\theta''_{\omega,\epsilon_N,n}}{\lambda_{\omega,0}\nu_{\sigma^{-n}\omega,\epsilon_N}(\varphi_{\sigma^{-n}\omega,0})}.$$
	We begin developing an upper bound for $|g_{N,n}^{(3)}|$.
	Using \eqref{def:ta''_ep,n}, (\ref{C2}), (\ref{C4'}), and  (\ref{C7'}) we have for sufficiently large $N$ that 
	\begin{align*}
		|\ta_{\om,\ep_N,n}''|&:=
		\nu_{\sg\om,0}(\cL_{\om,0}(\ind_{H_{\om,\ep_N}}Q^n_{\sg^{-n}\om,\ep}(\phi_{\sg^{-n}\om,0})))
		\\
		&=
		\lambda_{\omega,0}\nu_{\om,0}(\ind_{H_{\om,\ep_N}}Q_{\sg^{-n}\om,\ep_N}^n\phi_{\sg^{-n}\om,0})		\\
		&\leq 
		\lambda_{\omega,0}\nu_{\om,0}(H_{\om,\ep_N})\norm{Q_{\sg^{-n}\om,\ep_N}^n\phi_{\sg^{-n}\om,0}}_{\infty,\om}
		\\
		&\leq
		\lambda_{\omega,0}\nu_{\om,0}(H_{\om,\ep_N})C_{\phi_0}\norm{\phi_{\sg^{-n}\om,0}}_{\cB_{\sg^{-n}\om}}\al(n)
		\\
		&\leq
		\lambda_{\omega,0}C_{\phi_0}C_2C_3\al(n)\mu_{\om,0}(H_{\om,\ep_N})
		\\
		&=
		\frac{\lambda_{\omega,0}C_{\phi_0}C_2C_3\al(n) (t_\om+\xi_{\om,N})}{N}.
	\end{align*}
	Therefore, using Lemma~\ref{lemC9'}
	$$|g_{N,n}^{(3)}(\omega)|\le \frac{N\lm_{\om,0}C_{\phi_0}C_2C_3\al(n) (t_\om+\xi_{\om,N})}{N\lm_{\om,0}\nu_{\sg^{-n}\om,\ep_{N}}(\phi_{\sg^{-n}\om,0}))}\le C_{\epsilon_N}C_{\phi_0}C_2C_3\al(n) (t_\om+\xi_{\om,N})=:\tilde{g}^{(3)}_{N,n}(\omega).
	$$
	We set $\tilde{g}^{(3)}_n(\omega)=C_{\varphi_0}C_2C_3\alpha(n)t_\omega$. 
	Integrability of $\tilde{g}_{N,n}^{(3)}$ and $\tilde{g}_{n}^{(3)}$ follows from (\ref{xibound}) and the fact $t\in L^\infty(m)$ and $|\xi_{\om,N}|\le W,$ for almost all $\om.$ 
	Similarly, (recalling that $C_{\ep_N}\to 1$ as $N\to\infty$ by Lemma~\ref{lemC9'}) for each $n$, $\tilde{g}_{N,n}^{(3)}\to \tilde{g}_{n}^{(3)}$  almost everywhere as $N\to\infty$.
	For the same reasons
	by \eqref{xibound} and dominated convergence  
	we also have that $\lim_{N\to\infty}\|\tilde{g}_{N,n}^{(3)}-\tilde{g}_n^{(3)}\|_1\to 0$ for each $n$.
	
	Referring to $Y_{\om,\ep_N,n}^{(3)}$ in \eqref{Y^ieqn}, we may now apply Lemma \ref{ptwiselemma2} to $\tilde{g}_{N,n}^{(3)}$ and $\tilde{g}_n^{(3)}$ to conclude that 
	$$\frac{1}{N}\sum_{i=0}^{N-1}g_{N,n}^{(3)}(\sigma^i\omega)\le\frac{1}{N}\sum_{i=0}^{N-1}\tilde{g}_{N,n}^{(3)}(\sigma^i\omega)
	\to \int_\Omega \tilde{g}_n^{(3)}(\omega)\ dm(\omega)
	=C_{\varphi_0}C_2C_3\alpha(n)\int_\Omega t_\omega\ dm(\omega)$$
	as $N\to\infty$ for each $n$ and a.e.\ $\omega$.

	\textbf{Step 6: Finishing up.}\label{step6}
	Recall from (\ref{evtexpression3}) that
	\begin{equation*}
		(\ref{evtexpression})=\frac{\lm_{\om,\ep_N}^N}{\lm_{\om,0}^N}
		\lt(\nu_{\om,\ep_N}(\ind)+\nu_{\sg^N\om,0}\lt(Q_{\om,\ep_N}^N( \ind)\rt)\rt).
	\end{equation*}
	Using (\ref{C4'}) and Lemma~\ref{lemC9'} we see that $\lim_{N\to\infty}\lt(\nu_{\om,\ep_N}(\ind)+\nu_{\sg^N\om,0}\lt(Q_{\om,\ep_N}^N( \ind)\rt)\rt)=1$ for $m$-a.e.\ $\omega\in\Omega$.
	By (\ref{82a}) and Steps 3, 4, and 5 we see that for any $n$ 
	\begin{eqnarray}
		\label{ratiobound}	\lefteqn{\lim_{N\to\infty}\exp\left(	-\frac{1}{N}\sum_{i=0}^{N-1}\left(g^{(1)}_{N,n}(\sg^i\om)+g^{(2)}_{N,n}(\sg^i\om)+g^{(3)}_{N,n}(\sg^i\om)\right)\right)\exp\left( -\sum_{i=0}^{N-1}\frac{Y_{\sg^i\om,\ep_N,n}^2}{2(1-y)^2}
			\right)}\\
		\nonumber	&&\le \lim_{N\to\infty} \frac{\lambda_{\omega,\epsilon_N}^N}{\lambda_{\omega,\epsilon_N}^N}\\
		\nonumber	&&\le
		\lim_{N\to\infty}\exp\left(
		-\frac{1}{N}\sum_{i=0}^{N-1}\left(g^{(1)}_{N,n}(\sg^i\om)+g^{(2)}_{N,n}(\sg^i\om)-\tilde{g}^{(3)}_{N,n}(\sg^i\om)\right)\right)\exp\left( -\sum_{i=0}^{N-1}\frac{Y_{\sg^i\om,\ep_N,n}^2}{2(1-y)^2}
		\right),
	\end{eqnarray}
	where $0\le y\le Y_{\sg^i\om,\ep_N,n}$.
	We now treat the Taylor remainder terms.
	From Steps 3, 4, and 5 for all $N$ sufficiently large we have and for almost all $\om$: 
	\begin{eqnarray}
		\label{g1bound}|g^{(1)}_{N,n}|&\le& C_{\epsilon_N}(|t|_\infty+W), 
		\\
		\label{g2bound}|g^{(2)}_{N,n}|&\le& {C_{\epsilon_N}C_2C_3(C_2C_{\ep_N}+C_2C_{\phi_0}\al)}\cdot \frac{C_3n(|t|_\infty+W)^2}{N},
		\\
		\label{g3bound}|g^{(3)}_{N,n}|&\le& \~ g^{(3)}_{N,n}:= C_{\phi_0}C_2C_3\al(|t|_\infty+W).
	\end{eqnarray}
	Further,
	\begin{equation}
		\label{remainderbound}\sum_{i=0}^{N-1}\frac{Y_{\sg^i\om,\ep_N,n}^2}{2(1-y)^2}\le \sum_{i=0}^{N-1}\frac{Y_{\sg^i\om,\ep_N,n}^2}{2(1-Y_{\sg^i\om,\ep_N,n})^2}=\sum_{i=0}^{N-1}\frac{(G_{\sg^i\om,\ep_N,n}/N)^2}{2(1-(G_{\sg^i\om,\ep_N,n}/N))^2},
	\end{equation}
	where $G_{\omega,\ep_N,n}:=g^{(1)}_{N,n}+g^{(2)}_{N,n}+g^{(3)}_{N,n}$.
	Using the bounds (\ref{g1bound})--(\ref{g3bound}) we see that (\ref{remainderbound}) approaches $0$ for almost all $\omega$ for each $n$ as $N\to\infty$.
	Therefore, combining the expressions developed in Steps 3, 4, and 5 for the $N\to\infty$ limits with (\ref{ratiobound}) we have
	\begin{eqnarray*}
		\exp\lt(-
		\int_{\Omega}\theta_{\omega,0,n}t_\omega\ dm(\omega)
		\rt)
		\le 
		\lim_{N\to\infty} \frac{\lambda_{\omega,\epsilon_N}^N}{\lambda_{\omega,0}^N}
		\le 
		\exp\lt(-
		\int_{\Omega}\theta_{\omega,0,n}t_\omega\ dm(\omega)+C_2C_3C_{\varphi_0}\alpha(n)\int_\Omega t_\omega\ dm(\omega)
		\rt).
	\end{eqnarray*}
	Recalling the definitions of $\theta_{\omega,0,n}$ \eqref{theta_0,n def} and $\theta_{\omega,0}$ \eqref{eq: def of theta_0} and the fact that $0\le\theta_{\omega,0,n}\le 1$ (by \eqref{theta in [0,1]}), we may use dominated convergence to take the $n\to\infty$ limit to obtain
	$$
	\exp\lt(-
	\int_{\Omega}\theta_{\omega,0}t_\omega\ dm(\omega) \rt)
	\le \lim_{N\to\infty} \frac{\lambda_{\omega,\epsilon_N}^N}{\lambda_{\omega,0}^N}
	\le \exp\lt(-
	\int_{\Omega}\theta_{\omega,0}t_\omega\ dm(\omega)\rt),$$
	thus completing the proof that 
	\begin{equation*}
		\lim_{N\to\infty}\nu_{\om,0}\left(X_{\om,N-1,\ep_N}\right)
		=
		\lim_{N\to\infty}\frac{\lm_{\om,\ep_N}^N}{\lm_{\om,0}^N}
		=
		\exp\left(-\int_\Om t_\om\ta_{\om,0}\, dm(\om)\right).
	\end{equation*}
	To see that $\lim_{N\to\infty}\mu_{\om,0}\left(X_{\om,N-1,\ep_N}\right)$ is also equal to this value, we simply recall that \eqref{evtexpression2mu} gives that
	\begin{align*}
		\mu_{\om,0}(X_{\om,N-1,\ep_N})
		=
		\frac{\lm_{\om,\ep_N}^N}{\lm_{\om,0}^N}
		\lt(\nu_{\om,\ep_N}(\phi_{\om,0})+\nu_{\sg^N\om,0}\lt(Q_{\om,\ep_N}^N( \phi_{\om,0})\rt)\rt).
	\end{align*}
	Now since (\ref{C4'}) and Lemma~\ref{lemC9'} together give that $$\lim_{N\to\infty}\lt(\nu_{\om,\ep_N}(\phi_{\om,0})+\nu_{\sg^N\om,0}\lt(Q_{\om,\ep_N}^N( \phi_{\om,0})\rt)\rt)=1$$ for $m$-a.e.\ $\omega\in\Omega$, we must in fact have that 
	\begin{equation*}
		\lim_{N\to\infty}\mu_{\om,0}\left(X_{\om,N-1,\ep_N}\right)
		=
		\lim_{N\to\infty}\frac{\lm_{\om,\ep_N}^N}{\lm_{\om,0}^N}
		=
		\exp\left(-\int_\Om t_\om\ta_{\om,0}\, dm(\om)\right),
	\end{equation*}
	which completes the proof of Theorem~\ref{evtthm}.
	
\end{proof}

\subsection{The relationship between condition \eqref{os} and the H\"usler condition}
We now return to the discussion initiated in the Introduction to compare our assumption (\ref{os}) for the thresholds $z_{\om, N}$ with the H\"usler type condition (\ref{BLH}). 
We show that in the more general situation considered in our paper with random boundary level $t_{\omega}$, the limit (\ref{BLH}) will follow from  the simpler assumption (\ref{os}),  provided we replace $t$ in (\ref{BLH}) with the expectation of $t_{\omega}.$ 

Recall our  assumption for the choice of the thresholds (\ref{xibound}) is $\mu_{\omega,0}(h_{\omega}(x)> z_{\omega, N})=(t_{\omega}+\xi_{\omega, N})/N$,
where $\xi_{\om, N}$ goes to zero almost surely when $N\rightarrow \infty$ and $\xi_{\omega,N}\le W$ for a.e. $\omega$.  It is immediate to see by dominated convergence that: 
\begin{equation}\label{dc}
	\lim_{N\rightarrow \infty}\int_\Om |N\mu_{\omega,0}(h_{\omega}(x)> z_{\omega, N})-t_{\omega}|\ dm(\omega)=0.
\end{equation}
Applying our non-standard ergodic  Lemma \ref{ptwiselemma2} with $g_N(\om):=N\mu_{\om,0}(h_{\omega}(x)> z_{\omega, N})$ and $g(\om):=t_\om$, 
one may transform the sum in  (\ref{BLH}) as follows: 
$$
\lim_{N\rightarrow \infty}\frac{1}{N} \sum_{i=0}^{N-1}N\mu_{\omega,0}(h_{\sigma^j\omega}(T^j_{\omega}(x))> z_{\sigma^j\omega, N})
=\int_\Om t_{\omega}\ dm(\omega),
$$
which is the condition (\ref{BLH}) with $t$ replaced by the expectation of $t_{\omega}.$ 
\subsection{Hitting time statistics}
\label{sec:hts}
It is well known that in the deterministic setting there is a close relationship between extreme value theory and the statistics of first hitting time, see for instance \cite{FFT10, book}. We now show how our Theorem \ref{evtthm}, with a slight modification,  can be interpreted in terms of a suitable definition of (quenched) first hitting time distribution.
Let us consider as in the previous sections, a sequence of small random holes $\mathcal{H}_{\om,N}:=\{H_{\sigma^j\om, \epsilon_N}\}_{j\ge 0},$ and define
the first random hitting time as
$$
\tau_{\om, \mathcal{H}_{\om,N}}(x)=\inf\{k\ge 1, T^k_{\om}(x)\in H_{\sigma^k\om, \epsilon_N}\}.
$$

We recall that the usual statistics of hitting times is written in the form
$
\mu_{\om,0}\left(\tau_{\om, \mathcal{H}_{\om,N}}>t\right),
$
for nonnegative values of $t.$ 
Since the sets $H_{\sigma^j\om, \epsilon_N}$ have measure tending to zero when $N\rightarrow \infty,$ and therefore the first  hitting times could eventually grow to infinity, one needs a rescaling in order to get a meaningful limit distribution. This is achieved in the next Proposition. 
In our current setting, condition (\ref{xibound}) reads:
$
	\mu_{\om,0}(H_{\om, \epsilon_N})=\frac{t_{\om}+\xi_{\om, N}}{N}, 
	$
	with 
 $\lim_{N\to\infty} \xi_{\omega,N}=0$ for a.e.\ $\omega$ and 
  $|\xi_{\omega,N}|\le W$ for a.e.\ $\omega$ and all $N\ge 1$.

  
  
\begin{proposition}\label{hve}
	If our random open system $(\mathlist{\bcomma}{\Om, m, \sg, \cJ_0, T, \cB, \cL_0, \nu_0, \phi_0, H_\ep})$ satisfies the assumptions of Theorem \ref{evtthm} with  the sequence $H_{\om,\ep_N}$ verifying condition (\ref{xibound}),
	then the first random hitting time satisfies the limit, for $\om$ $m$-a.e.
	\begin{equation}
		\label{eqhit}
		\lim_{N\to\infty}\mu_{\om,0}\left(\tau_{\om, \mathcal{H}_{\om,N}} \mu_{\om,0}(H_{\om, \epsilon_N})>t_{\om}\right) = \exp\left(-{\int_\Om t_{\om} \theta_{\om,0}dm}\right).
	\end{equation}
\end{proposition}
\begin{proof}
	For $N\ge 1$ the event, 
	\begin{equation}\label{E2}
		\{\tau_{\om,  \mathcal{H}_{\om,N}}>N\}=\{x\in I; T_{\om}(x)\in H_{\sigma\om, \epsilon_N}^c, \dots, T^N_{\om}(x)\in H_{\sigma^N\om, \epsilon_N}^c\}
	\end{equation}
	is also equal to
	$$
	T^{-1}_{\om}\left(x\in I; x\in H_{\sigma\om, \epsilon_N}^c, T_{\sigma\om}(x)\in H_{\sigma^2\om, \epsilon_N}^c,\dots, T_{\sigma\om}^{N-1}(x)\in H_{\sigma^N\om, \epsilon_N}^c\right).
	$$
	Then, by  equivariance of $\mu_0$ we obtain the link between the statistics of hitting time and extreme value theory:
	\begin{eqnarray}\label{li}
		\mu_{\om,0}\left(\tau_{\om, \mathcal{H}_{\om,N}}>N\right)&=&\mu_{\om,0}\left(T^{-1}_{\om}\left(x\in I: x\in H_{\sigma\om, \epsilon_N}^c, T_{\sigma\om}(x)\in H_{\sigma^2\om, \epsilon_N}^c,\dots, T_{\sigma\om}^{N-1}(x)\in H_{\sigma^N\om, \epsilon_N}^c\right)\right)\\
		\nonumber&=&
		\mu_{\sigma\om,0}\left(x\in H_{\sigma\om, \epsilon_N}^c, T_{\sigma\om}(x)\in H_{\sigma^2\om, \epsilon_N}^c,\dots, T_{\sigma\om}^{N-1}(x)\in H_{\sigma^N\om, \epsilon_N}^c\right)\\
		\label{X}&=&\mu_{\sigma\omega,0}(X_{\sigma\om,N-1,\ep_N}).
	\end{eqnarray}
	
	In order to rescale the eventually growing first random  hitting times, we invoke the condition  \eqref{xibound};  by substituting $N=(t_{\om}+\xi_{\om, N})/\mu_{\om,0}(H_{\om, \epsilon_N})$ in the LHS of (\ref{li})
	we have
	\begin{equation}\label{ev}
		\mu_{\om,0}\left(\tau_{\om, \mathcal{H}_{\om,N} }>N\right)=\mu_{\om,0}\left(\tau_{\om, \mathcal{H}_{\om,N}} \mu_{\om,0}(H_{\om, \epsilon_N})>t_{\om}+\xi_{\om, N}\right).
	\end{equation}
	
	Our final preparation before applying Theorem \ref{evtthm} is to show that 
	\begin{equation}\label{ee}
		|\mu_{\om,0}\left(\tau_{\om, \mathcal{H}_{\om,N}} \mu_{\om,0}(H_{\om, \epsilon_N})>t_{\om}+\xi_{\om, N}\right)-\mu_{\om,0}\left(\tau_{\om, \mathcal{H}_{\om,N}} \mu_{\om,0}(H_{\om, \epsilon_N})>t_{\om}\right)|\rightarrow 0, \ N\rightarrow \infty.
	\end{equation}
	Since by a standard trick, see for instance eq. 5.3.6 in \cite{book},
	$$
	\{\tau_{\om, \mathcal{H}_{\om,N}} \mu_{\om,0}(H_{\om, \epsilon_N})>t_{\om}\}\backslash \{\tau_{\om, \mathcal{H}_{\om,N}} \mu_{\om,0}(H_{\om, \epsilon_N})>t_{\om}+\xi_{\om, N}\}\subset \bigcup_{j=\left\lceil{\frac{t_{\om}}{\mu_{\omega,0}(H_{\om, \epsilon_N})}}\right\rceil}^{\left\lceil{\frac{t_{\om}+\xi_{\omega,n}}{\mu_{\omega,0}(H_{\om, \epsilon_N})}}\right\rceil}
	T_{\om}^{-j}(H_{\sigma^j\om, \epsilon_N})
	$$
	we have by equivariance
	$$
	|\mu_{\om,0}\left(\tau_{\om, \mathcal{H}_{\om,N}} \mu_{\om,0}(H_{\om, \epsilon_N})>t_{\om}+\xi_{\om, N}\right)-\mu_{\om,0}\left(\tau_{\om, \mathcal{H}_{\om,N}} \mu_{\om,0}(H_{\om, \epsilon_N})>t_{\om}\right)|\le \sum_{j=\left\lceil{\frac{t_{\om}}{\mu_{\omega,0}(H_{\om, \epsilon_N})}}\right\rceil}^{\left\lceil{\frac{t_{\om}+\xi_{\omega,n}}{\mu_{\omega,0}(H_{\om, \epsilon_N})}}\right\rceil}
	\mu_{\sigma^j\om,0}({H}_{\sigma^j\om,N}).
	$$
	For $N$ large enough:
	$$
	\sum_{j=\left\lceil{\frac{t_{\om}}{\mu_{\omega,0}(H_{\om, \epsilon_N})}}\right\rceil}^{\left\lceil{\frac{t_{\om}+\xi_{\omega,n}}{\mu_{\omega,0}(H_{\om, \epsilon_N})}}\right\rceil}
	\mu_{\sigma^j\om,0}({H}_{\sigma^j\om,N})
	\le \left\lceil{\frac{|\xi_{\om, N}|}{\mu_{\omega,0}(H_{\om,\epsilon_N})}}\right\rceil\frac{|t|_\infty+W}
	{N}
	\le \left\lceil{\frac{|\xi_{\om, N}|\ N}{t_{\om}-|\xi_{\om, N}|}}\right\rceil\frac{|t|_\infty+W}{N},
	$$
	which goes to zero by \eqref{xibound}.
	Recalling $t_\om>0$ for a.e.\ $\om$, the final expression above goes to zero as $N\rightarrow \infty$  for $\om$ $m$-a.e. 
	
	Using \eqref{X}--\eqref{ee} and noting that $\lim_{N\to\infty}\mu_{\sigma\om,0}(X_{\sigma\om,N-1,\ep_N})$ is nonrandom, applying Theorem \ref{evtthm} yields
	\begin{equation}\label{hts}
		\lim_{N\to\infty}\mu_{\om,0}\left(\tau_{\om, \mathcal{H}_{\om,N}} \mu_{\om,0}(H_{\om, \epsilon_N})>t_{\om}\right) = \exp\left(-\int_\Om t_{\om} \theta_{\om,0}\ dm\right).
	\end{equation}
\end{proof}

\section{Quenched Thermodynamic Formalism for Random Open Interval Maps via Perturbation}\label{sec: existence}
In this section we present an explicit class of random piecewise-monotonic interval maps for which our Theorem~\ref{thm: dynamics perturb thm}, Corollary~\ref{esc rat cor},  and Theorem~\ref{evtthm}  apply. 
Using a perturbative approach, we introduce a family of small random holes parameterised by $\epsilon>0$ into a random closed dynamical system, and for every small $\ep$ we prove (i) the existence of a unique random conformal measure $\set{\nu_{\om,\ep}}_{\om\in\Om}$ with fiberwise support in $X_{\om,\infty,\ep}$ and (ii) a unique random absolutely continuous invariant measure $\set{\mu_{\om,\ep}}_{\om\in\Om}$ which satisfies an exponential decay of correlations and is the unique relative equilibirum state for the random open system $(\mathlist{\bcomma}{\Om, m, \sg, \cJ_0, T, \cB, \cL_0, \nu_0, \phi_0, H_\ep})$. In addition, we prove the existence of a random absolutely continuous (with respect to $\nu_{\om,0}$) conditionally invariant probability measure $\set{\vrho_{\om,\ep}}_{\om\in\Om}$ with fiberwise support in $[0,1]\bs H_{\om,\ep}$.

We now suppose that the spaces $\cJ_{\om,0}=[0,1]$ for each $\om\in\Om$ and the maps $T_\om:[0,1]\to[0,1]$ are surjective, finitely-branched, piecewise monotone, nonsingular (with respect to Lebesgue), and that there exists $C\geq 1$ such that 
\begin{align}\label{E1}\tag{E1}
	\esssup_\om |T_\om'|\leq C
	\qquad\text{ and }\qquad
	\esssup_\om D(T_\om)
	\leq C,
\end{align}
where $D(T_\om):=\sup_{y\in[0,1]}\# T_\om^{-1}(y)$.
We let $\cZ_{\om,0}$ denote the (finite) monotonicity partition of $T_\om$ and for each $n\geq 2$ we let $\cZ_{\om,0}^{(n)}$ denote the partition of monotonicity of $T_\om^n$.
\begin{enumerate}[align=left,leftmargin=*,labelsep=\parindent]
	\item[\mylabel{MC}{M}]
	The map $\sg:\Om\to\Om$ is a homeomorphism, the skew-product map $T:\Om\times [0,1]\to\Om\times [0,1]$ is measurable, and $\omega\mapsto T_\omega$ has countable range. 
\end{enumerate}
\begin{remark}\label{M2 holds}
Under assumption (\ref{M}), the family of transfer operator cocycles $\{\mathcal{L}_{\omega,\epsilon}\}_{\epsilon\ge 0}$ satisfies the conditions of Theorem 17 \cite{FLQ2} ($m$-continuity and $\sigma$ a homeomorphism). Note that condition \eqref{M}
	implies that $T$ satisfies \eqref{M1} and the cocycle generated by $\cL_0$ satisfies condition \eqref{M2}.
\end{remark}


Let the \emph{variation} of $f:[0,1]\to\mathbb{R}_+$ on $Z\subset [0,1]$ be
\begin{align*}
    \var_{Z}(f)=\underset{x_0<\dots <x_k, \, x_j\in Z}{\sup}\sum_{j=0}^{k-1}\absval{f(x_{j+1})-f(x_j)}, 
\end{align*}
and  $\var(f):=\var_{[0,1]}(f)$. We let $\BV=\BV([0,1])$ denote the set of functions on $[0,1]$ that have bounded variation. 
Given a non-atomic and fully supported measure $\nu$ (i.e.\ for any nondegenerate interval $J\sub [0,1]$ we have $\nu(J)>0$) we let $\BV_\nu\sub L^\infty(\nu)$ be the set of (equivalence classes of) functions of \emph{bounded variation} on $[0,1]$, with norm given by
\begin{align*}
    \|f\|_{\BV_\nu}:=\underset{\tilde{f}=f \ \nu\text{ a.e.}}{\inf} \var(\tilde{f})+\nu(|f|).    
\end{align*} 
If we require to emphasise that elements of $\BV_\nu$ are equivalence classes, we denote these by $[f]_\nu$ (resp. $[f]_1$).
Note that if $f\in\BV$ is a function of bounded variation, then it is always possible to choose a representative of minimal variation from the equivalence class $[f]_\nu$.
We define $\BV_{1}\sub L^\infty(\Leb)$ and $\|\spot\|_{\BV_1}$ similarly, with the measure $\nu$ replaced with Lebesgue measure.  We denote the supremum norm on $L^\infty(\Leb)$ by $\|\spot\|_{\infty,1}$. 
It follows from Rychlik \cite{rychlik_bounded_1983} that $\BV_\nu$ and $\BV_1$ are Banach spaces. 
The following proposition gives the equivalence of the norms $\|\spot\|_{\BV_\nu}$ and $\|\spot\|_{\BV_1}$.
\begin{proposition}\label{prop norm equiv}
Given a fully supported and non-atomic measure $\nu$ on $[0,1]$ and $f\in\BV$ we have that 
\begin{equation*}
	(1/2) \|f\|_{\BV_1}\le \|f\|_{\BV_\nu}\le 2\|f\|_{\BV_1}.
\end{equation*}
\end{proposition} 

\begin{proof}
We first show that for $f\in\BV$ we have 
\begin{align}\label{eq class cap BV}
        [f]_\nu\cap \BV = [f]_1\cap\BV.
\end{align}
To see this let $\~f\in[f]_\nu\cap \BV$. As $\nu$ is a fully supported and non-atomic measure, we must have that the set $\{x: f(x)\neq\~f(x)\}$ is countable. Thus $\~f\in[f]_1\cap \BV$. As $\Leb$ is also fully supported and non-atomic the same reasoning implies that the reverse inclusion also holds, proving \eqref{eq class cap BV}. As a direct consequence of \eqref{eq class cap BV} we have that 
\begin{align}\label{equiv cl var eq}
    \underset{\tilde{f}=f \ \nu\text{ a.e.}}{\inf} \var(\tilde{f})
    =
    \underset{\tilde{f}=f \ \Leb\text{ a.e.}}{\inf} \var(\tilde{f}).
\end{align}
 Since $f$ is continuous everywhere except on a set of at most countably many points, letting $\cC$ denote the set of intervals of continuity for $f$,  we have 
    \begin{align}\label{nu Leb ineq 1}
        \Lebessinf_{[0,1]} f 
        &=\inf_{J \in\cC } \Lebessinf_J f
        =\inf_{J \in\cC } \inf_J f
        =\inf_{J \in\cC } \nuessinf_J f
        = \nuessinf_{[0,1]} f. 
    \end{align}
    Using similar reasoning we must also have 
    \begin{align}\label{nu Leb ineq 2}
         \Lebesssup f = \nuesssup f.
    \end{align}
    Combining \eqref{equiv cl var eq} and \eqref{nu Leb ineq 1} we have 
    \begin{align*}
        \|f\|_{\BV_1}=\inf_{\tilde{f}=f\ \Leb\ a.e.} \var(\tilde{f})+\Leb(|f|) &\leq 
        2\inf_{\tilde{f}=f\ \Leb\ a.e.} \var(\tilde{f})+\Lebessinf |f|
        \\
        &= 2\inf_{\tilde{f}=f\ \nu\ a.e.} \var(\tilde{f})+\nuessinf |f|
        \\
        &\leq 2\inf_{\tilde{f}=f\ \nu\ a.e.} \var(\tilde{f})+\nu(|f|)
        \leq 2\|f\|_{\BV_\nu}.
    \end{align*}

    Similarly, using \eqref{equiv cl var eq} and \eqref{nu Leb ineq 2} we have $\|f\|_{\BV_\nu}\leq 2 \|f\|_{\BV_1}$, and thus the proof is complete. 

\end{proof}
Proposition \ref{prop norm equiv} will be used later to provide (non-random) equivalence of $\|\spot\|_{\BV_{\nu_{\omega,0}}}$ and $\|\spot\|_{\BV_1}$ for each $\omega\in\Omega$.
We set $J_\om:=|T_\om'|$ and define the random Perron--Frobenius operator, acting on functions in $BV$
$$
P_\om(f)(x):=\sum_{y\in T_\om^{-1}(x)}\frac{f(y)}{J_\om(y)}.
$$
The operator $P$ satisfies the well-known property that 
\begin{align}\label{PF prop}
	\int_{[0,1]}P_\om(f)\,d\Leb=\int_{[0,1]}f\,d\Leb
\end{align}
for $m$-a.e. $\om\in\Om$ and all $f\in\BV$.
Recall from Section~\ref{sec:ROS} that $g_0=\set{g_{\om,0}}_{\om\in\Om}$ and that
$$\cL_{\om,0}(f)(x):=\sum_{y\in T_\om^{-1}(x)}g_{\om,0}(y)f(y),\quad f\in\BV.$$
We assume that the weight function $g_{\om,0}$ lies in $\BV$ for each $\om\in\Om$ and satisfies
\begin{equation}
	\label{E2}\tag{E2}
	\esssup_\omega \| g_{\omega,0}\|_{\infty,1}<\infty,
\end{equation}
and 
\begin{equation}
	\label{E3}\tag{E3}
	\essinf_\omega \inf g_{\omega,0}>0.
\end{equation}
Note that \eqref{E1} and \eqref{E2} together imply 
\begin{align}\label{fin sup L1}
	\esssup_\om\norm{\cL_{\om,0}\ind}_{\infty,1}\leq \esssup_\om D(T_\om)\norm{g_{\om,0}}_{\infty,1}<\infty
\end{align}
and 
\begin{align}\label{fin gJ}
	\esssup_\om\norm{g_{\om,0}J_\om}_{\infty,1}<\infty.
\end{align}
We also assume a uniform covering condition\footnote{We could replace the covering condition with the assumption of a strongly contracting potential. See \cite{AFGTV-IVC} for details.}
:
\begin{enumerate}[align=left,leftmargin=*,labelsep=\parindent]
	\item[\mylabel{E4}{E4}]
	For every subinterval $J\subset [0,1]$ there is a $k=k(J)$ such that for a.e.\ $\omega$ one has $T^k_\omega(J)=[0,1]$.
\end{enumerate}
Concerning the open system we assume that the holes $H_{\om,\ep}\sub [0,1]$ are chosen so that assumption
\eqref{A} 
holds. We also assume for each $\om\in\Om$ and each $\ep>0$ that $H_{\om,\ep}$ is composed of a finite union of intervals such that 
\begin{enumerate}[align=left,leftmargin=*,labelsep=\parindent]
	\item[\mylabel{E5}{E5}]
	There is a uniform-in-$\epsilon$ and uniform-in-$\omega$ upper bound on the number of connected components of $H_{\om,\ep}$,
\end{enumerate}
and
\begin{align}\label{E6}\tag{E6}
	\lim_{\ep\to 0}\esssup_\om \Leb(H_{\om,\ep})=0,
\end{align}
and 
    
\begin{enumerate}[align=left,leftmargin=*,labelsep=\parindent]
	\item[\mylabel{EX}{EX}]
	There exists an $\ep>0$ and an open neighborhood $\~H_{\om,\ep}\bus H_{\om,\ep}$ such that  
    $T_\om(U_\om)\bus \~H_{\sg\om,\ep}^c$, where $U_\om:=\cup_{Z\in\cZ_{\om,0}}\ol{Z}_{\ep}$ and $\ol{Z_{\ep}}$ denotes the closure of $Z_\ep\in A_{\om,\ep}:=\{Z\cap \~H_{\om,\ep}^c:Z\in\cZ_{\om,0}\}$ with $m(\set{\om\in\Om: \#A_{\om,\ep}\geq 2})>0$.
\end{enumerate}
\begin{remark}\label{rem check cond X}
    Assumption \eqref{EX} is satisfied for any random open system such that each map contains at least two intervals of monotonicity, the holes are contained in the interior of exactly one interval of monotonicity, and the image of the complement of the hole is the full interval, i.e. $T_\om(H_{\om,\ep}^c)=[0,1]$. In particular, \eqref{EX} is satisfied if there exists a full branch outside of the hole. 
\end{remark}

The following proposition ensures that condition \eqref{cond X} holds.
\begin{proposition}\label{prop check cond X}

The assumption \eqref{EX} implies \eqref{cond X}.
\end{proposition}
\begin{proof}
    In light of Remark \ref{rem check X}, to show that \eqref{cond X} is satisfied, it suffices to show that there is some $\ep>0$ such that  $X_{\om,\infty,\ep}\neq\emptyset$ for $m$-a.e. $\om\in\Om$.
    
    Let $\~T_{\om,Z}$ denote the continuous extension of $T_\om$ onto $\ol{Z}_{\ep}$ for each $Z_\ep\in A_{\om,\ep}$, and let $\~X_{\om,\infty,\ep}$ denote the survivor set for the open system consisting of the maps $\~T_\om$ and holes $\~H_{\om,\ep}$. By Proposition~\ref{prop surv nonemp} (taking $V_\om=[0,1]\bs\~H_{\om,\ep}$ and $U_{\om,j}=\ol{Z}_{\ep}$ for each $1\leq j\leq \#A_{\om,\ep}$) we see that $\~X_{\om,\infty,\ep}$ is uncountable. 
    Let 
    $$
        \cD_\om:=\bigcup_{j\geq 0}\~T_\om^{-j}(\cup_{Z_\ep\in A_{\sg^j\om,\ep}}\ol{Z}_{\ep}\bs Z_{\ep}).
    $$
    Since the survivor set for the original (unmodified) open system 
    $X_{\om,\infty,\ep}\sub \~X_{\om,\infty,\ep}\bs\cD_\om$,
    and since $\cD_\om$ is at most countable, we must in fact have that $X_{\om,\infty,\ep}\neq\emptyset$, thus satisfying \eqref{cond X}.
\end{proof}

Further, we suppose that for $m$-a.e. $\om\in\Om$ and all $\ep>0$ sufficiently small 
\begin{align}\label{E7}\tag{E7}
	T_\om(\cJ_{\om,\ep})=[0,1]
\end{align}
and there exists $n'\geq1$ and $\epsilon_0>0$ such that
\footnote{Note that the $9$ appearing in \eqref{E8} is not optimal. See \cite{AFGTV20,AFGTV21} for how this assumption may be improved. }
\begin{equation}
	\label{E8}\tag{E8}
	9\cdot\esssup_\om\|g_{\omega,0}^{(n')}\|_{\infty,1}
	<
	\essinf_\om\inf_{0\leq\ep\le \ep_0}\inf\cL_{\om,\ep}^{n'}\ind,
\end{equation}
where 
$$
    \cL_{\om,\ep}(f)(x):=\cL_{\om,0}(\ind_{H_{\om,\ep}^c}f)(x)
    =
    \sum_{y\in T_{\om}^{-1}(x)}g_{\om,\ep}(y)f(y)
    ,\quad f\in\BV
$$
and $g_{\om,\ep}:=g_{\om,0}\ind_{H_{\om,\ep}^c}$ as in Section~\ref{sec: open systems}.

Note that \eqref{E7} and \eqref{E3} together imply that $\cL_{\om,\ep}\ind(x)>0$       
for all $x\in [0,1]$:
\begin{align}\label{uniflbLeps*}
	\essinf_\om\inf_{\ep\le \ep_0}\inf\cL_{\om,\ep}\ind
	>\essinf_\om\inf g_{\om,0}>0,
\end{align}
and since $\norm{g_{\om,\ep}}_{\infty,1}\leq \norm{g_{\om,0}}_{\infty,1}$ for all $\ep>0$, \eqref{E8} is equivalent to the following 
\begin{align}\label{strong cont pot}
		9\cdot\sup_{0\leq \ep\leq \ep_0}\esssup_\om\|g_{\omega,\ep}^{(n')}\|_{\infty,1}
		<
		\essinf_\om\inf_{0\leq\ep\le \ep_0}\inf\cL_{\om,\ep}^{n'}\ind.
\end{align}

\begin{remark}
Note that the assumption \eqref{E7} is equivalent to there existing
 $N'\geq 1$ such that for $m$-a.e. $\om\in\Om$ and all $\ep\geq0$ sufficiently small 
	$$
	    T_\om^{N'}(X_{\om,N'-1,\ep})=[0,1].
	$$
        Indeed, since the surviving sets are forward invariant \eqref{surv set forw inv}, we have that 
        $T_\om^{N'-1}(X_{\om,N'-1,\ep})\sub X_{\sg^{N'-1}\om,0,\ep}=\cJ_{\sg^{N'-1}\om,\ep}$, and thus, 
        $$
            [0,1]=T_\om^{N'}(X_{\om,N'-1,\ep})\bus T_{\sg^{N'-1}\om}(\cJ_{\sg^{N'-1}\om,\ep}).
        $$
\end{remark}

For each $n\in\NN$ and $\om\in\Om$ we let $\sA_{\om,0}^{(n)}$ be the collection of all finite partitions of $[0,1]$ such that
\begin{align}\label{eq: def A partition}
	\var_{A_i}(g_{\om,0}^{(n)})\leq 2\norm{g_{\om,0}^{(n)}}_{\infty,1}
\end{align}
for each $\cA=\set{A_i}\in\sA_{\om,0}^{(n)}$.
Given $\cA\in\sA_{\om,0}^{(n)}$, let $\widehat\cZ_{\om,\ep}^{(n)}(\cA)$ be the coarsest partition amongst all those finer than $\cA$ and $\cZ_{\om,0}^{(n)}$ such that all elements of $\widehat\cZ_{\om,\ep}^{(n)}(\cA)$ are either disjoint from $X_{\om,n-1,\ep}$ or contained in $X_{\om,n-1,\ep}$. 
\begin{remark}
    Note that if $\var_Z(g_{\om,0})\leq 2\norm{g_{\om,0}^{(n)}}_{\infty,1}$ for each $Z\in\cZ_{\om,0}^{(n)}$ then we can take the partition $\cA=\cZ_{\om,0}^{(n)}$. Furthermore, the $2$ above can be replaced by some $\hat\al\geq 0$ (depending on $g_{\om,0}$) following the techniques of \cite{AFGTV20,AFGTV21}. 
\end{remark}
Define the subcollection 
\begin{align}\label{def of Z_*}
    \cZ_{\om,*,\ep}^{(n)}:=\{Z\in \widehat\cZ_{\om,\ep}^{(n)}(\cA): Z\sub X_{\om,n-1,\ep} \}.    
\end{align}
Recalling that $g_{\om,\ep}^{(n)}:=g_{\om,0}^{(n)}\ind_{X_{\om,n-1,\ep}}$, \eqref{eq: def A partition} implies that 
\begin{align}\label{eq: def A partition for g_ep}
	\var_{Z}(g_{\om,\ep}^{(n)})\leq 2\norm{g_{\om,0}^{(n)}}_{\infty,1}
\end{align}
for each $Z\in \cZ_{\om,*,\ep}^{(n)}$.    
We assume the following covering condition for the open system 
\begin{enumerate}[align=left,leftmargin=*,labelsep=\parindent]
	\item[\mylabel{E9}{E9}]
	There exists $k_o(n')\in\NN$ such that for $m$-a.e. $\om\in\Om$, all $\ep>0$ sufficiently small, and all $Z\in\cZ_{\om,*,\ep}^{(n')}$ we have $T_\om^{k_o(n')}(Z)=[0,1]$, where $n'$ is the number coming from \eqref{E8}.
\end{enumerate}
\begin{remark}
	Note that the uniform open covering time assumption \eqref{E9} clearly holds if \eqref{E4} holds and if there are only finitely many maps $T_\om$. In Remark~\ref{Alt E9 Remark} we present an alternative assumption to \eqref{E9}.
\end{remark}



The following lemma extends several results in \cite{DFGTV18a} from the specific weight $g_{\omega,0}=1/|T'_\omega|$ to general weights satisfying the conditions just outlined.
\begin{lemma}
	\label{DFGTV18alemma}
	Assume that a family of random piecewise-monotonic interval maps $\{T_{\omega}\}$ satisfies (\ref{E2}), (\ref{E3}), and (\ref{E4}), as well as (\ref{E8}) and \eqref{M} for $\epsilon=0$.
	Then 
    \eqref{C1'} and	the $\epsilon=0$ parts of \eqref{C2}, \eqref{C3}, \eqref{C4'}, \eqref{C5'}, and  (\ref{C7'}) as well as \eqref{CCM} hold. 
	Further, $\nu_{\om,0}$ is fully supported, condition \eqref{C4'} holds with $C_f=K$, for some $K<\infty$, and with $\alpha(N)=\gamma^N$ for some $\gamma<1$.
\end{lemma}
\begin{proof}
	See Appendix \ref{appB}.
\end{proof}

In what follows we consider transfer operators acting on the Banach spaces $\mathcal{B}_\omega=\BV_{\nu_{\om,0}}$ for a.e. $\om\in\Om$.
The norm we will use is $\|\cdot\|_{\mathcal{B}_\omega}=\|\cdot\|_\BV:=\var(\cdot)+\nu_{\omega,0}(|\cdot|)$.
As $\nu_{\om,0}$ is fully supported and non-atomic (Lemma \ref{DFGTV18alemma}), Proposition \ref{prop norm equiv} implies that 
\begin{equation}
	\label{normequiv}
	(1/2) \|f\|_{\BV_1}\le \|f\|_{\mathcal{B}_\omega}\le 2\|f\|_{\BV_1}
\end{equation}
for $m$-a.e. $\om\in\Om$ and $f\in\BV$.
Furthermore, applying \eqref{normequiv} twice, we see that 
\begin{equation}
	\label{normequiv2}
	(1/4) \|f\|_{\cB_\om}\le \|f\|_{\mathcal{B}_{\sg^n\omega}}\le 4\|f\|_{\cB_\om}
\end{equation}
for $m$-a.e. $\om\in\Om$ and all $n\in\ZZ$.
It follows from the proof of Proposition \ref{prop norm equiv} that 
\begin{equation}
	\label{sup norm equal}
	 \|f\|_{\infty,\om} = \|f\|_{\infty,1}
\end{equation}
for all $f\in\BV$ and $m$-a.e. $\om\in\Om$, where $\|\spot\|_{\infty,\om}$ denotes the supremum norm with respect to $\nu_{\om,0}$.
From \eqref{normequiv} we see that \eqref{B} is clearly satisfied.

From Lemma~\ref{DFGTV18alemma} we have that $\lm_{\om,0}:=\nu_{\sg\om,0}(\cL_{\om,0}\ind)$ and thus we may update \eqref{uniflbLeps*} to get 
\begin{align}\label{uniflbLeps}
	\essinf_\om\lm_{\om,0}^{n'}
	\geq \essinf_\om\inf_{\ep\le \ep_0}\inf\cL_{\om,\ep}^{n'}\ind
	\geq \essinf_\om\inf g_{\om,0}^{(n')}>0.
\end{align}
Note that since the conditions \eqref{B}, \eqref{cond X}, and \eqref{CCM} have been verified and we have assumed \eqref{A} and \eqref{M}, 
we see that $(\Om, m, \sg, [0,1], T, \BV, \cL_0, \phi_0, H_\ep)$ forms a random open system as defined in Section~\ref{sec: open systems} for all $\ep>0$ sufficiently small.  
We now use hyperbolicity of the $\epsilon=0$ transfer operator cocycle to guarantee that we have hyperbolic cocycles for small $\epsilon>0$, which will yield \eqref{C2}, \eqref{C3}, \eqref{C4'}, \eqref{C5'}, \eqref{C6}, and \eqref{C7'} for small positive $\epsilon$.

\begin{lemma}
	\label{harrylemma2}
	Assume that the conditions (\ref{E1})--(\ref{E9}) hold 
	for the random open system 
	$(\mathlist{\bcomma}{\Om, m, \sg, [0,1], T, \BV_1, \cL_0, \nu_0, \phi_0, H_\ep})$.
	Then for sufficiently small $\epsilon>0$, conditions \eqref{C2}, \eqref{C3}, \eqref{C4'}, \eqref{C5'}, \eqref{C6}, and \eqref{C7'} hold.
	Furthermore, the functionals $\nu_{\om,\ep}\in\BV_1^*$ can be identified with non-atomic Borel measures.
\end{lemma}
\begin{proof}
	For each $\om$ and $\ep>0$ we define $\hat\cL_{\om,\ep}:=\lm_{\om,0}^{-1}\cL_{\om,\ep}$;  note that $\hat\cL_{\om,0}=\~\cL_{\om,0}$.
	Our strategy is to apply Theorem 4.8 \cite{crimmins_stability_2019}, to conclude that for small $\epsilon$ the cocycles $\{\hat\cL_{\omega,\epsilon}\}$ are uniformly hyperbolic when considered as cocycles on the Banach space  $(\BV_1,\|\cdot\|_{\BV_1})$. 
	Because of (\ref{normequiv}), we will conclude the existence of a uniformly hyperbolic splitting in $\|\cdot\|_{\cB_\om}$ for a.e. $\om$.

	First, we note that Theorem 4.8 \cite{crimmins_stability_2019} assumes that the Banach space on which the transfer operator cocycle acts is separable.
	A careful check of the proof of Theorem 4.8 \cite{crimmins_stability_2019} shows that it holds  for the Banach space $(\BV_1,\|\cdot\|_{\BV_1})$ under the alternative condition \eqref{M}
	(see Appendix \ref{appA}).
	To apply Theorem A \cite{crimmins_stability_2019} we require, in our notation, that:
	\begin{enumerate}
		\item\label{harry assum 1 hat} $\hat\cL_{\omega,0}$ is a hyperbolic transfer operator cocycle on $\BV_1$ with norm $\|\cdot\|_{\BV_1}$ and a one-dimensional leading Oseledets space (see Definition 3.1 \cite{crimmins_stability_2019}), and slow and fast growth rates $0<\gamma<\Gamma$, respectively. We will construct $\gamma$ and $\Gamma$ shortly.
		\item\label{harry assum 2 hat} The family of cocycles $\{\hat\cL_{\omega,\epsilon}\}_{0\le \epsilon\le\epsilon_0}$ satisfy a uniform Lasota--Yorke inequality 
		\begin{equation*}
			\|\hat\cL^k_{\omega,\epsilon}f\|_{\BV_1}\le A\alpha^k\|f\|_{\BV_1}+B^k\|f\|_{1}
		\end{equation*}
		for a.e.\ $\omega$ and $0\le\epsilon\le\epsilon_0$, where $\alpha\le \gamma<\Gamma\le B$.
		\item\label{harry assum 3 hat} $\lim_{\epsilon\to 0}\esssup_{\omega}\trinorm{\hat\cL_{\omega,0}-\hat\cL_{\omega,\epsilon}}= 0$, where $\vertiii{\spot}$ is the $\BV-L^1(\Leb)$ triple norm.
	\end{enumerate}
	By Lemma \ref{DFGTV18alemma}
	we obtain a unique measurable family of equivariant functions $\{\varphi_{\omega,0}\}$ satisfying (\ref{C7'}) and (\ref{C5'}) for $\epsilon=0$.
	We have the equivariant splitting $\spn\{\varphi_{\omega,0}\}\oplus V_\om$, where $V_\om=\{f\in \BV_1: \nu_{\omega,0}(f)=0\}$.
	We claim that this splitting is hyperbolic in the sense of Definition 3.1 \cite{crimmins_stability_2019};  this will yield item \eqref{harry assum 1 hat} above.
	To show this, we verify conditions (H1)--(H3) in \cite{crimmins_stability_2019}.
	In our setting, Condition (H1) \cite{crimmins_stability_2019} requires the norm of the projection onto the top space spanned by $\varphi_{\omega,0}$, along the annihilator of $\nu_{\omega,0}$, to be uniformly bounded in $\omega$.
	This is true because this projection acting on $f\in \BV_1$ is $\nu_{\omega,0}(f)\varphi_{\omega,0}$ and therefore 
	$$
	\|\nu_{\omega,0}(f)\varphi_{\omega,0}\|_{\BV_1}\le\esssup_\omega\|\varphi_{\omega,0}\|_{\BV_1}\cdot \nu_{\omega,0}(f)\le 2\esssup_\omega\|\varphi_{\omega,0}\|_{\BV_1}\cdot\|f\|_{\BV_1},
	$$
	using (\ref{C5'}) and equivalence of $\|\cdot\|_{\BV_1}$ and $\|\cdot\|_{\cB_\om}$ \eqref{normequiv}.
	Next, we define 
	$$
	\alpha^{n'}:=\frac{9\esssup_\om \|g_{\om,0}^{( n')}\|_{\infty,1}}{\essinf_\om\inf_{\ep\geq 0} \inf\mathcal{L}^{ n'}_{\om,0}\ind}
	<1,
	$$ 
	which is possible by (\ref{E8}).
	Condition (H2) requires $\|\hat\cL^n_{\omega,0}\varphi_{\omega,0}\|_{\BV_1}\ge C\Gamma^n\|\varphi_{\omega,0}\|_{\BV_1}$ for some $C>0$, $\Gamma>0$, all $n$ and a.e.\ $\omega$.
	By (\ref{C7'}) one has 
	$$
	\|\hat\cL^n_{\omega,0}\varphi_{\omega,0}\|_{\BV_1}=\|\varphi_{\sigma^n\omega,0}\|_{\BV_1}\ge \essinf_\omega \inf\varphi_{\sigma^n\omega,0}>0,
	$$
	and thus we obtain (H2) with $C=\Gm=1$.
	Condition (H3) requires $\|\hat\cL^n_{\omega,0}|_{V_\om}\|_{\BV_1}\le  K\gamma^n$ for some $K<\infty$, $\al\leq \gamma<1$, all $n$ and a.e.\ $\omega$.
	This is provided by the  $\epsilon=0$ part of (\ref{C4'})---specifically the stronger exponential version guaranteed by Lemma \ref{DFGTV18alemma}---and the equivalence of $\|\cdot\|_{\BV_1}$ and $\|\cdot\|_{\cB_{\sg^n\om}}$.
	
	For item \eqref{harry assum 2 hat} we begin with the Lasota--Yorke inequality for $\var(\cdot)$ and $\nu_{\omega,0}(|\cdot|)$ provided by the final line of the proof of Lemma \ref{closed ly ineq App} (equation (\ref{ORLYLY})).
	Dividing through by $\lm_{\om,0}^{n'}$ we obtain 
	\begin{align}
		\var(\hat\cL_{\omega,\epsilon}^{n'} f)
		&\le 
		\frac{9\|g_{\omega,\epsilon}^{({n'})}\|_{\infty,1}}{\lm_{\om,0}^{n'}}\var(f)
		+\frac{8\|g_{\omega,\epsilon}^{(n')}\|_{\infty,1}}{\lm_{\om,0}^{n'}\min_{Z\in \mathcal{Z}_{\omega,*,\ep}^{(n')}(\mathcal{A})}\nu_{\omega,0}(Z)}\nu_{\omega,0}(|f|)
		\nonumber\\
		&\le 
		\alpha^{n'}\var(f)+\frac{\al^{n'}}{\min_{Z\in \mathcal{Z}_{\omega,*,\ep}^{(n')}(\mathcal{A})}\nu_{\omega,0}(Z)}\nu_{\omega,0}(|f|),
		\label{LYineqorly hat}
	\end{align}
	noting that $\essinf_\om\min_{Z\in\mathcal{Z}_{\omega,*,\ep}^{(n')}(\mathcal{A})}\nu_{\omega,0}(Z)>0$ since the uniform open covering assumption \eqref{E9} together with \eqref{E3} and \eqref{uniflbLeps} imply that 
	and equivariance of the backward adjoint cocycle together imply that 
	for $Z\in\cZ_{\om,*,\ep}^{(n')}$ we have
	\begin{align}\label{LY LB calc}
		\nu_{\om,0}(Z)
		=
		\nu_{\sg^{k_o(n')}\om,0}\left(\left(\lm_{\om,0}^{k_o(n')}\right)^{-1}\cL_{\om,0}^{k_o(n')}\ind_Z\right)
		\geq
		\frac{\inf g_{\om,0}^{k_o(n')}}{\lm_{\om,0}^{k_o(n')}}>0.
	\end{align}
	As the holes $H_{\om,\ep}$ are composed of finite unions of disjoint intervals, assumption \eqref{E6} implies that the radii of each of these intervals must go to zero as $\ep\to 0$. Thus, using \eqref{E6} together with the fact that $\nu_{\om,0}$ is fully supported and non-atomic,
	we see that \eqref{C6} must hold.
	
	We construct a uniform Lasota--Yorke inequality for all $n$ in the usual way by using blocks of length $jn'$;  we write this as
	\begin{equation}
		\label{LYfull hat}
		\var(\hat\cL_{\omega,\epsilon}^n f)\le A_1\alpha^n\var(f)+A_2^n\nu_{\omega,0}(|f|)
	\end{equation}
	for some $A_2>\alpha$.
	We now wish to convert this to an inequality
	\begin{equation}
		\label{targetinequality hat}
		\var(\hat\cL_{\omega,\epsilon}^n f)\le A_1'\alpha^n\var(f)+(A_2')^n\|f\|_{1}.
	\end{equation}
	In light of \eqref{E2}, \eqref{uniflbLeps}, \eqref{E1}, and \eqref{PF prop}, we see that there is a constant $B$ so that for $m$-a.e.\ $\omega\in\Om$ we have
	\begin{align*}
		\norm{\hat\cL_{\om,\ep}^{n'}f}_1
		&=
		\left(\lm_{\om,0}^{n'}\right)^{-1}\int_{[0,1]} \left|\sum_{y\in T_\om^{-n'}x} g_{\om,0}^{(n')}(y) \hat X_{\om,n'-1,\ep}(y)f(y)\right| d\Leb(x)
		\\
		&\leq
		\frac{\norm{g_{\om,0}^{(n')} J_\om^{(n')}}_{\infty,1}}{\lm_{\om,0}^{n'}} \int_{[0,1]} \left|\sum_{y\in T_\om^{-n'}(x)} \frac{f(y)}{J_\om^{(n')}(y)} \right| d\Leb(x)
		\\
		&=
		\frac{\norm{g_{\om,0}^{(n')} J_\om^{(n')}}_{\infty,1}}{\lm_{\om,0}^{n'}} \int_{[0,1]} \left| P_\om^{n'}(f) \right| d\Leb(x)
		\\
		&\leq
		\frac{\norm{g_{\om,0}^{(n')} J_\om^{(n')}}_{\infty,1}}{\lm_{\om,0}^{n'}}\norm{f}_1
		\leq B^{n'}\norm{f}_1.
	\end{align*}
	Using the non-atomicicty of $\nu_{\om,0}$ from \eqref{CCM} (shown in \eqref{DFGTV18alemma}) and the fact that $\var(|f|)\le \var(f)$, we may apply Lemma 5.2 \cite{BFGTM14} to $\nu_{\omega,0}$ to obtain that for each $\zeta>0$, there is a $B_\zeta<\infty$ such that $\nu_{\omega,0}(|f|)\le \zeta\var(f)+B_\zeta\|f\|_1$.
	Now using (\ref{LYfull hat}) we see that
	\begin{align*}
		\|\hat{\mathcal{L}}_{\omega,\epsilon}^nf\|_{\BV_1}
		&=\var(\hat{\mathcal{L}}_{\omega,\epsilon}^nf)+\|\hat{\mathcal{L}}_{\omega,\epsilon}^n(f)\|_1
		\\
		&\le A_1\alpha^n\var(f)+A_2^n\nu_{\omega,0}(|f|)+B^n\|f\|_1
		\\
		&\le(A_1\alpha^n+A_2^n\zeta)\var(f)+(B^n+A_2^nB_\zeta)\|f\|_1
		\\
		&\le(A_1\alpha^n+A_2^n\zeta)\|f\|_{\BV_1}+(B^n+A_2^n(B_\zeta-\zeta)-A_1\alpha^n))\|f\|_1.
	\end{align*}
	Selecting $\zeta$ sufficiently small and $n''$ sufficiently large so that $C(\alpha^{n''}+\zeta)<1$ we again (by proceeding in blocks of $n''$) arrive at a uniform Lasota--Yorke inequality of the form (\ref{targetinequality hat}) for all $n\ge 0$.
	
	For item \eqref{harry assum 3 hat} we note that 
	\begin{align*}
		\vertiii{\hat\cL_{\omega,0}-\hat\cL_{\omega,\epsilon}}
		&:=\sup_{\|f\|_{\BV_1}=1}\|(\hat\cL_{\omega,0}-\hat\cL_{\omega,\epsilon})f\|_1
		=\sup_{\|f\|_{\BV_1}=1}\|\hat\cL_{\omega,0}(f\ind_{H_{\om,\ep}})\|_1
		\le 
		\|\hat\cL_{\omega,0}(\ind_{H_{\om,\ep}})\|_1
		\\
		&=
		\lm_{\om,0}^{-1}\int_{[0,1]} \left|\sum_{y\in T_\om^{-1}(x)} g_{\om,0}(y) \ind_{H_{\om,\ep}}(y)\right| d\Leb(x)
		\\
		&\leq
		\frac{\norm{g_{\om,0} J_\om}_{\infty,1}}{\lm_{\om,0}} \int_{[0,1]} \left|\sum_{y\in T_\om^{-1}(x)} \frac{\ind_{H_{\om,\ep}}(y)}{J_\om(y)} \right| d\Leb(x)
		\\
		&=
		\frac{\norm{g_{\om,0} J_\om}_{\infty,1}}{\lm_{\om,0}}\int_{[0,1]} \left| P_\om(\ind_{H_{\om,\ep}}) \right| d\Leb(x)
		\\
		&\le 
		\esssup_\omega \frac{\norm{g_{\om,0} J_\om}_{\infty,1}}{\lm_{\om,0}}\cdot\esssup_\omega\mathrm{Leb}(H_{\omega,\epsilon}).
	\end{align*}
	Because \eqref{E2}, \eqref{E1} and \eqref{uniflbLeps} imply $\esssup_\omega\sfrac{\norm{g_{\om,0} J_\om}_{\infty,1}}{\lm_{\om,0}}<\infty$, and since \eqref{E6} implies that $\lim_{\epsilon\to 0}\esssup_\omega\mathrm{Leb}(H_{\omega,\epsilon})=0$, we obtain item \eqref{harry assum 3 hat}.
	
	We may now  apply Theorem 4.8 \cite{crimmins_stability_2019} to conclude that given $\delta>0$ there is an $\epsilon_0>0$ such that for all $\epsilon\le \epsilon_0$ the cocycle generated by $\hat\cL_\epsilon$ is hyperbolic, with 
	\begin{enumerate}[i]
		\item[\mylabel{i}{h1}] the existence of an equivariant family $\hat\phi_{\om,\ep}\in\BV$ with $\esssup_\omega\|\hat\phi_{\om,\ep}-\varphi_{\om,0}\|_{\BV_1}<\delta$,
		\item[\mylabel{ii}{h2}] existence of corresponding Lyapunov multipliers $\hat\lambda_{\om,\ep}$ satisfying $|\hat\lambda_{\om,\ep}-1|<\delta$,
		\item[\mylabel{iii}{h3}] operators $\hat Q_{\om,\ep}$ satisfying $\|(\hat Q_{\om,\ep})^n\|_{\BV_1}\le K'(\gamma+\delta)^n$, where $\gamma$ is the decay rate for $Q_{\om,0}$ from the proof of Lemma \ref{DFGTV18alemma}.
	\end{enumerate}
	To obtain an $\hat\cL_{\om,\ep}^*$-equivariant family of linear functionals $\hat\nu_{\omega,\epsilon}\in\BV_1^*$  we apply Corollary 2.5 \cite{DFGTV18b}.
	Using the one-dimensionality of the leading Oseledets space for the forward cocycle, this result shows
	that the leading Oseledets space for the backward adjoint cocycle  is also one-dimensional.
	This leading Oseledets space is spanned by some $\hat\nu_{\omega,\epsilon}\in \BV_1^*$, satisfying $\hat\nu_{\sigma\om,\ep}(\hat\cL_{\om,\ep}(f))=\hat\vta_{\om,\ep}\hat\nu_{\om,\ep}(f)$, for Lyapunov multipliers $\hat\vta_{\om,\ep}$.
	By  Lemma 2.6 \cite{DFGTV18b}, we may scale the $\hat\nu_{\om,\ep}$ so that $\hat\nu_{\om,\ep}(\hat\phi_{\om,\ep})=1$ for a.e.\ $\omega$.
	We show that in fact $\hat\vta_{\om,\ep}=\hat\lambda_{\om,\ep}$ $m$-a.e.
	Indeed, 
	$$
	1=\hat\nu_{\sg\om,\ep}(\hat\varphi_{\sg\om,\ep})=\hat\nu_{\sg\om,\ep}(\hat\cL_{\om,\ep}\hat\phi_{\om,\ep}/\hat\lambda_{\om,\ep})=(\hat\vta_{\om,\ep}/\hat\lambda_{\om,\ep})\hat\nu_{\om,\ep}(\hat\phi_{\om,\ep})=\hat\vta_{\om,\ep}/\hat\lambda_{\om,\ep}.
	$$
	Note that $\cL_{\om,\ep} = \lm_{\om,0}\hat\cL_{\om,\ep}$, and we now define $\varphi_{\om,\ep}, \lambda_{\om,\ep}, Q_{\om,\ep}$, and $\nu_{\om,\ep}$ by the following:
	\begin{align*}
		\phi_{\om,\ep}&:=\frac{1}{\nu_{\om,0}(\hat\phi_{\om,\ep})}\cdot \hat\phi_{\om,\ep},
		&
		\nu_{\om,\ep}(f)&:=\nu_{\om,0}(\hat\phi_{\om,\ep})\hat\nu_{\om,\ep}(f),
		\\
		\lm_{\om,\ep}&:=\lm_{\om,0}\frac{\nu_{\sg\om,0}(\hat\phi_{\sg\om,\ep})}{\nu_{\om,0}(\hat\phi_{\om,\ep})}\hat\lm_{\om,\ep},
		&
		Q_{\om,\ep}(f)&:=\frac{\nu_{\om,0}(\hat\phi_{\om,\ep})}{\nu_{\sg\om,0}(\hat\phi_{\sg\om,\ep})}\hat Q_{\om,\ep}(f)
	\end{align*}
	Clearly all of the properties of \eqref{C2} and \eqref{C3} are now satisfied except for the log-integrability of $\lambda_{\om,\ep}$ in \eqref{C2}.
	To demonstrate this last point, we note that by uniform hyperbolicity of the perturbed cocycles, $\lambda'_{\om,\ep}$ are uniformly bounded below and are therefore log-integrable.
	Since 
	\begin{align*}
		|1-\nu_{\omega,0}(\hat\varphi_{\omega,\epsilon})|=|\nu_{\omega,0}(\varphi_{\omega,0})-\nu_{\omega,0}(\hat\varphi_{\omega,\epsilon})|\le \|\hat\phi_{\om,\ep}-\varphi_{\om,0}\|_{\BV_1}<\delta,
	\end{align*}
	the $\lambda_{\omega,\epsilon}$ are uniformly small perturbations of the $\hat\lambda_{\omega,\epsilon}$, and since $\lm_{\om,0}$ is $\log$-integrable by \eqref{E2} and \eqref{uniflbLeps}, we must therefore have that the log integrability condition on $\lambda_{\om,\ep}$ in \eqref{C2} is satisfied. 
	Point \eqref{h1} above, combined with the $\ep=0$ part of \eqref{C5'} (resp. \eqref{C7'}) and the uniform estimate for $|1-\nu_{\om,0}(\hat\phi_{\om,\ep})|$, immediately yields the $\ep>0$ part of \eqref{C5'} (resp. \eqref{C7'}).
	Point \eqref{h3} above combined with the same estimates also ensures that the norm of $\|Q_{\om,\ep}^n\|_{\infty,1}$ decays exponentially fast, uniformly in $\om$ and $\ep$, satisfying the stronger exponential version of \eqref{C4'}. In fact point \eqref{h3} implies the stronger statement that $\|Q_{\om,\ep}^n\|_{\BV_1}$ decays exponentially fast, uniformly in $\om$ and $\ep$.

	Finally we show that $\nu_{\om,\ep}:C^0([0,1])\to\mathbb{C}$ is a positive linear functional with $\nu_{\om,\ep}(f)\in \mathbb{R}$ if $f$ is real.
	From this fact it will follow by Riesz-Markov (e.g.\ Theorem A.3.11 \cite{viana_oliveira}) that $\nu_{\om,\ep}$ can be identified with a real finite Borel measure on $[0,1]$.
	By linearity we may consider the two cases: (i) $f=\varphi_{\om,\ep}>0$ (the generator of the leading Oseledets space) and (ii) $f\in F_{\omega,\ep}$, where $F_{\omega,\ep}$ is the Oseledets space complementary to $\mathrm{span\{\varphi_{\om,\ep}\}}$.
	In case (i)  $\nu_{\om,\ep}(\varphi_{\om,\ep})=1>0$.
	In case (ii), Lemma 2.6 \cite{DFGTV18b} implies $\nu_{\om,\ep}(f)=0$. 
	In summary we see that $\nu_{\om,\ep}$ is positive.
\end{proof}

\begin{remark}
	As we have just shown that assumptions \eqref{C1'}, \eqref{C2}, \eqref{C3}, \eqref{C4'}, \eqref{C5'}, \eqref{C6}, \eqref{C7'} (Lemmas~\ref{DFGTV18alemma} and \ref{harrylemma2}), we see that Proposition~\ref{prop: escape rates} holds as well as Theorem~\ref{thm: dynamics perturb thm} for the random open system $(\Om, m, \sg, [0,1], T, \BV, \cL_0, \nu_0, \phi_0, H_\ep)$ under the additional assumption of \eqref{C8}.
	In light of Remark~\ref{rem checking esc cor cond}, if we assume \eqref{xibound} in addition to \eqref{C8} then both Corollary~\ref{esc rat cor} and Theorem~\ref{evtthm} apply. 
\end{remark}

The following theorem is the main result of this section and elaborates on the dynamical significance of the perturbed objects produced in Lemma~\ref{harrylemma2}.
\begin{theorem}\label{existence theorem}
	Suppose $(\Om, m, \sg, [0,1], T, \cB, \cL_0, \nu_0, \phi_0, H_\ep)$ is a random open system and that the assumptions of Lemma~\ref{harrylemma2} hold. 
	Then there exists $\ep_0>0$ sufficiently small such that for every $0\leq\ep<\ep_0$ we have the following:
	\begin{enumerate}
		\item\label{item 1}
		There exists a unique random probability measure $\zt_\ep=\set{\zt_{\om,\ep}}_{\om\in\Om}$ on $[0,1]$ such that, for $\ep>0$, $\zt_{\om,\ep}$ is supported in $X_{\om,\infty,\ep}$ and 
		\begin{align*}
			\zt_{\sg\om,\ep}(\cL_{\om,\ep} f)=\rho_{\om,\ep}\zt_{\om,\ep}(f),
		\end{align*}
		for $m$-a.e. $\om\in\Om$ and each $f\in\BV$, where 
		\begin{align*}
			\rho_{\om,\ep}:=\zt_{\sg\om,\ep}(\cL_{\om,\ep}\ind).
		\end{align*} 	
		Furthermore, for $\ep=0$ we have $\zt_{\om,0}=\nu_{\om,0}$ and $\rho_{\om,0} = \lm_{\om,0}$ and for $\ep>0$ we have that $C_1^{-1}\leq \rho_{\om,\ep}\leq C_1$ for $m$-a.e. $\om\in\Om$. 
		
		\
		
		\item\label{item 2}
		There exists a measurable function $h_\ep:\Om\times[0,1]\to(0,\infty)$ 
		such that $\zt_{\om,\ep}(h_{\om,\ep})=1$ and  
		\begin{align*}
			\cL_{\om,\ep} h_{\om,\ep}=\rho_{\om,\ep} h_{\sg\om,\ep}
		\end{align*}
		for $m$-a.e. $\om\in\Om$.
		Moreover, $h_\ep$ is unique modulo $\zt_\ep$, and there exists $C\geq 1$ such that $C^{-1}\leq h_{\om,\ep}\leq C$ for $m$-a.e. $\om\in\Om$. Furthermore, for $m$-a.e. $\om\in\Om$ we have that $\rho_{\om,\ep}\to\rho_{\om,0}=\lm_{\om,0}$ and 
		$h_{\om,\ep}\to h_{\om,0}=\phi_{\om,0}$ (in $\cB_\om$) as $\ep\to 0$, where $\phi_{\om,0}$ and $\lm_{\om,0}$ are defined in Lemma~\ref{DFGTV18alemma}.
		
		\
		
		\item\label{item 3} 
		The random measure $\mu_\ep=\set{\mu_{\om,\ep}:=h_{\om,\ep}\zt_{\om,\ep}}_{\om\in\Om}$ is a $T$-invariant and ergodic random probability measure whose fiberwise support, for $\ep>0$, is contained in $X_{\om,\infty,\ep}$. 
		Furthermore, $\mu_\ep$ is the unique relative equilibrium state, i.e. 
		\begin{align*}
			\int_\Om\log\rho_{\om,\ep}\ dm(\om)&=:\cEP_\ep(\log g_0)
			=
			h_{\mu_\ep}(T)+\int_{\cJ_0} \log g_0 \,d{\mu_\ep}
			=
			\sup_{\eta_\ep\in\cP_{T,m}^{H_\ep}(\cJ_0)} h_{\eta_\ep}(T)+\int_{\cJ_0} \log g_0 \,d\eta_\ep,
		\end{align*}
		where $h_{\eta_\ep}(T)$ is the entropy of the measure $\eta_\ep$, $\cEP_\ep(\log g_0)$ is the expected pressure of the weight function $g_0=\set{g_{\om,0}}_{\om\in\Om}$, 
		and $\cP_{T,m}^{H_\ep}(\cJ_0)$ denotes the collection of $T$-invariant random probability measures $\eta_\ep$ on $\cJ_0$ whose disintegration $\{\eta_{\om,\ep}\}_{\om\in\Om}$ satisfies $\eta_{\om,\ep}(H_{\om,\ep})=0$ for $m$-a.e. $\om\in\Om$. Furthermore, $\lim_{\ep\to 0}\cEP_\ep(\log g_0)=\cEP_0(\log g_0):=\int_\Om\log \lm_{\om,0}\,dm.$
		
		\
		
		\item\label{item 4}
		For $\ep>0$, let $\vrho_\ep=\set{\varrho_{\om,\ep}}_{\om\in\Om}$ be the random probability measure with fiberwise support in $[0,1]\bs H_{\om,\ep}$ whose disintegrations are given by
		\begin{align*}
			\varrho_{\om,\ep}(f):=\frac{\nu_{\om,0}\lt(\ind_{H_{\om,\ep}^c} h_{\om,\ep} f\rt)}{\nu_{\om,0}\lt(\ind_{H_{\om,\ep}^c} h_{\om,\ep}\rt)}
		\end{align*} 
		for all $f\in\BV$.  $\vrho_{\om,\ep}$ is the unique random conditionally invariant probability measure that is absolutely continuous (with respect to $\set{\nu_{\om,0}}_{\om\in\Om}$) with density of bounded variation. 

		\
		
		\item\label{item 5}
		For each $f\in\BV$ there exists $D>0$ and $\kp_\ep\in(0,1)$ such that for $m$-a.e. $\om\in\Om$ and all $n\in\NN$ we have
		\begin{align*}
			\norm{\lt(\rho_{\om,\ep}^n\rt)^{-1}\cL_{\om,\ep}^n f - \zt_{\om,\ep}(f)h_{\sg^n\om,\ep}}_{\cB_{\sg^n\om}}\leq D\norm{f}_{\cB_\om}\kp_\ep^n.
		\end{align*}
		Furthermore, for all $A\in\sB$ and $f\in\BV$ we have 
		\begin{align*}
			\absval{
				\nu_{\om,0}\lt(T_\om^{-n}(A)\,\rvert\, X_{\om,n,\ep}\rt)
				-
				\varrho_{\sg^n\om,\ep}(A)	
			}
			\leq 
			D\kp_\ep^n,
		\end{align*}
		and 
		\begin{align*}
			\absval{
				\frac{\varrho_{\om,\ep}\lt(f\rvert_{X_{\om,n,\ep}}\rt)}
				{\varrho_{\om,\ep}(X_{\om,n,\ep})}
				-
				\mu_{\om,\ep}(f)	
			}
			\leq 
			D\norm{f}_{\cB_\om}\kp_\ep^n	.		
		\end{align*}
		In addition, we have $\lim_{\ep\to 0}\kp_\ep=\kp_0$, where $\kp_0$ is defined in Lemma~\ref{DFGTV18alemma}.
		
		\
		
		\item\label{item 6}
		There exists $C>0$ such that for every $f,h\in \BV$,
	  every $n\in\NN$ sufficiently large, and for $m$-a.e. $\om\in\Om$ we have 
		\begin{align*}
			\absval{
				\mu_{\om,\ep}
				\lt(\lt(f\circ T_{\om}^n\rt)h \rt)
				-
				\mu_{\sg^{n}\om,\ep}(f)\mu_{\om,\ep}(h)
			}
			\leq C
			\norm{f}_{\infty,\om}\norm{h}_{\cB_\om}\kp_\ep^n.
		\end{align*} 
		
		
	\end{enumerate}
	
\end{theorem}
\begin{proof}

	First we note that the claims of items \eqref{item 1} -- \eqref{item 3} and \eqref{item 5} -- \eqref{item 6} above for $\ep=0$ follow immediately from Lemma~\ref{DFGTV18alemma}. Now we are left to prove each of the claims for $\ep>0$.

	Claims \eqref{item 1} -- \eqref{item 3} follow from Lemma~\ref{harrylemma2} with the scaling:
	\begin{align*}
		h_{\om,\ep}&:=\nu_{\om,\ep}(\phi_{\om,\ep}) \phi_{\om,\ep},
		&
		\zt_{\om,\ep}(f)&:=\frac{\nu_{\om,\ep}(f)}{\nu_{\om,\ep}(\ind)},
		&
		\rho_{\om,\ep}&:=\frac{\nu_{\om,\ep}(\ind)}{\nu_{\sg\om,\ep}(\ind)}\lm_{\om,\ep}.
	\end{align*}
	The uniform boundedness on $\rho_{\om,\ep}$ and $h_{\om,\ep}$ follows from the uniform boundedness on $\lm_{\om,\ep}$ and $\phi_{\om,\ep}$ coming from Lemma~\ref{harrylemma2}.
	The fact that $\mu_\ep$ is the unique relative equilibrium state follows similarly to the proof of Theorem 2.23 in \cite{AFGTV20} (see also Remark 2.24, Lemma 12.2 and Lemma 12.3).
	The claim that $\supp(\zt_{\om,\ep})\sub X_{\om,\infty,\ep}$ follows similarly to Lemma 11.11 of \cite{AFGTV21}. Noting that $\|f\|_{\infty,\om}=\|f\|_{\infty,1}$ for $f\in\BV$ by the proof of Proposition \ref{prop norm equiv}, we now proof Claims \eqref{item 4} -- \eqref{item 6}.
	
	Claim \eqref{item 4} follows from Lemma 3.5 \cite{AFGTV21}  and the uniqueness of the density $h_{\om,\ep}\in\BV$.
	
	The first item of Claim \eqref{item 5} follows from Lemma \ref{harrylemma2} and the remaining items are proved similarly to Corollary 12.8 \cite{AFGTV21}.
	
	Claim \eqref{item 6} 
	follows from Claim \eqref{item 5} and is proven in Appendix \ref{appDec}.
\end{proof}
\begin{remark}\label{rem seq exist}
If one considers a two-sided (bi-infinite) sequential analogue of random open systems, then because Theorem 4.8 \cite{crimmins_stability_2019} also applies to two-sided sequential systems,  one could prove similar results to items \eqref{item 1}, \eqref{item 2}, \eqref{item 4}, \eqref{item 5}, and \eqref{item 6} of Theorem \ref{existence theorem}.
\end{remark}

\section{Limit theorems}\label{sec: limit theorems}
 In this section we prove  a few  limit theorems   for the {\em closed} systems  discussed in Section~\ref{sec: existence} $(\Om, m, \sg, [0,1], T, \BV, \cL_0, \nu_0, \phi_0)$. We will in fact show that such systems   are  {\em admissible} in the sense of \cite{dragicevic_spectral_2018}. This will allow us to adapt  to our setting the spectral approach {\em \`a la Nagaev-Guivarc'h} developed in \cite{dragicevic_spectral_2018} and get a quenched  central limit theorem, a quenched large deviation theorem and a quenched local central limit theorem. We will also present an alternative approach based on martingale techniques \cite{DFGTV18a,atnipASIP,DHasip}, which produces an  almost sure invariance principle (ASIP) for the random measure $\mu_0=\set{\mu_{\om,0}}_{\om\in\Om}$. Moreover, the ASIP implies that  $\mu_{0}$ satisfies the central limit theorem as well as the law of the iterated logarithm. The martingale approach will also give an upper bound for any (large) deviation from the expected value and a Borel-Cantelli dynamical lemma. At the moment we could not extend the previous limit theorems to the {\em open} systems investigated in Section~\ref{sec: existence}. There are  a few  reasons for that which concern the Banach space  $\mathcal{B}_{\om,\ep}$ associated to those systems and  defined by the norm: $||\cdot||_{\mathcal{B}_{\om,\ep}}=\text{var}(\cdot)+\zeta_{\om,\ep}(|\cdot|).$ First of all,  we do not know if the  random cocycle $\mathcal{R}_{\eps}=(\Omega, m, \sigma, \mathcal{B}_{\om,\ep}, \tcL_{\om, \epsilon})$ is quasi-compact which is an essential requirement for admissibility. 
 Second, the results of Theorem \ref{existence theorem} are not particularly compatible with the Banach spaces  $\mathcal{B}_{\om,\ep}$, $\ep>0$, since the inequalities of items \eqref{item 5} and \eqref{item 6} are in terms of the norms $\|\spot\|_{\infty,\om}$ and $\|\spot\|_{\cB_\om}$ which are defined modulo $\nu_{\om,0}$ and the $\mathcal{B}_{\om,\ep}$ norm is defined via $\zt_{\om,\ep}$, a measure which is supported on a $\nu_{\om,0}$-null set.

\subsection {The Nagaev-Guivarc'h approach}\label{NG approach} The paper \cite{dragicevic_spectral_2018} developed a general scheme to adapt  the  Nagaev-Guivarc'h approach to random quenched dynamical systems,  allowing one to prove limit theorems by exploring the connection between a {\em twisted} operator cocycle and the distribution of the Birkhoff sums. The results in \cite{dragicevic_spectral_2018} were confined to the geometric potential $|\det(DT_{\om})|^{-1}$ and the associated conformal measure, Lebesgue measure. We now show how  to extend those results to the systems verifying the assumptions stated in Section \ref{sec: existence} and the results of Theorem \ref{existence theorem}, whenever $\epsilon=0,$ that is we will consider random closed systems  for a larger class of potentials. 

 The starting point is to replace the linear operator $\cL_{\om,0}$ associated to the geometric potential and the (conformal) Lebesgue measure introduced in \cite{dragicevic_spectral_2018}, with our operator $\cL_{\om, 0}$ and the associated conformal measures $\nu_{\om, 0}.$ In particular, if we work with the normalized operator $\tcL_{\om, 0} := \lm_{\om,0}^{-1}\cL_{\om,0}$ the results in \cite{dragicevic_spectral_2018} are reproducible almost verbatim with a few precautions which we are going to explain. 
 As before, let $\mathcal{B}_{\om}$ be the Banach space defined by $||\cdot||_{\mathcal{B}_{\om}}=\text{var}(\cdot)+\nu_{\om,0}(|\cdot|),$ where the variation is defined using equivalence classes mod-$\nu_{\om, 0}.$ 
 In order to apply the theory in \cite{dragicevic_spectral_2018} we must show that our random cocycle is {\em admissible}. This  reduces to check two sets of properties which were listed in  \cite{dragicevic_spectral_2018} respectively as conditions (V1) to (V9) and conditions  (C0) up to (C4).  The first set of conditions reproduces the classical properties of the total variation of a function and its relationship with the $L^1(\Leb)$ norm. We should emphasize that in our case the variation is defined using equivalence classes mod-$\nu_{\om, 0}.$ 
 It is easy to  check that properties (V1, V2, V3, V5, V8, V9) hold  with respect to this variation. In particular (V3) asserts that for any $f\in \mathcal{B}_{\om}$ we have $||f||_{L^{\infty}(\nu_{\om, 0})}\le ||f||_{\mathcal{B}_{\om}};$ we will refer to it in the following just as the (V3) property.
 {\em Notation}: recall that given an element $f\in \mathcal{B}_{\om}$, the $L^{\infty}$ norm of $f$ with respect to the measure $\nu_{\om, 0}$ is denoted by $||f||_{\infty,\om}$ in the rest of this section.
 
 Property (V7) is not used in the current paper; (V6) is a general density embedding result proved in  Hofbauer and Keller (Lemma 5 \cite{HHKK}).
 We elaborate on property (V4). To obtain (V4) in our situation one needs to prove that  the unit ball of $\mathcal{B}_{\om}$ is compactly injected  into $L^{1}(\nu_{\om, 0})$. As we will see, this is used  to get the quasi-compactness of the random cocycle. The result follows easily by adapting Proposition 2.3.4 in \cite{gora} to our conformal measure $\nu_{\om, 0},$ which is fully supported and non-atomic. 
 We now rename the other set of  properties (C0)--(C4) in \cite{dragicevic_spectral_2018} as  ($\mathcal{C}$0) to ($\mathcal{C}$4) to distinguish them from our (C) properties stated earlier in Sections \ref{sec:goodrandom} and \ref{EEVV}.  \begin{itemize}
    \item 
 Assumption ($\mathcal{C}$0) coincides with our condition \eqref{M}.
\item 
Condition ($\mathcal{C}$1)  requires us to prove in our case that \begin{equation}\label{C1CMP}
||\tilde{\mathcal{L}}_{\om,0}f||_{\mathcal{B}_{\sigma\om}}\le K ||f||_{\mathcal{B}_{\om}},
\end{equation}
for every $f\in\mathcal{B}_{\om}$ and for $m$-a.e. $\omega$, with $\om$-independent $K$.
\item 
Condition ($\mathcal{C}$2) asks that there exists $N\in \mathbb{N}$ and measurable $\tilde{\alpha}^N, \tilde{\beta}^N:\Omega\rightarrow (0, \infty),$ with $\int_{\Omega}\log\tilde{\alpha}^N(\om)d m(\om)<0,$ such that for every $f\in \mathcal{B}_{\om}$ and $m$-a.e. $\om\in \Omega,$
\begin{equation}\label{C2CMP}
    ||\tilde{\mathcal{L}}^N_{\om,0}f||_{\mathcal{B}_{\sigma^N\om}}\le \tilde{\alpha}^N(\om)||f||_{\mathcal{B}_{\om}}+\tilde{\beta}^N(\om)||f||_{L^1(\nu_{\om, 0})}.
\end{equation}
\item Condition ($\mathcal{C}$3)  is  the content of the first display equation in item  \ref{item 5} in the statement of Theorem \ref{existence theorem}. 
\item Condition ($\mathcal{C}$4) is only used to obtain Lemma 2.11 in \cite{dragicevic_spectral_2018}. There are three results in that Lemma which we now compare with our situation. The third result is the decay of correlations stated in item \ref{item 6} of Theorem \ref{existence theorem}. The second result  is the almost-sure strictly positive lower bound for the density $\phi_{\om, 0}$ stated in item \ref{item 2} of Theorem \ref{existence theorem}. The first result requires that $\esssup_{\om\in \Omega} ||\phi_{\om,0}||_{\mathcal{B}_{\om}}<\infty.$ This follows by checking that the proof of Proposition 1 in \cite{DFGTV18a} works in our current setting with the obvious modifications;  we note that Proposition 1 \cite{DFGTV18a} only assumes conditions  ($\mathcal{C}$1) and ($\mathcal{C}$3).
\end{itemize}
We are thus left with showing conditions ($\mathcal{C}$1) and ($\mathcal{C}$2) in our setting. We will get both at the same time as a consequence of the following argument, which consists in adapting to our current situation  the final part 
of Lemma \ref{closed ly ineq App}. Our starting point will be the inequality (\ref{closed var ineq over partition}). 
Since we are working with the normalized operator cocycle $\tilde{\mathcal{L}}_{\omega,0}^N$, we have to divide the expression in  (\ref{closed var ineq over partition}) by $\lm^{n'}_{\om, 0};$ moreover  we have to replace the measure $\nu_{\om, 0}$  with the equivalent conformal probability measure $\nu_{\om, 0}.$
Therefore (\ref{closed var ineq over partition}) now becomes
\begin{equation}\label{zero}
(\lm_{\om, 0}^n)^{-1}\var\lt(\ind_{T_\om^n(Z)} \lt((f g_{\om, 0}^{(n)})\circ T_{\om,Z}^{-n}\rt)\rt)
\le 9\frac{\norm{g_{\om, 0}^{(n)}}_{\infty,\om}}{\lm^n_{\om, 0 }}\text{var}_Z(f)+
		\frac{8\norm{g_{\om, 0}^{(n)}}_{\infty,\om}}
{\lm^n_{\om, 0}\nu_{\om, 0}(Z)}\nu_{\om, 0}(|f|_Z),
	\end{equation}
	where $Z$ is an element of $\mathcal{Z}^{(n)}_{\om, 0}.$ Since each  element  of the partition  $\mathcal{Z}^{(n)}_{\om, 0}$ has nonempty interior and the measure $\nu_{\om, 0}$  charges open intervals, if we set
	$$
\overline{\nu}^{(n)}_{\om,0}:=\min_{Z\in\cZ_{\om,0}^{(n)}}
\nu_{\om,0}(Z)>0,
$$
and we take the sum over the $Z\in \mathcal{Z}^{(n)}_{\om, 0}$ we finally get
\begin{equation}\label{p}
\text{var}(\tilde{\mathcal{L}}^n_{\om,0}f)\le 9\frac{\norm{g_{\om,0}^{(n)}}_{\infty,\om}}{\lm^{n}_{\om, 0}}\text{var}(f)+8\frac{\norm{g_{\om,0}^{(n)}}_{\infty,\om}\nu_{\om,0}(|f|)}{\lm^n_{\om, 0}\overline{\nu}_{\om,0}^{(n)}}.
\end{equation}Notice that  condition ($\mathcal{C}$2) requires that  $\tilde{\alpha}^N<1$, which in our case becomes
\begin{equation}\label{fgr}
9\frac{\esssup_{\om}\norm{g_{\om,0}^{(n)}}_{\infty,\om}}{\essinf_{\om}\lm^{n}_{\om, 0}}<1, 
\end{equation}
or equivalently 
$$
9\frac{\esssup_{\om}\norm{g_{\om,0}^{(n)}}_{\infty,\om}}{\essinf_{\om}\inf\mathcal{L}^n_{\om,0}\ind }<1,
$$
which is guaranteed by (E8).
We now move on to check condition ($\mathcal{C}$1).
From (\ref{p}), setting $n=1$ we  see that a sufficient condition for ($\mathcal{C}$1) is
\begin{equation}\label{dd}
  \frac{\esssup_{\om}\norm{g_{\om, 0}^{(1)}}_{\infty,\om}}{\essinf_{\om}\lm^1_{\om, 0}\overline{\nu}_{\om,0}}<\infty.
\end{equation}
or equivalently
\begin{equation}\label{d}
\frac{\esssup_{\om}\norm{g_{\om,0}^{(1)}}_{\infty,\om}}{\essinf_{\om}\inf\mathcal{L}_{\omega,0}\ind\overline{\nu}_{\om,0}}<\infty
\end{equation}
Condition (\ref{d}) is in principle checkable ($n=1$) and we could {\em assume} it as a part of the admissibility condition for our random cocycle $\mathcal{R}=(\Omega, m, \sigma, \mathcal{B}_{\om}, \tcL_{\om,0}).$ 

The admissibility conditions stated in \cite{dragicevic_spectral_2018}, in particular ($\mathcal{C}$2) and  ($\mathcal{C}$3), were sufficient to prove the quasi-compactness of the random cocycle introduced in  \cite{dragicevic_spectral_2018},  but they rely on another assumption, which was a part of the Banach space construction, namely that  $\BV_1$ was compactly injected into $L^1(\text{Leb}).$ We saw above that the same result holds for our Banach space $\mathcal{B}_{\om}$ and our measure $\nu_{\om, 0}.$ In order to prove quasi-compactness following Lemma 2.1 in \cite{dragicevic_spectral_2018}, we use condition  
($\mathcal{C}$2) with the almost-sure bound  $\tilde{\alpha}^N(\om)<1.$
We must additionally prove that the top Lyapunov exponent $\Lambda(\mathcal{R})$ defined by the limit 
$\Lambda(\mathcal{R})=\lim_{n\rightarrow \infty}\frac1n \log||\tcL_{\om,0}^n||_{\mathcal{B}_{\om}}$ is not less than zero.
This will follow by applying item \ref{item 5} and the bound on the density $\phi_{\om,0}$ in item \ref{item 2}, both in Theorem \ref{existence theorem}.
Since the density $\phi_{\om, 0}$ is for $m$-a.e. $\om\in \Omega$ bounded from below by the constant $C^{-1},$ we have that $\limsup_{n\rightarrow \infty}\frac{1}{n}\log \|\phi_{\sigma^n\om,0}\|_{\mathcal{B}_{\sg^n\om}}\ge\limsup_{n\rightarrow \infty}\frac{1}{n}\log C^{-1}=0$.
Now we continue as in the proof of Theorem 3.2 in \cite{dragicevic_spectral_2018}. Let $N$ be the integer for which 
$\tilde{\alpha}^N(\om)<1$  $\om$ a.s. We then consider the cocycle $\mathcal{R}_N$  generated by the map $\om\rightarrow \cL_{\om,0}^N.$ It is easy to verify that
 $\Lambda(\mathcal{R}_N)=N\Lambda(\mathcal{R})$ and the index of compactness
$\kappa$ (see section 2.1 in \cite{dragicevic_spectral_2018} for the definition) 
of the two cocycles also satisfies
$\kappa(\mathcal{R}_N)=N\kappa(\mathcal{R}).$ Because $\int \log \tilde{\alpha}^N(\om)\ 
dm < 0 \le \Lambda(\mathcal{R}_N)$, Lemma 2.1 in \cite{dragicevic_spectral_2018} guarantees that $\kappa(\mathcal{R}_N) \le \int \log \tilde{\alpha}^N(\om)\ 
dm$, 
proving
 quasicompactness of $\mathcal{R}.$

As we said at the beginning of this section, we could now follow almost verbatim the proofs in \cite{dragicevic_spectral_2018} by using  our operators $\cL_{\om,0}$ and by replacing the Lebesgue measure with  the conformal measures $\nu_{\om, 0}.$ We briefly sketch the main steps of the approach; we first define the  {\em observable} $v$ as a measurable map $v:\Omega\times [0,1] \rightarrow \mathbb{R}$ with the additional properties 
that $||v(\omega, x)||_{L^1(\nu_0)}<\infty$ and 
$\text{ess}\sup_{\om}||v_{\om}||_{\infty,\om}<\infty,$ where we set  $v_{\om}(x)=v(\om, x).$ 
Moreover we assume $v_\omega$ is fibrewise centered: $\int v_{\om}(x)d\mu_{\om,0}(x)=0,$ for $m$-a.e. $\om\in \Omega.$ We then define the twisted (normalized) transfer operator cocycle as:
$$
\tcL_{\om,0}^{\theta}(f)=\tcL_{\om,0}(e^{\theta v_{\om}(\cdot)}f), \ f\in\mathcal{B}_{\om}. 
$$
The link with the limit theorems is provided by the following equality which follows easily by the duality property of the operator:
$$
\int \tcL_{\om,0}^{\theta,n}(f)d\nu_{\sigma^n\om, 0}=\int e^{\theta S_n v_{\om}(\cdot)}f d\nu_{\om, 0},
$$
where $\tcL_{\om,0}^{\theta,n}=\tcL_{\sigma^{n-1}\om,0}^{\theta}\circ \cdots \circ \tcL_{\om,0}^{\theta}.$ The adaptation of Theorem 3.12 in \cite{dragicevic_spectral_2018} to our case, shows that the twisted operator is quasi-compact (in $\mathcal{B}_{\om}$) for $\theta$ close to zero. Moreover, by denoting with $\Lambda(\theta)=\lim_{n\rightarrow \infty}\frac1n\log ||\tcL_{\om,0}^{\theta,n}||_{\mathcal{B}_{\om}}$ the top Lyapunov exponent of the cocycle, the map $\theta\rightarrow \Lambda(\theta)$ is of class $C^2$ and strictly convex in a neighborhood of zero. We have also the analog of Lemma 4.3 in \cite{dragicevic_spectral_2018}, linking the asymptotic behavior of characteristic functions associated to Birkhoff sums with $\Lambda(\theta):$
$$
\lim_{n\rightarrow \infty}\frac1n\log\left|\int e^{\theta S_nv_{\om}(x)}d\mu_{\om, 0}\right|=\Lambda(\theta).
$$
We are now ready to collect our results on a few limit theorems; we first define the variance $\Sigma^2=\int v_{\om}(x)^2\ d\mu_{\om, 0}\ dm+2\sum_{n=1}^{\infty}\int v_{\om}(x)v(T_\om^n(x))\ d\mu_{\om, 0}\ dm.$
We also define the {\em aperiodicity condition} by asking that for $m$-a.e. $\om\in \Omega$ and for every compact interval $J\subset \mathbb{R}\setminus \{0\},$ there exists $C(\om)>0$ and $\rho\in(0,1)$ such that $||\tcL_{\om,0}^{it,n}||_{\mathcal{B}_{\om}}\le C(\omega)  \rho^n,$ for $t\in J$ and $n\ge 0.$
\begin{theorem}
Suppose that our random  cocycle $\mathcal{R}=(\Omega, m, \sigma, \mathcal{B}_{\om}, \tcL_{\om,0})$ is admissible and take the centered observable $v$ verifying  $\text{ess}\sup_{\om}||v_{\om}||_{\infty,\om}<\infty.$ Then:
\begin{itemize}
\item ({\bf Large deviations}). There exists $\varkappa_0>0$ and a non-random function $c:(-\varkappa_0, \varkappa_0)\rightarrow\mathbb{R}$ which is nonnegative, continuous, strictly convex, vanishing only at $0$ and such that
$$
\lim_{n\rightarrow \infty} \frac1n\log \mu_{\om, 0}(S_nv_{\om}(\cdot)>n\varkappa)=-c(\varkappa), \ \text{for}\ 0<\varkappa<\varkappa_0, \ \text{and}\ m-\text{a.e.} \ \om\in \Omega.
$$
\item ({\bf Central limit theorem}). Assume that $\Sigma^2>0.$ Then for every bounded and continuous function $\phi:\mathbb{R}\rightarrow \mathbb{R}$ and $m$-\text{a.e.}  $\om\in \Omega$, we have
$$
\lim_{n\rightarrow \infty}\int \phi\left(\frac{S_nv_{\om}(x)}{\sqrt{n}}\right)d\mu_{\om, 0}=\int \phi\ d\mathcal{N}(0, \Sigma^2).
$$
\item ({\bf Local central limit theorem}). Suppose the aperiodicity condition holds. Then for $m$-a.e. $\om\in \Omega$  and every bounded interval $J\subset \mathbb{R},$ we have
$$
\lim_{n\rightarrow \infty}\sup_{s\in \mathbb{R}}\left|\Sigma\sqrt{n}\mu_{\om,0}(s+S_nv_{\om}(\cdot)\in J)-\frac{1}{\sqrt{2\pi}}e^{-\frac{s^2}{2n\Sigma^2}}|J|\right|=0.
$$

    \end{itemize}
\end{theorem}
\subsection{The martingale approach}
The  deviation result quoted above allows to control, asymptotically, deviations of order $\varkappa,$ for $\varkappa$ in a sufficiently small bounded interval around $0.$ We now show how to extend that result to   any $\varkappa$ by getting an exponential bound on the deviation of the distribution function instead of an asymptotic expansion; in particular our bound shows that deviation probability stays small for finite $n$,which is a typical concentration property.  
 We now derive this result since it will provide us with the martingale that is used to obtain the ASIP. 
 Recall that the equivariant measure $\mu_{\om,0}$ is equivalent to $\nu_{\om, 0}.$ We consider again the fibrewise centered observable $v$ from the previous section, and 
we wish to estimate
$$\mu_{\om,0}\left(\left|\frac1n\sum_{k=0}^{n-1}v_{\sigma^k\om}\circ
T^k_{\om}\right|>\varkappa\right).
$$
We will use the following result (Azuma, \cite{AZ}):
Let $\{M_i\}_{i\in \mathbb{N}}$ be a sequence of martingale differences.
If there is $a>0$ such that $||M_i||_{\infty}<a$ for all $i,$ then we have
for all $b\in \mathbb{R}:$
$$
\mu_{\om,0}\left(\sum_{i=1}^nM_i\ge nb\right)\le e^{-n\frac{b^2}{2a^2}}.
$$

If we denote  $\mathcal{F}^k_{\om}:=(T^k_{\om})^{-1}(\mathcal{F}),$
we can easily prove that for a measurable map
 $\phi:[0,1]\rightarrow \mathbb{{R}},$ we have (the expectations $\mathbb{E}_{\om}$
  will be
 taken with respect to $\mu_{\om,0}$):
\begin{equation}\label{one}
\mathbb{E}_{\om}(\phi\circ T^l_{\o}|\mathcal{F}^n_{\om})=
\left(\frac{\tcL_{\sigma^l\om,0}^{n-l}(\lm_{\sigma^l\om,0}\phi)}{\lm_{\sigma^n\om,0}}\right)
\circ T_{\om}^n
\end{equation}
We now set:
 $$M_n:= v_{\sigma^n\om}+G_n-G_{n+1}\circ T_{\sigma^n\om},$$ with $G_0=0$ and
 \begin{equation}\label{two}
 G_{n+1}=\frac{\tcL_{\sigma^n\om,0}(v_{\sigma^n\om}\lm_{\sigma^n\om,0}+G_n \lm_{\sigma^n\om,0})} {\lm_{\sigma^{k+1}\om,0}}
 \end{equation}
   It is easy to check that
   $$
   \mathbb{E}_{\om}(M_n\circ T_{\om}^n| \mathcal{F}^{n+1}_{\om} )=0,
   $$
   which means that the sequence $(M_n\circ T_{\om}^n)$ is a reversed martingale difference with respect to the filtration $\mathcal{F}^{n}_{\om}.$
   By iterating (\ref{two}) we get
   \begin{equation}\label{three}
   G_n=\frac{1}{\lm_{\sigma^n\om,0}}\sum_{j=0}^{n-1}
   \tcL^{(n-j)}_{\sigma^j\om,0}
   (v_{\sigma^j\om}
   \lm_{\sigma^j\om,0}),
   \end{equation}
  and by a telescopic trick we have
   $$\sum_{k=0}^{n-1}M_k\circ T_{\om}^k=\sum_{k=0}^{n-1}v_{\sigma^k\om}T_{\om}^n-G_n\circ T_{\om}^n.$$
Suppose for the moment we could bound $G_n$ uniformly in $n$, but not necessarily in $\om$, by $||G_n||_{\infty,\om}\le C_1(\om).$ Since by assumption there exists a constant $C_2$ such that $\text{ess}\sup_{\om}||v_{\om}||_{\infty,\om}\le C_2,$  we have $||M_n||_{\infty,\om}\le C_2+2C_1({\om}),$ and by Azuma:
$$
\mu_{\om,0}\left(\left|\frac1n\sum_{k=0}^{n-1}M_k\circ
T^k_{\om}\right|>\frac{\varkappa}{2}\right)\le 2\exp\left\{-\frac{\varkappa^2}{8(C_2+2C_1({\om}))^2}n\right\}.
$$
Therefore
\begin{align*}
\mu_{\om,0}\left(\left|\frac1n\sum_{k=0}^{n-1}v_{\sigma^k\om}\circ
T^k_{\om}\right|>\varkappa\right)
&\le 
\mu_{\om,0}\left(\left|\frac1n\sum_{k=0}^{n-1}M_k\circ
T^k_{\om}\right|+\frac1n C_1(\om)>\varkappa\right)
\\
&\le
\mu_{\om,0}\left(\left|\frac1n\sum_{k=0}^{n-1}M_k\circ
T^k_{\om}\right|>\frac{\varkappa}{2}\right)
\\
&\le 
2\exp\left\{-\frac{\varkappa^2}{8(C_2+2C_1({\om}))^2}n\right\},
\end{align*}
provided $n>n_0,$ where $n_0$ verifies $\frac1n_0 C_1(\om)\le\frac{\varkappa}{2}.$\\
 In order to estimate $C_1,$ we proceed in the following manner. We have
\begin{equation}\label{four}
||G_n||_{\infty}\le \frac{1}{\lm_{\sigma^n\om,0}}\sum_{j=0}^{n-1}
   ||\tcL^{(n-j)}_{\sigma^j\om,0}
   (v_{\sigma^j\om}
   \lm_{\sigma^j\om,0})||_{\infty,\om}.
\end{equation}
The multiplier $\lm_{\om,0}$ 
is bounded from above  and  from below $\om$ a.s. respectively by, say, $U$ and $1/U$ by  conditions $(\mathcal{C}1)$ and item \ref{item 1} Theorem \ref{existence theorem}, respectively. Then we use item  \ref{item 5} of Theorem \ref{existence theorem} and property (V3) to bound the term into the sum. In conclusion we get:

$$
||G_n||_{\infty,\om}\le U^2DC_2 \sum_{j=0}^{\infty} \kappa^{n-j}:=C_1.
$$
We summarize  this result in the following
\begin{theorem}({\bf Large deviations bound})\\
Suppose that our random  cocycle $\mathcal{R}=(\Omega, m, \sigma, \mathcal{B}_{\om}, \tcL_{\om,0})$ is admissible and take the fibrewise centered observable $v$ satisfying  $\text{ess}\sup_{\om}||v_{\om}||_{\infty,\om}=C_2<\infty.$ Then there exists $C_1>0$ such that for $m$-a.e. $\om\in \Omega$
and  for any $\varkappa>0:$
$$
\mu_{\om,0}\left(\left|\frac1n\sum_{k=0}^{n-1}v_{\sigma^k\om}\circ
T^k_{\om}\right|>\varkappa\right)\le2\exp\left\{-\frac{\varkappa^2}{8(C_2+2C_1)^2}n\right\},
$$
provided $n>n_0,$ where $n_0$ is any integer number satisfying $\frac{1}{n_0} C_1\le\frac{\varkappa}{2}.$  
\end{theorem}

We point out that whenever our random cocycle  $\mathcal{R}=(\Omega, m, \sigma, \mathcal{B}_{\om}, \tcL_{\om,0})$ is admissible, we satisfy all the assumptions in the paper \cite{DFGTV18a}, which allows us to get the almost sure invariance principle. We  should simply replace the (conformal) Lebesgue measure  with our random conformal measure $\nu_{\om, 0},$ and use our operators $\tcL_{\om,0},$ as we already did in Section \ref{NG approach} for the Nagaev-Guivarc'h approach. It is worthwhile to observe that \cite{DFGTV18a} relies on the construction of a suitable martingale, which in our case is precisely  the reversed martingale difference $M_n\circ T^n_{\om}$ obtained above in the proof of the large deviation bound. We denote with $\Sigma^2_n$ the variance $\Sigma^2_n=\mathbb{E}_{\om}\left(\sum_{k=0}^{\infty}v_{\sigma^k\om}\circ T_{\om}^k\right)^2,$ where $v_{\om}$ is a centered observable, and with $\Sigma^2$ the quantity introduced above, in the statement of the central limit theorem. We have the following:
\begin{theorem}({\bf Almost sure invariance principle})

Suppose that our random  cocycle $\mathcal{R}=(\Omega, m, \sigma, \mathcal{B}_{\om}, \tcL_{\om,0})$ is admissible and consider  a centered   observable $v_{\om}.$ Then $\lim_{n\rightarrow \infty}\frac1n\Sigma^2_n=\Sigma^2.$ Moreover one of the following cases hold:

(i) $\Sigma=0,$ and this is equivalent to the existence of $\varphi\in L^2(\mu_0)$
$$
v=\varphi-\varphi\circ T,
$$
where $T$ is the induced skew product map and $\mu$ its invariant measure, see Definition \ref{def: random prob measures}.

(ii) $\Sigma^2>0$ and in this case for $m$ a.e. $\om\in \Omega,$ $\forall \varrho>\frac54$, by enlarging the probability space $(X, \mathcal{F}, \mu_{\om,0})$ if necessary, it is possible to find a sequence $(B_k)_k$ of independent centered Gaussian random variables such that
$$
\sup_{1\le k \le n} \left|\sum_{i=1}^k(v_{\sigma^i\om}\circ T_{\om}^i)-\sum_{i=1}^k B_i\right|=O(n^{1/4}\log^{\varrho}(n)),\ \mu_{\om,0} \ \text{a.s.}.
$$

\end{theorem}

We show in Section \ref{hts} that the distribution of the first hitting random time in a decreasing sequence of holes follows an exponential law, see Proposition \ref{hve}. 
We now prove a recurrence result by giving a quantitative estimate of the {\em asymptotic number of entry times in a descending sequence of balls}. This is known as the {\em shrinking target problem}. 
\begin{proposition}\label{shri}
Suppose that our random cocycle $\mathcal{R}=(\Omega, m, \sigma, \mathcal{B}_{\om}, \tcL_{\om,0})$ is admissible.
Take $B_{\xi_{\sigma^k\om}}(p)$ as a  descending sequence of balls centered at the point $p$ and with  radius $\xi_{\sigma^k\om},$  such that 
$$
E_{\om, n}:=\sum_{k=0}^{n-1}\mu_{\sigma^k_\om,0}(B_{\xi_{\sigma^k\om}}(p))\rightarrow \infty 
\quad \mbox{ for $m$-a.e.\ $\omega$}.
$$
Then for $m$-a.e. $\om$ and $\mu_{\om,0}$-almost all $x,$ $T^k_{\om}(x)\in B_{\xi_{\sigma^k\om}}(p)$ for infinitely many $k$ and
$$
\frac{1}{E_{\om, n}}\#\{k\ge 0, T^k_{\om}(x)\in B_{\xi_{\sigma^k\om}}(p)\}\rightarrow 1.
$$
\end{proposition}
The proposition is a simple consequence of the following {\em Borel-Cantelli} like property, whenever we put the observable $v$ in the theorem below as $v_{\sigma^k\om}=\ind_{B_{\xi_{\sigma^k\om}}(p)}.$
\begin{theorem}\label{BC}
Suppose that our random cocycle $\mathcal{R}=(\Omega, m, \sigma, \mathcal{B}_{\om}, \tcL_{\om,0})$ is admissible. Take a nonnegative  observable $v,$  satisfying  $\text{ess}\sup_{\om}||v_{\om}||_{\mathcal{B}_{\om}}\le M.$
Suppose that for $m$ a.e. $\om\in \Omega$
$$
E_{\om, n}:=\sum_{k=0}^{n-1} \int v_{\sigma^k\om}(x)d\mu_{\sigma^k_{\om},0}(x)\rightarrow \infty, \ n\rightarrow \infty.
$$
Then for $m$ a.e. $\om\in \Omega$
$$
\lim_{n\rightarrow\infty}\frac{1}{E_{\om, n}}\sum_{k=0}^{n-1}v_{\sigma^k\om}(T^k_{\om}x)=1,
$$
for $\mu_{\om,0}$-almost all $x$.
\end{theorem}
\begin{proof}
Let us write 
$$
S_{\om}:=\int \left(\sum_{m<k\le n}\left(v_{\sigma^k\om}(T^k_{\om}y)-\int v_{\sigma^k\om}d\mu_{\sigma^k_{\om},0}\right)\right)^2d\mu_{\om,0}(y).
$$
If we could prove that $S_{\om}\le \text{constant}\sum_{m<k\le n} \int v_{\sigma^k\om}d\mu_{\sigma^k\om,0},$
then the result will follow by applying Sprindzuk theorem \cite{Spri}\footnote{We recall here the Sprindzuk theorem. Let $(\Omega, \mathcal{B}, \mu_0)$ be a probability space and let $(f_k)_k$ be a sequence of nonnegative and measurable functions on $\Omega.$ Moreover, let $(g_k)_k$ and $(h_k)_k$ be bounded sequences of real numbers such that $0\le g_k\le h_k.$ Assume that there exists $C>0$ such that
$$
\int \left(\sum_{m<k\le n}(f_k(x)-g_k)\right)^2d\mu_0(x)\le C\sum_{m<k\le n}h_k
$$
for $m<n.$ Then, for every $\eps>0$ we have
$$
\sum_{1\le k\le n}f_k(x)=\sum_{1\le k\le n}g_k+O(\Theta^{1/2}(n)\log^{3/2+\eps}\Theta(n)),
$$
for $\mu$ a.e. $x\in \Omega,$ where $\Theta(n)=\sum_{1\le k\le n}h_k.$
}, where we identify $\int v_{\sigma^k\om}d\mu_{\sigma^k\om,0}$ with the functions $g_k$ and $h_k$ in the previous footnote. So we have
\begin{align*}
S_{\om}&=\int \left[\sum_{m<k\le n}v_{\sigma^k\om}(T^k_{\om}y)\right]^2d\mu_{\om,0}(y)-\left[\sum_{m<k\le n}\int v_{\sigma^k\om}d\mu_{\sigma^k_{\om},0}\right]^2
\\
&=
\sum_{m<k\le n}\int v_{\sigma^k\om}(T^k_{\om}y)^2d\mu_{\om,0}+2\sum_{m<i<j\le n}\int v_{\sigma^i\om}(T^i_{\om}y)v_{\sigma^j\om}(T^j_{\om}y)d\mu_{\om,0}
\\
&\qquad-
\sum_{m<k\le n}\left[\int v_{\sigma^k\om}d\mu_{\sigma^k\om,0}\right]^2-2\sum_{m<i<j\le n}\int v_{\sigma^i\om}d\mu_{\sigma^i\om,0}\int v_{\sigma^j\om}d\mu_{\sigma^j\om,0}.
\end{align*}
We now discard the third negative piece and bound the first by H\"older inequality as
$$
\sum_{m<k\le n}\int v_{\sigma^k\om}(T^k_{\om}y)^2d\mu_{\om,0}\le \sum_{m<k\le n}M \int v_{\sigma^k\om}d\mu_{\sigma^k\om,0}.
$$
Then for the remaining two pieces we use the decay of correlations given in Theorem 2.21 of our paper \cite{AFGTV20}, where the observables are taken in $\mathcal{B}_{\om}$ and in $L^1(\mu_{\om,0}).$  We can apply it to our case since the two observables coincide with $v_{\om}$, and the measures $\mu_{\om,0}$ and $\nu_{\om,0}$ are equivalent. We point out that this result improves the decay bound given in item \eqref{item 6} of Theorem \ref{existence theorem} where the presence of holes for $\epsilon>0$ forced us to use $L^{\infty}(\nu_{\om,\eps})$ instead of $L^1(\nu_{\om,\eps}).$ 
Hence:
$$
\left|\int v_{\sigma^i\om}(T^i_{\om}y)v_{\sigma^j\om}(T^j_{\om}y)d\mu_{\om,0}-\int v_{\sigma^i\om}d\mu_{\sigma^i\om,0}\int v_{\sigma^j\om}d\mu_{\sigma^j\om,0}\right|\le
CM \kappa^{j-i}\int v_{\sigma^j\om}d\mu_{\sigma^j\om,0}.
$$
In conclusion, we get
\begin{align*}
S_{\om}&\le \sum_{m<k\le n}M \int v_{\sigma^k\om}d\mu_{\sigma^k\om,0}+C M\sum_{m<i<j\le n}\kappa^{j-i}\int v_{\sigma^j\om}d\mu_{\sigma^j\om,0}
\\
&\le
M\left(\frac{C}{1-\kappa}+1\right) \sum_{m<k\le n} \int v_{\sigma^k\om}d\mu_{\sigma^k\om,0},
\end{align*}
which satisfies Sprindzuk.
\end{proof}

\section{Examples}
\label{sec:examples}
In this section we present explicit examples of random dynamical systems to illustrate the quenched extremal index formula in Theorem \ref{evtthm}.
In all cases, 
Theorem \ref{existence theorem} can be applied to guarantee the existence of all objects therein for the perturbed random open system (in brief, the thermodynamic formalism for the closed random system implies a thermodynamic formalism for the perturbed random open system).
In Examples \hyperref[example1]{1}--\hyperref[example3]{3} we derive explicit expressions for the quenched extreme value laws. 

Our examples are piecewise-monotonic interval map cocycles 
whose transfer operators act on $\mathcal{B}_\omega=\BV:=\BV([0,1])$ for a.e.\ $\omega$.
The norm we will use is $\|\cdot\|_{\mathcal{B}_\omega}=\|\cdot\|_\BV:=\var(\cdot)+\nu_{\omega,0}(|\cdot|)$.
In Section \ref{sec: existence} we noted that (\ref{B}) is automatically satisfied.
Lemma \ref{DFGTV18alemma} shows that (\ref{CCM}) holds under assumptions (\ref{E2}), (\ref{E3}), (\ref{E4}), and (\ref{E8}).
To obtain a random open system, in addition to (\ref{B}) and (\ref{CCM}) we need (\ref{M}) (implying \eqref{M1} and  (\ref{M2})), (\ref{A}), and (\ref{cond X}).
Therefore, in each example we verify conditions (\ref{A}), (\ref{cond X}) (Section~\ref{sec: open systems})  and (\ref{M}) and (\ref{E1})--(\ref{E9}) (Section~\ref{sec: existence});  and where required (\ref{xibound}). 

\subsection{Example 1\xlabel[1]{example1}:  Random maps and random holes centred on a non-random fixed point (random observations with non-random maximum)}
We show that nontrivial quenched extremal indices occur even for holes centred on a non-random fixed point $x_0$.
    To define a random (open) map, let $(\Om, m)$ be a complete probability space,
and $\sg:(\Om, m)\to (\Om, m)$ be an ergodic, invertible, probability-preserving transformation. For example, one could consider an irrational rotation on the circle. 
Let $\Om=\cup_{j=1}^\infty \Om_j$ be a countable partition of $\Om$ into measurable sets on which $\omega\mapsto T_\omega$ is constant. 
This ensures that \eqref{M} is satisfied.
For each $\om \in \Om$ let $\cJ_{\om,0} = [0,1]$ and let
$T_\omega:\mathcal{J}_{\omega,0}\to \mathcal{J}_{\sigma\omega,0}$ be random, with all maps fixing $x_0\in [0,1]$.
The observation functions $h_\omega:\mathcal{J}_{\omega,0}\to\mathbb{R}$ have a unique maximum at $x_0$ for a.e.\ $\omega$.

To make all of the objects and calculations as simple as possible and illustrate some of the underlying mechanisms for nontrivial extremal indices we use the following specific family of maps $\{T_\omega\}$:
\begin{equation}
	\label{eg1}
	T_\omega(x)=\left\{
	\begin{array}{ll}
		1-2x/(1-1/s_\omega), & \hbox{$0\le x\le (1-1/s_\omega)/2$;} \\
		s_\omega x-(s_\omega-1)/2, & \hbox{$(1-1/s_\omega)/2\le x\le (1+1/s_\omega)/2$;} \\ 1-(2x-(1+1/s_\omega))/(1-1/s_\omega)& \hbox{$(1+1/s_\omega)/2\le x\le 1$,}
	\end{array}
	\right.
\end{equation}
where $1<s\le s_\omega\le S<\infty$ and $s\rvert_{\Om_j}$ is constant for each $j\geq 1$;  thus (\ref{M1}) holds. 
These maps have three full linear branches, and therefore Lebesgue measure is preserved by all maps $T_\omega$;  i.e.\ $\mu_{\omega,0}=\mathrm{Leb}$ for a.e. $\omega$.
The central branch has slope $s_\omega$ and passes through the fixed point $x_0=1/2$, which lies at the centre of the central branch;  see Figure \ref{fig:map}.
\begin{figure}[hbt]
	\centering
	\includegraphics[width=6cm]{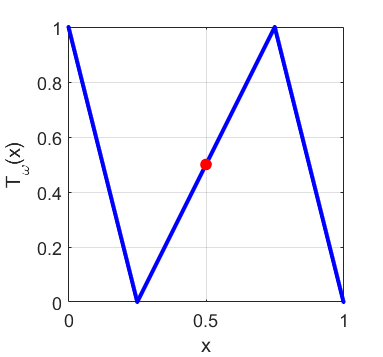}
	\caption{Graph of map $T_\omega$, with $s_\omega=2$.}\label{fig:map}
\end{figure}
A specific random driving could be $\sigma:S^1\to S^1$ given by $\sigma(\omega)=\omega+\alpha$ for some $\alpha\notin\mathbb{Q}$ and $s_\omega=s^0+s^1\cdot\omega$ for $1< s^0<\infty$ and $0<s^1<\infty$, but only the ergodicity of $\sigma$ will be important for us.
Setting $g_{\omega,0}=1/|T'_\omega|$ it is immediate that (\ref{E1}), (\ref{E2}), (\ref{E3}), and (\ref{E4}) hold. 

We select a measurable family of  observation functions $h:\Omega\times [0,1]\to\mathbb{R}$.
For a.e.\ $\omega\in\Omega$, $h_\omega$ is $C^1$, has a unique maximum at $x_0=1/2$,  and is a locally even function about $x_0$.
For small $\epsilon_N$ we will then have that $H_{\omega,\epsilon_N}=\{x\in [0,1]:h_\omega(x)>z_{\omega,N}\}$ is a small neighbourhood of $x_0$, satisfying (\ref{E5}) and (\ref{A}). 
    In light of Remark~\ref{rem check cond X} and Proposition \ref{prop check cond X} we see that \eqref{cond X} holds.
The corresponding cocycle of open operators $\mathcal{L}_{\om,\ep}$ satisfy (\ref{M}).
Given an essentially bounded function $t$, the $z_{\omega,N}$ are chosen to satisfy (\ref{xibound}). 
Because the $h_\omega$ are $C^1$ with unique maxima and $\mu_{\omega,0}$ is Lebesgue for all $\omega$ we can make such a choice of $z_{\omega,N}$ with\footnote{In the situation where $z_{\omega,N}$ is constant a.e.\ for each $N\ge 1$, in general we cannot satisfy (\ref{xibound}) for a given fixed scaling function $t$.  In the case where $t$ is also constant a.e.\ if $h_\omega$ is random but sufficiently close to a non-random observation $h$ for each $\omega$, we can satisfy (\ref{xibound}) by incorporating the error term $\xi_{\omega,N}$. This situation corresponds to a decreasing family of holes whose length fluctuates slightly in $\omega$ with the fluctation decreasing to zero sufficiently fast.  A similar situation can occur if the scaling function is a general measurable function.  These situations motivate the use of the error term $\xi_{\omega,N}$ in condition (\ref{xibound}).} $\xi_{\omega,N}\equiv 0$. 
Furthermore, as $\mu_{\om,0}$ is Lebesgue we have that \eqref{xibound} implies that assumption \eqref{E6} is automatically satisfied.
With the above choices,  (\ref{E7}) is clearly satisfied for $n'=1$.
To show condition (\ref{E8}) with $n'=1$ we note that 
$\|g_{\omega,0}\|_\infty=(1-1/s_\omega)/2<(1-1/S)/2$, while
$\inf\mathcal{L}_{\omega,\ep}\ind\ge 2\cdot (1-1/s_\omega)/2\ge 1-1/s$.
Therefore we require $(1-1/S)/2<1-1/s$, which by elementary rearrangement is always true for $0<s<S<\infty$.
Because there is a full branch outside the branch with the hole, (\ref{E9}) is satisfied with $n'=1$ and $k_o(n')=1$.

At this point we have checked all of the hypotheses of Theorem \ref{existence theorem} and we obtain that for sufficiently small holes as defined above, the corresponding random open dynamical system has a quenched thermodynamic formalism and quenched decay of correlations. 
We note that this result does not require (\ref{xibound});  the holes $H_{\epsilon_N}$ do not need to scale in any particular way with $N$, they simply need to be sufficiently small.
In order to next apply Theorem \ref{evtthm} 
we do require (\ref{xibound}) and additionally \eqref{C8}.

To finish the example, we verify (\ref{C8}).
We claim that for each fixed $k>0$, one has $\hat{q}_{\omega,\epsilon_N}^{(k)}=0$ for sufficiently sufficiently large $N$ (sufficiently small $\epsilon_N$).
If this were not the case, there must be a positive $\mu_{\omega,0}$-measure set of points that (i) lie in $H_{\sigma^{-k}\omega,\epsilon_N}$, (ii)  land outside the sequence of holes $H_{\sigma^{-(k-1)}\omega,\epsilon_N},\ldots,H_{\sigma^{-1}\omega,\epsilon_N}$ for the next $k-1$ iterations, and then (iii) land in $H_{\omega,\epsilon_N}$ on the $k^{\mathrm{th}}$ iteration.
In this example, because all $H_{\omega,\epsilon_N}$ are neighbourhoods of $x_0$ of diameter smaller than $(|t|_\infty+\gm_N)/N$, and the maps are locally expanding about $x_0$, for fixed $k$ one can find a large enough $N$ so that it is impossible to leave a small neighbourhood of $x_0$ and return after $k$ iterations.
For $k=0$, by definition \begin{equation}
	\label{q0}
	\hat{q}_{\omega,\epsilon_N}^{(0)}=\frac{\mu_{\sigma^{-1}\omega,0}(T^{-1}_{\sigma^{-1}\omega} H_{\omega,\epsilon_N}\cap H_{\sigma^{-1}\omega,\epsilon_N})}{\mu_{\sigma^{-1}\omega,0}(T^{-1}_{\sigma^{-1}\omega}H_{\omega,\epsilon_N})}=\frac{\mu_{\sigma^{-1}\omega,0}(T^{-1}_{\sigma^{-1}\omega} H_{\omega,\epsilon_N}\cap H_{\sigma^{-1}\omega,\epsilon_N})}{\mu_{\omega,0}(H_{\omega,\epsilon_N})}.
\end{equation}
There are two cases to consider:

\textit{Case 1:} $H_{\sigma^{-1}\omega,\epsilon_N}$ is larger than the local preimage of $H_{\omega,\epsilon_N}$;  that is, $T^{-1}_{\sigma^{-1}\omega} H_{\omega,\epsilon_N}\cap H_{\omega,\epsilon_N}\subset  H_{\sigma^{-1}\omega,\epsilon_N}$. Because of the linearity of the branch containing $x_0$, one has $\hat{q}_{\omega,\epsilon_N}^{(0)}=1/T'_{\sigma^{-1}\omega}(x_0)$.

\textit{Case 2:} $H_{\sigma^{-1}\omega,\epsilon_N}$ is smaller than the local preimage of $H_{\omega,\epsilon_N}$;  that is, $T^{-1}_{\sigma^{-1}\omega} H_{\omega,\epsilon_N}\cap H_{\omega,\epsilon_N}\supset  H_{\sigma^{-1}\omega,\epsilon_N}$.
By (\ref{q0}) and (\ref{xibound}) we then have $$\frac{t_{\sigma^{-1}\omega}-\gm_N}{t_\omega+\gm_N}\le \hat{q}_{\omega,\epsilon_N}^{(0)}\le \frac{t_{\sigma^{-1}\omega}+\gm_N}{t_\omega-\gm_N}$$ and thus for such an $\omega$, $\lim_{N\to\infty}\hat{q}_{\omega,\epsilon_N}^{(0)}=t_{\sigma^{-1}\omega}/t_{\omega}$.

Thus, combining the two cases,
$$\hat{q}^{(0)}_{\omega,0}:=\lim_{N\to\infty}\hat{q}^{(0)}_{\omega,\epsilon_N}=\min\left\{\frac{t_{\sigma^{-1}\omega}}{t_\omega},\frac{1}{|T'_{\sigma^{-1}\omega}(x_0)|}\right\}$$ exists for a.e.\ $\omega$, verifying $\eqref{C8}$.

Recalling that $\theta_{\omega,0}=1-\sum_{k=0}^\infty \hat{q}^{(k)}_{\omega,0}=1-\hat{q}^{(0)}_{\omega,0}$, we may now apply Theorem \ref{evtthm} to obtain the quenched extreme value law:
\begin{eqnarray}
	\nonumber
	\lim_{N\to\infty}\nu_{\om,0}\left(X_{\om,N-1,\ep_N}\right)
	&	=&
	\exp\left(-\int_\Om t_\om\ta_{\om,0}\, dm(\om)\right)\\
	\label{evtlaweg1}&=&\exp\left(-\int_\Om t_\omega\left(1- \min\left\{\frac{t_{\sigma^{-1}\omega}}{t_\omega},\frac{1}{|T'_{\sigma^{-1}\omega}(x_0)|}\right\}\right)\, dm(\om)\right)
\end{eqnarray}

This formula is a generalisation of the formula in Remark 8 \cite{keller_rare_2012}, where we may create nontrivial laws from either the random dynamics $T_\omega$, or the random scalings $t_\omega$, or both.
The following two special cases consider these mechanisms separately.
\begin{enumerate}
	\item \textit{Random maps, non-random scaling ($t_\omega$ takes a constant value $t>0$):}  Since $|T_\omega'|>1$, in this case (\ref{evtlaweg1}) becomes
	\begin{equation}
		\label{fixedscale}
		\lim_{N\to\infty}\nu_{\om,0}\left(X_{\om,N-1,\ep_N}\right)=\exp\left(-t\left[1-\int_\Omega \frac{1}{|T_\omega'(x_0)|}\ dm(\omega)\right]\right),
	\end{equation}
	and we see that we can interpret $\theta=1-\int_\Omega \frac{1}{|T_\omega'(x_0)|}\ dm(\omega)$ as an extremal index.
	
	\item \textit{Fixed map, random scaling:}
	Suppose $T_\omega\equiv T$;  then we may replace $T_\omega'(x_0)$ with $T'(x_0)$ in (\ref{evtlaweg1}),  
	and we see we the extremal index  depends on the choice of random scalings $t_\omega$ alone;  of course the thresholds $z_{\omega,N}$ depend on the chosen $t_\omega$.
\end{enumerate}

Similar results could be obtained with the $T_\omega$ possessing nonlinear branches.
The arguments above can also be extended to the case where $x_0$ is a periodic point of prime period $p$ for all maps;   the formula (\ref{evtlaweg1}) now includes $(T_\omega^p)'(x_0)$ and $t_\omega,   t_{\sigma^{1}\omega},\ldots,t_{\sigma^{-(p-1)}\omega}$.
If the scaling $t_\omega=t$ is non-random, one would simply replace $T'_\omega(x_0)$ in (\ref{fixedscale}) with $(T_\omega^p)'(x_0)$.

We recall that for deterministic $T$ (including non-uniformly hyperbolic maps), 
the extremal index enjoys a dichotomy, in the sense that it is equal to $1$ when a single hole $H_{\ep_N}$ shrinks to an aperiodic point, and it is strictly smaller than $1$ when the hole shrinks to a periodic point.
In the latter case the extremal index can be expressed in terms of the period; see \cite{book} for a general account of the above facts. 
Example \hyperref[example1]{1} shows that in a simple random setting there are many more ways to obtain nontrivial exponential limit laws, e.g.\ by random scaling, or by the existence of periodic orbits for only a positive measure set of $\omega$.

\subsection{Example 2\xlabel[2]{example2}: Random $\beta$-maps, random holes containing a non-random fixed point (random observations with non-random maximum), general geometric potential}
We show how a nontrivial extremal index can arise from random $\beta$-maps where statistics are generated by an equilibrium state of a general geometric potential.
Consider the ``no short branches'' random $\beta$-map example of Section 13.2 \cite{AFGTV20}, where $\beta_\omega\in \{2\}\cup\bigcup_{2\le k\le K}[k+\delta,k+1]$ for a.e.\ $\omega$, and some $\delta>0$ and $K<\infty$.
The measurability of $\omega\mapsto\beta_\omega$ yields (\ref{M1}).
We use the weight $g_\omega=1/|T'_\omega|^r$, $r\geq 0$.
To obtain (\ref{M}) the base dynamics is driven by an ergodic homeomorphism $\sigma$ on a Borel subset $\Omega$ of a separable, complete metric space, assuming that $\omega\mapsto \mathcal{
	L}_{\omega,0}$ has countable range, as in Example~\ref{example1}.
Our random observation function $h:\Omega\times [0,1]\to \mathbb{R}$ is measurable and for a.e.\ $\omega$, $h_\omega$ is $C^1$ with a unique maximum at $x=0$.
We select a measurable scaling function $t_\omega$;  either or both of $h_\omega$ and $t_\omega$ could be $\omega$-independent.
By assigning thresholds $z_{\omega,N}$ (which could also be $\omega$-independent) we obtain a decreasing family of holes $H_{\om,\ep_N}$, which are decreasing intervals with left endpoint at $x=0$.

Clearly (\ref{E2}) and (\ref{E3}) hold and Lemma 13.5 \cite{AFGTV20} provides (\ref{E4}).
Conditions (\ref{E5}) and (\ref{E6}), and (\ref{A}) are immediate, as is (\ref{E7}) with $n'=1$.	
Condition (\ref{EX}) holds because there is at least one full branch outside the branch with the hole.
Regarding (\ref{E8}), 
$\|g_{\omega,\epsilon}\|_\infty\le \beta_\omega^{-r}$ and arguing as in the previous example, 
since $\inf\mathcal{L}_{\omega,\epsilon}\ind\ge \lfloor\beta_\omega\rfloor/\beta_\omega^r$, we see (\ref{E8}) holds with $n'=1$ because $\beta_\omega\ge 2$.
Because there is at least one full branch outside the branch with the hole, (\ref{E9}) holds with $n'=1$ and $k_o(n')=1$.
We have now checked all of the hypotheses of Theorem \ref{existence theorem}.
By this theorem, for sufficiently small holes $H_\epsilon$, the corresponding random open dynamical system has a quenched thermodynamic formalism and quenched decay of correlations. 
As in the previous example, we emphasise that this result does not require (\ref{xibound}).

Because $\varphi_{\omega,0}$ is uniformly bounded above and below, as long as $z_{\omega,N}$ is random, we may adjust $z_{\omega,N}$ to obtain (\ref{xibound}).
We now verify (\ref{C8}) for our holes, which are of the form $H_{\omega,\epsilon_N}=[0,r_{\omega,N}]$, 
    in preparation to apply Theorem \ref{evtthm}. 
The same arguments from Case 1 and Case 2 of the previous example apply.
Case 2 is unchanged.
In Case 1 we have
$$\hat{q}^{(0)}_{\omega,\epsilon_N}=\frac{\mu_{\sigma^{-1}\omega,0}(T^{-1}_{\sigma^{-1}\omega} H_{\omega,\epsilon_N}\cap H_{\sigma^{-1}\omega,\epsilon_N})}{\mu_{\sigma^{-1}\omega,0}(T^{-1}_{\sigma^{-1}\omega}H_{\omega,\epsilon_N})}.$$
Because $T^{-1}_{\sigma^{-1}\omega} H_{\omega,\epsilon_N}$ is a finite union of left-closed intervals, using the fact that $\varphi_{\sigma^{-1}\omega,0}\in BV$, and therefore has left- and right-hand limits everywhere, we may redefine $\varphi_{\sigma^{-1}\omega,0}(y)$ at the finite collection of points $y\in T^{-1}_{\sigma^{-1}\omega}(0)$ so that 
$$\lim_{N\to\infty}\frac{\mu_{\sigma^{-1}\omega,0}(T^{-1}_{\sigma^{-1}\omega} H_{\omega,\epsilon_N}\cap H_{\sigma^{-1}\omega,\epsilon_N})}{\mu_{\sigma^{-1}\omega,0}(T^{-1}_{\sigma^{-1}\omega}H_{\omega,\epsilon_N})}=\frac{\varphi_{\sigma^{-1}\omega,0}(0)}{\sum_{y\in T^{-1}_{\sigma^{-1}\omega}(0)} \varphi_{\sigma^{-1}\omega,0}(y)},$$
recalling that in Case 1, 
$T^{-1}_{\sigma^{-1}\omega} H_{\omega,\epsilon_N}\subset H_{\sigma^{-1}\omega,\epsilon_N}$.
Taking the minimum of Cases 1 and 2 as in Example 1, we see that
$\hat{q}_{\omega,0}^{(0)}=\lim_{N\to\infty} \hat{q}_{\omega,\epsilon_N}^{(0)}$ exists, which verifies \eqref{C8}.
Finally,
\begin{equation}
	\label{evtlaweg2}
	\lim_{N\to\infty}\nu_{\om,0}\left(X_{\om,N-1,\ep_N}\right)
	=\exp\left(-\int_\Om t_\omega\left(1- \min\left\{\frac{t_{\sigma^{-1}\omega}}{t_\omega},\frac{\varphi_{\sigma^{-1}\omega,0}(0)}{\sum_{y\in T^{-1}_{\sigma^{-1}\omega}(0)} \varphi_{\sigma^{-1}\omega,0}(y)}\right\}\right)\, dm(\om)\right)
\end{equation}

\subsection{Example 3\xlabel[3]{example3}: A fixed map with random holes containing a fixed point (random observations with a non-random maximum)}

Let us now consider more closely the case of a fixed map and holes moving randomly around a fixed point $z$, a situation previously considered in \cite{bahsoun-vaienti-2013} in the annealed framework. 

We take a piecewise uniformly expanding map $T$ of the unit interval $I$ of class $C^2$, and such that $T$ is surjective on the domains of local injectivity. Moreover $T$  preserves a mixing measure $\mu$ equivalent to the Lebesgue measure $\text{Leb},$ with strictly positive density $\rho$.  
We moreover assume that $T$ and $\rho$ are continuous at a fixed point $x_0$. 
The observation functions $h_\omega$ have a common maximum at $x_0$, leading to holes $H_{\omega,\epsilon_N}$ that are closed intervals  containing the point $x_0$ for any $N$.
In Example \hyperref[example1]{1} the holes were centred on $x_0$ but could vary dramatically in diameter between $\omega$-fibers.
In this example the holes need not be centred on $x_0$ but must become more identical as they shrink.
Specifically, we assume
\begin{equation}\label{rf}
	\sup_k\frac{\text{Leb}(H_{\omega,\epsilon_N}\Delta H_{\sigma^{-k}\omega,\epsilon_N})}{\text{Leb}(H_{\omega,\epsilon_N})}\rightarrow 0, \ N\rightarrow \infty,
\end{equation} 
where the use of Lebesgue is for simplicity.

Since the sample measures $\mu_{\omega,0}$ coincide with $\mu,$ 
we may easily verify  condition (\ref{xibound}) by choosing conveniently the size of the hole $H_{\omega,\epsilon_N}$.  
    Moreover, by choosing the holes to be contained entirely within exactly one interval of monotonicity, we see that \eqref{EX} holds, and thus \eqref{cond X} holds via Remark \ref{rem check cond X} and Proposition \ref{prop check cond X}. 
Since the weights $g_\omega$ are nonrandom and equal to $1/|T'|,$ conditions (\ref{E1}) to (\ref{E7}) clearly hold. 
    Since $T$ is continuous in $x_0,$, the holes will belong to one branch of $T.$  Therefore if $D(T)$ will denote the number of branches of $T$ and $\lambda_m:=\min_{I} |T'|,$ $\lambda_M:=\max_{I} |T'|,$ it will be enough to have $D(T)-1>\frac{\lambda_M}{\lambda_m}$ in order to satisfy (\ref{E9}) with $n'=1.$ Moreover, still keeping $n'=1$ and since we have finitely many branches, condition (\ref{E9a}) in Remark \ref{Alt E9 Remark} is satisfied with $k_o(n')=1$.
We may now apply Theorem \ref{existence theorem} to obtain a quenched thermodynamical formalism for our fixed map with sufficiently small random holes.

Whenever the point $x_0$ is not periodic, one obtains that all the $\hat{q}^{(k)}_{\om,\epsilon_N}$ are zero by repeating the argument given in Example \hyperref[example1]{1} for $k\ge 0$ and using the continuity of $T$ at $x_0$. 
We now take $x_0$ as a periodic point of minimal period $p$ and we assume that $T^p$ and $\rho$ are continuous at $x_0$.
We now begin to evaluate 
\begin{equation}\label{fed}
	\hat{q}_{\omega,\epsilon_N}^{(p-1)}=\frac{\mu(T^{-p} H_{\omega,\epsilon_N}\cap H_{\sigma^{-p}\omega,\epsilon_N})}{\mu(H_{\omega,\epsilon_N})}. 
\end{equation}
Since $T^p$ is continuous and expanding in the neighborhood of  $x_0$, by taking $N$ large enough, the set $T^{-p} H_{\omega,\epsilon_N}\cap H_{\sigma^{-p}\omega,\epsilon_N}$ 
has only one connected component.
Denote the local branch of $T$ through $x_0$ by $T_{x_0}$.
Therefore by a local change of variable  we have for the upper bound of the numerator 
\begin{eqnarray}
	\label{eg3eq1}
	\mu(T^{-p} H_{\omega,\epsilon_N}\cap H_{\sigma^{-p}\omega,\epsilon_N})&=& \int_{T^p(T^{-p} H_{\omega,\epsilon_N}\cap H_{\sigma^{-p}\omega,\epsilon_N})}\rho(T^{-p}_{x_0}y)|DT^p(T^{-p}_{x_0}y)|^{-1}d\text{Leb}(y)\\
	\nonumber &\le& \sup_{H_{\omega,\epsilon_N}}|DT^p|^{-1}\sup_{H_{\omega,\epsilon_N}}\rho\ \text{Leb}(H_{\omega,\epsilon_N}).
\end{eqnarray}
For the lower bound of the numerator, since $T^p$ is locally expanding, $T^p(H_{\sigma^{-p}\omega,\epsilon_N})\supset H_{\sigma^{-p}\omega,\epsilon_N}$, and by (\ref{eg3eq1})
\begin{align*}
	&\mu(T^{-p} H_{\omega,\epsilon_N}\cap H_{\sigma^{-p}\omega,\epsilon_N})
	\ge\int_{H_{\omega,\epsilon_N}\cap H_{\sigma^{-p}\omega,\epsilon_N}}\rho(T^{-p}_{x_0}y)|DT^p(T^{-p}_{x_0}y))|^{-1}d\text{Leb}(y)
	\\
	& \quad
	\ge\inf_{H_{\omega,\epsilon_N}}|DT^p|^{-1}\inf_{H_{\omega,\epsilon_N}}\rho\ \text{Leb}(H_{\omega,\epsilon_N})-\int_{H_{\omega,\epsilon_N}\setminus H_{\sigma^{-p}\omega,\epsilon_N}}\rho(T^{-p}_{x_0}y)|DT^p(T^{-p}_{x_0}y)|^{-1}d\text{Leb}(y)    
\end{align*}

We can bound the second negative term as
$$
\int_{H_{\omega,\epsilon_N}\setminus H_{\sigma^{-p}\omega,\epsilon_N}}\rho(T^{-p}_{x_0}y)|DT^p(T^{-p}_{x_0}y)|^{-1}(y)d\text{Leb}(y)\le \sup_{H_{\omega,\epsilon_N}}\rho \  \sup_{H_{\omega,\epsilon_N}}|DT^p|^{-1}\ \text{Leb}(H_{\omega,\epsilon_N}\Delta H_{\sigma^{-p}\omega,\epsilon_N}).
$$


We are now in a position to verify the existence of the limit $\lim_{\epsilon_N\to 0} \hat{q}_{\omega,\epsilon_N}^{(p-1)}$.
Using the above bounds we have 
$$ \frac{\inf_{H_{\omega,\epsilon_N}}|DT^p|^{-1}\inf_{H_{\omega,\epsilon_N}}\rho\ \text{Leb}(H_{\omega,\epsilon_N})}{\sup_{H_{\omega,\epsilon_N}}\rho\text{Leb}(H_{\omega,\epsilon_N})}-\frac{\  \sup_{H_{\omega,\epsilon_N}}|DT^p|^{-1}\ \text{Leb}(H_{\omega,\epsilon_N}\Delta H_{\sigma^{-p}\omega,\epsilon_N})}{\text{Leb}(H_{\omega,\epsilon_N})}
$$
$$
\le \hat{q}_{\omega,\epsilon_N}^{(p-1)}\le \frac{\sup_{H_{\omega,\epsilon_N}}|DT^p|^{-1}\sup_{H_{\omega,\epsilon_N}}\rho\ \text{Leb}(H_{\omega,\epsilon_N})}{\inf_{H_{\omega,\epsilon_N}}\rho\text{Leb}(H_{\omega,\epsilon_N})}.$$

Since the map $T^p$ and the density $\rho$  are continuous at ${x_0},$  and by using the assumption (\ref{rf}), we will finally obtain the $\omega$-independent value
$$
\theta_{\omega, 0}=1-\frac{1}{|DT^p({x_0})|}.
$$
Theorem \ref{evtthm}  may now be applied to obtain the  quenched extreme value law. 

\subsection{Example 4\xlabel[4]{example4}: Random maps  with random holes}


We saw in Examples \hyperref[example1]{1} and \hyperref[example3]{3} that an extremal index less than one could be obtained for observables reaching their maximum around a point which was periodic for all the random maps, or, for a fixed map, when the holes shrink around the periodic point.
We now produce an example where periodicity is not responsible for getting an extremal index less than one. 
This example is the quenched version of the annealed cases investigated  in sections 4.1.2 and  4.2.1 of the paper \cite{Caby_et_al_2020}.

Let $\Omega=\{0,\dots,l-1\}^{\mathbb{Z}},$ with $\sigma$ the bilateral shift map, and $m$ an invariant ergodic measure.
To each letter $j=0,\dots,l-1$ we associate a point $v(j)$ in the unit circle  $S^1$ and we consider the well-known observable in the extreme value theory literature:
$$
h_{\omega}(x)=-\log|x-v(\omega_0)|, \ x\in S
$$
where $\omega_0$ denotes the $0$-th  coordinate of $\omega\in \Omega.$

In this case the hole $H_{\omega, \epsilon_N}$ will be the ball $B(v(\omega_0), e^{-z_N(\om)}),$ of center $v(\omega_0)$ and radius $e^{-z_N(\om)};$.
For each $\omega\in \Omega$ we associate a map $T_{\omega_0},$ where $T_0,\dots,T_{l-1}$ are maps  of the circle which we will take as $\beta$-maps of the form\footnote{The reason for this choice is that,  in order to compute the quantities  $\hat{q}_{\omega,0}^{(k)}$, we have to follow the itinerary of the points $v(i)$ under the composition of the maps $T_i$ and compare with their predecessors.  
	As it will be clear in the computation below for $k=0$ and $1,$ 
	the task will be relatively easy
	and generalisable to any $k > 0$, if all the maps are at least $C^1$ which in particular means that the image of the point $T_i(v_j), i,j=0,\dots,l-1$ is not a discontinuous point of the $T_i, i=0,\dots, l-1,$. 
	It will also be true that all the maps $T_i, 0=1,\dots, l-1$ are differentiable with bounded derivative in small neighbourhoods of any $v_i, i=0,\dots, l-1$.
	If these conditions are relaxed, it could be that the  limit defining the  $\hat{q}_{\omega,0}^{(k)}$ for some $k$ is not immediately computable, we defer to section 3.3 and Proposition 3.4 in \cite{AFV15} for a detailed discussion of the computation of the extremal index in presence of discontinuities. Another advantage of our choice is that, as in Example \hyperref[example1]{1}, all the sample measure $\mu_{\om,0}$ are equivalent to Lebesgue, $\text{Leb}.$} $T_i(x)=\beta_ix+r_i\pmod 1$, with $\beta_i\in \mathbb{N}$, $\bt_i\geq 3$, and $0\le r_i<1$. Since the range of $\om\mapsto \mathcal{L}_{\omega,0}$ is finite and the shift is a homeomorphism on $\Omega$ with respect to the usual metrics for $\Omega$,  assumption (\ref{M}) is verified.
Since the potential is equal to $1/|T_\om'|,$ conditions (\ref{E1}) to (\ref{E7}) clearly hold.
    Condition (\ref{E8}) holds as in Example \hyperref[example1]{1}, which uses similar piecewise linear expanding maps. Condition (\ref{E9}) is a consequence of  the fact that we have finitely many maps each of which is full branches; it will be therefore enough to invoke (\ref{E9a}) with $n'=k_o(n')=1.$
    As we have chosen $\bt_i\geq 3$ we have that \eqref{EX} holds and thus \eqref{cond X} follows via Proposition \ref{prop check cond X}.

At this point we may apply Theorem \ref{existence theorem} to obtain a quenched thermodynamic formalism for sufficiently small holes.


As a more concrete example, we will consider an alphabet of four letters $\mathcal{A}:=\{0,1,2,3\}$  and we set the associations
$$
i\rightarrow v_i:=v(i), \  i=0,1,2,3,
$$
where the points $v_i\in (0,1)$, are chosen on the unit interval  according to the following prescriptions:
\begin{equation}\label{grammar}
	T_1(v_1)=T_2(v_2)=v_0; \ T_0(v_0)=v_3; \ T_3(v_i)\neq v_3, i=0,1,2,3.
\end{equation}
Since the sample measures coincide with the Lebesgue measure, condition (\ref{xibound}) reduces to 
$2e^{-z_N(\om)}=\frac{t_{\om}+\xi_{\om, N}}{N}$ and we can solve for $z_N(\om)$ by setting $\xi_{\om,N}\equiv0$. 
The prescription (\ref{grammar}) clearly avoids any sort of periodicity when we take the first iteration of the random maps needed to compute $\hat{q}_{\omega,0}^{(0)};$ nevertheless we will show that  $\hat{q}_{\omega,0}^{(0)}>0$ and this is sufficient to conclude that $\theta_{\om,0}$ is smaller than $1$. 
Of course we need to prove that the limits defining  all the other 
$\hat{q}_{\omega,0}^{(k)}$ for any $k>0$ exist.   We will show it for $k=1$ and the same arguments could be generalized to any $k>1.$

Using $T$-invariance of Lebesgue, we have (for simplicity we write $z_N$ instead of $z_N(\om)$):
$$
\hat{q}_{\omega,0}^{(0)}= \lim_{N\rightarrow \infty}\frac{\text{Leb}\left(T^{-1}_{\sigma^{-1}\om}
	B(v(\omega_0), e^{-z_N})\cap B(v((\sigma^{-1}\omega)_0), e^{-z_N})\right)}{\text{Leb}(B(v(\omega_0), e^{-z_N}))},
$$
provided the limit  exists, which we are going to establish.
Consider first a point $\omega$ in the cylinder $[\omega_{-1}=0, \omega_0=3]$; the quantity we have to compute is therefore
\begin{equation}\label{ratio}
	\hat{q}_{\omega,0}^{(0)}=\frac{\text{Leb}\left(T^{-1}_{0}
		B(v_3, e^{-z_N})\cap B(v_0, e^{-z_N})\right)}{\text{Leb}(B(v_3, e^{-z_N}))}
\end{equation}
If $N$ is large enough, by the prescription (\ref{grammar}), $T_0(v_0)= v_3,$ we see that the local preimage of $B(v_3, e^{-z_N})$ under $T_0^{-1}$ will be strictly included into $B(v_0, e^{-z_N})$ and its length will be contracted by a factor $\beta_0^{-1}.$ Therefore the ratio (\ref{ratio}) will be simply $\beta_0^{-1}.$ The same happens with the cylinders $[\omega_{-1}=1, \omega_0=0]$ and $[\omega_{-1}=2, \omega_0=0],$  producing  respectively the quantities $\beta_1^{-1}$ and $\beta_2^{-1}.$ For all the other cylinders the ratio will be zero since the sets entering the numerator of (\ref{ratio}) will be disjoint for $N$ large enough. In conclusion we have
$$
\int_{\Omega}t_{\om}\hat{q}_{\omega,0}^{(0)}dm=\int_{[\omega_{-1}=0, \omega_0=3]}t_{\om}\beta_0^{-1}dm+\int_{[\omega_{-1}=1, \omega_0=0]}t_{\om}\beta_1^{-1}dm+\int_{[\omega_{-1}=2, \omega_0=0]}t_{\om}\beta_2^{-1}dm.
$$
If we now take $t_{\om}=t$, $\om$-independent, and we choose $m$ as the Bernoulli measure with equal weights $1/4$, the preceding expression assumes the simpler form
$$
\int_{\Omega}t\hat{q}_{\omega,0}^{(0)}dm=\frac{t}{16}(\beta_0^{-1}
+\beta_1^{-1}+\beta_2^{-1}).
$$
To compute $\hat{q}_{\omega,0}^{(1)}$ we have to split the integral over cylinders of length three. As a concrete example let us consider the cylinder $[\om_{-2}=i, \om_{-1}=j, \om_0=k],$ where $i,j,k\in \mathcal{A}.$ By using the preceding notations, we have to control the set
$$
T^{-1}_{i}T^{-1}_{j}
B(v_k, e^{-z_N})\cap T^{-1}_{i}B^c(v_j, e^{-z_N})\cap B(v_i, e^{-z_N})=
$$
$$
T^{-1}_{i}\left(T^{-1}_{j}
B(v_k, e^{-z_N})\cap B^c(v_j, e^{-z_N})\right)\cap B(v_i, e^{-z_N}).
$$
Let us first consider the intersection $T^{-1}_{j}
B(v_k, e^{-z_N})\cap B^c(v_j, e^{-z_N}).$ Call $u$ one of the preimages of $T^{-1}_{j}(v_k)$ and $T^{-1}_{j,u}$ the inverse branch giving $T^{-1}_{j,u}(v_k)=u$. If $u=v_j$ then the intersection $T^{-1}_{j,u}
B(v_k, e^{-z_N})\cap B^c(v_j, e^{-z_N})$ is  empty. Otherwise, 
by taking $N$ large enough, the set  $T^{-1}_{j,u}
B(v_k, e^{-z_N})$ will be completely included in $B^c(v_j, e^{-z_N})$ and moreover we have, by the linearity of the maps, $T^{-1}_{j,u}
B(v_k, e^{-z_N})=B(u, \beta_j^{-1}e^{-z_N}).$ We are therefore left with the computation of $T^{-1}_{i}B(u, \beta_j^{-1}e^{-z_N})\cap B(v_i, e^{-z_N}).$ If $u\in \mathcal{A}$ we proceed as in the computation of $\hat{q}_{\omega,0}^{(0)}$ by using the prescriptions (\ref{grammar}); otherwise such an intersection will be empty for $N$ large enough. For instance, if, in the example we are considering, we take $k=0, j=2, i=0$ and we suppose that among the $\beta_2$ preimages of $v_0$ there is, besides $v_0,$ $v_3$ too, namely $T_2(v_3)=v_0,$ and no other element of $\mathcal{A},$  \footnote{This condition does not intervene to compute $\hat{q}_{\omega,0}^{(0)}$ and show that it was strictly less than one.} then we get the contribution for $\hat{q}_{\omega,0}^{(1)}:$
$$
\int_{[\om_{-2}=0, \om_{-1}=2, \om_0=0]}t_{\om}\beta_2^{-1}\beta_0^{-1}dm.
$$
From the above it follows that $\theta_{\om,0}$ exists for a.e.\ $\omega$, and that we may apply Theorem \ref{evtthm}.\\

It is not difficult to give an example where all the $\hat{q}_{\omega,0}^{(k)}$ can be explicitly  computed. Let us take our  beta maps  $T_i(x)=\beta_ix+r_i$-$\text{mod}\ 1$  in the particular case where all the $r_i$ are  equal to the irrational number $r.$ Then  take a sequence of random balls  $B(v((\sigma^{k}\om)_0), e^{-z_N}), k\ge 0$ with the centers $v((\sigma^{k}\om)_0), k\ge 0$ which are rational numbers. From what we discussed above, it follows that  a necessary condition to get a $\hat{q}_{\om,0}^{(k)}\neq0$ is that the center $v((\sigma^{-(k+1)}\om)_0)$ will be sent to the center $v(\om_0).$  Let $z$ one of this rational centers; the iterate $T^n_{\om}(z)$ has the form $T^n_{\om}(z)=\beta_{\om_{n-1}}\cdots \beta_{\om_0}z+k_nr,$-mod$1,$ where $k_n$ is an integer number.
Therefore such an iterate will be never a rational number which 
shows that all the $\hat{q}_{\om,0}^{(k)}=0$ for any $k\ge 0$ and $\om,$ and therefore $\theta_{\om,0}=1.$

\section*{Acknowledgments}
The authors thank Harry Crimmins for helpful discussions.
JA, GF, and CG-T thank the Centro de Giorgi in Pisa and CIRM in Luminy for their support and hospitality.
JA is supported by the ARC Discovery projects DP180101223 and DP210100357, and thanks the School of Mathematics and Physics at the University of Queensland  for their hospitality.
GF, CG-T, and SV are partially supported by the ARC Discovery Project DP180101223.
The research of SV was supported by the project {\em Dynamics and Information Research Institute} within the agreement between UniCredit Bank and Scuola Normale Superiore di Pisa and by the Laboratoire International Associé LIA LYSM, of the French CNRS and  INdAM (Italy).

\appendix

\section{A version of \cite{DFGTV18a} with general weights}
\label{appB}

In this appendix we outline how to extend relevant results from \cite{DFGTV18a} from the Perron--Frobenius weight $g_\omega=1/|T_\omega'|$ to the  general class of weights $g_\omega$ in Section \ref{sec: existence}.
To begin, we note that there is a unique measurable equivariant family of functions $\{\varphi_{\omega,0}\}_{\omega\in\Omega}$ guaranteed by Theorem 2.19 \cite{AFGTV20} (called $q_\omega$ there).
We wish to obtain uniform control on the essential infimum and essential supremum of $\varphi_{\omega,0}$ for a suitable class of maps.

In \cite{DFGTV18a} we work with the space $\BV=\{h\in L^1(\Leb) :\var(h)<\infty\}$ and use the norm $\|\cdot\|_{\BV}=\var(\cdot)+\|\cdot\|_1$.
Here we have a measurable family of random conformal probability measures $\nu_{\omega,0}$ (guaranteed by Theorem 2.19 \cite{AFGTV20}) and we work with the random spaces $\BV_\omega=\{h\in L^1(\nu_{\omega,0}):\var(h)<\infty\}$ and the random norms $\|\cdot\|_{\BV_\omega}=\var(\cdot)+\|\cdot\|_{L^1(\nu_{\omega,0})}$.
We also work with the normalised transfer operator $\tilde{\mathcal{L}}_{\omega,0}(f):=\mathcal{L}_{\omega,0}(f)/\nu_{\omega,0}(\mathcal{L}_{\omega,0}\ind)$;  throughout Appendix \ref{appB} we simply denote $\mathcal{L}_{\omega,0}$ and $\tilde{\mathcal{L}}_{\omega,0}$ by $\mathcal{L}_{\omega}$ and $\tilde{\mathcal{L}}_{\omega}$ as we only deal with the closed random system. Similarly we denote $g_{\om,0}$, $\nu_{\om,0}$, and $\phi_{\om,0}$ by $g_\om$, $\nu_\om$, and $\phi_\om$ respectively. 
All of the variation axioms (V1)--(V8) in \cite{DFGTV18a} hold with the obvious replacements.

\begin{proof}[Proof of Lemma \ref{DFGTV18alemma}]
	\mbox{ }
	
	\textbf{C1':}
	Since $\essinf_\om\inf\cL_{\om,0}\ind\geq \essinf_\om\inf g_{\om}$, \eqref{E3} and \eqref{fin sup L1} together imply that \eqref{C1'} holds.
	\vspace{.25cm}
	
	\textbf{C7' ($\epsilon=0$):} We show $\essinf_\omega\inf\varphi_{\omega}>0$; this will give us \eqref{C7'}.
	The statement $\essinf_\omega\inf\varphi_{\omega}>0$ is a generalised version of Lemmas 1 and 5 \cite{DFGTV18a} and we follow the strategy in \cite{DFGTV18a}.  The result follows from Lemma \ref{lemma1} below, which in turn depends on (\ref{LY20}) and Lemma \ref{lemmaA1}.
	\vspace{.25cm}
	
	\textbf{C4' ($\epsilon=0$):}  It will be sufficient to show that there is a $K<\infty$ and $0<\gamma<1$ such that for all $f\in \BV$ with $\nu_\omega(f)=0$ and a.e.\ $\omega$, one has
	\begin{equation}
		\label{DECgeneral}
		\|\tilde{\mathcal{L}}_\omega^n f\|_{\BV_{\sigma^n\omega}}\le K\gamma^n\|f\|_{\BV_\omega}\mbox{ for all $n\ge 0$}.
	\end{equation}
	This is a generalised version of Lemma 2 \cite{DFGTV18a}, which has an identical proof, making the replacements outlined in the proof of Lemma \ref{lemmaA1} below, and using Lemmas \ref{lemmaA1}--\ref{lemmaA4} and (\ref{LY20}).
	We also use the non-random equivalence (\ref{normequiv}) of the $\BV_\omega$ norm to the usual $\BV$ norm.
	\vspace{.25cm}
	
	\textbf{C2 ($\epsilon=0$), C3 ($\epsilon=0$), C5' ($\epsilon=0$):}
	We wish to show that 
	there is a unique measurable, nonnegative family $\varphi_{\omega}$ with the property that $\varphi_{\omega}\in \BV_\omega$, $\int \varphi_{\omega}\ d\nu_{\omega}=1$, $\tilde{\mathcal{L}}_{\omega} \varphi_{\omega}=\varphi_{\sigma\omega}$ for a.e. $\omega$, $\esssup_\omega \|\varphi_{\omega}\|_{\BV_\omega}<\infty$, and
	\begin{equation}
		\label{unifeqnphi2}
		\esssup_\omega \|\varphi_{\omega}\|_{\BV_\omega}.
	\end{equation} 
	We note that again the norm equivalence (\ref{normequiv}).
	This is a generalised version of Proposition 1 \cite{DFGTV18a}.
	To obtain this generalisation, in the proof of Proposition 1 \cite{DFGTV18a}, one modifies the space $Y$ to become $Y=\{v:\Omega\times X\to\mathbb{R}: v\mbox{ measurable,} v_\omega:=v(\omega,\cdot)\in \BV_\omega\mbox{ and }\esssup_\omega\|v_\omega\|_{\BV_\omega}<\infty\}$.
	All of the arguments go through as per \cite{DFGTV18a} with the appropriate substitutions.
	Our modified proof of Proposition 1 \cite{DFGTV18a} will also use the modified Lemmas \ref{lemmaA1}--\ref{lemmaA4}, and inequality (\ref{LY20}) below.
	
	\textbf{CCM:} Finally, we note that \eqref{CCM} follows from \eqref{C2} ($\ep=0$) together with non-atomicity of $\nu_{\om}$. But non-atomicity of $\nu_{\om}$ follows from the random covering assumption as in the proof of Proposition 3.1 \cite{AFGTV20}

    \textbf{Full Support:}
    Following the proof of Claim 3.1.1 of \cite{AFGTV20}, we are able to show that $\nu_\om$ is fully supported on $[0,1]$, i.e. $\nu_{\om}(J)>0$ for any non-degenerate interval $J\sub [0,1]$.
\end{proof}

We note that by (\ref{LYfull hat}) 
with $\epsilon=0$, we have our uniform Lasota--Yorke equality.
\begin{equation}
	\label{LY20}
	\var(\tilde{\mathcal{L}}^{n}_\omega\phi)\le A \alpha^n\var(\phi)+B\nu_\omega(|\phi|),
\end{equation}
for all $n\ge 1$ and a.e.\ $\omega$.
This immediately provides a suitable general version of (H2) \cite{DFGTV18a}, which is that there is a $C<\infty$ such that
\begin{equation}
	\label{H2}
	\|\tilde{\mathcal{L}}_\omega\phi\|_{\BV_{\sigma\omega}}\le C\|\phi\|_{\BV_{\omega}}\mbox{ for a.e.\ $\omega$.}
\end{equation}
Define random cones $\mathcal{C}_{a,\omega}=\{\phi\in \BV_\omega: \phi\ge 0, \var\phi\le a\int \phi\ d\nu_\omega\}$.

\begin{lemma}[General weight version of Lemma A.1 \cite{DFGTV18a}]
	\label{lemmaA1} For sufficiently large $a>0$ we have that $\tilde{\mathcal{L}}_\omega^{RN}\mathcal{C}_{a,\omega}\subset \mathcal{C}_{a/2,\sigma^{RN}\omega}$ for sufficiently large $R$ and a.e.\ $\omega$.
\end{lemma}
\begin{proof}Identical to \cite{DFGTV18a}, substituting $\mathcal{C}_{a,\omega}$ for $\mathcal{C}_{a}$, $\tilde{\mathcal{L}}_\omega$ for $\mathcal{L}_\omega$, and $\nu_\omega$ for Lebesgue.
\end{proof}

\begin{lemma}[General weight version of Lemma 1 \cite{DFGTV18a}]
	\label{lemma1}
	If one has uniform covering in the sense of (11) \cite{DFGTV18a}, then there is an $N$ such that for each $a>0$ and sufficently large $n$, there exists $c>0$ such that
	\begin{equation}
		\label{l1eqn}
		\essinf_\omega \tilde{\mathcal{L}}_\omega^{Nn}h\ge (c/2)\nu_\omega(|h|)\mbox{ for every $h\in \mathcal{C}_{\omega,a}$ and a.e.\ $\omega$}.
	\end{equation}
\end{lemma}
\begin{proof}
	Making all of the obvious substitutions, as per Lemma \ref{lemmaA1} and its proof, we subdivide the unit interval into an equipartition according to $\nu_\omega$ mass.
	This is possible because $\nu_\omega$ is non-atomic (Proposition 3.1 \cite{AFGTV20}).
	We conclude, as in the proof of Lemma 1 \cite{DFGTV18a}, that there is an interval $J$ of $\nu_\omega$-measure $1/n$ such that for each $f\in\mathcal{C}_{\omega,a},$ one has $\inf_J f\ge (1/2)\nu_\omega(f)$.
	Then using uniform covering and the facts that $\essinf_\omega\inf g_\omega>0$, $\esssup_\omega g_\omega<\infty$, and $\esssup_\omega D(T_\omega)<\infty$, we obtain $\essinf_\omega\inf \tilde{\mathcal{L}}_\omega^{(k)}f\ge\alpha^*_0>0$, where $k$ is the uniform covering time for the interval $J$.
	The rest of the proof follows as in \cite{DFGTV18a}.
\end{proof}

\begin{lemma}[General weight version of Lemma A.2 \cite{DFGTV18a}]
	\label{lemmaA2}
	Assume that $\phi,\psi\in \mathcal{C}_{a,\omega}$ and  $\int \phi\ d\nu_\omega=\int \psi\ d\nu_\omega=1$. Then $\|\phi-\psi\|_{\BV_\omega}\le 2(1+a)\Theta_{a,\omega}(\phi,\psi)$.
\end{lemma}
\begin{proof}
	Identical to \cite{DFGTV18a}, substituting $\nu_\omega$ for Lebesgue.
	The randomness of the Hilbert metric $\Theta_{a,\omega}$ only appears because the functions lie in $\mathcal{C}_{a,\omega}$.
\end{proof}

\begin{lemma}[General weight version of Lemma A.3 \cite{DFGTV18a}]
	\label{lemmaA3}
	For any $a\ge 2\var(\ind_X)$, we have that $\tilde{\mathcal{L}}_\omega^{RN}$ is a contraction on $\mathcal{C}_{\omega,a}$ for any sufficiently large $R$ and a.e.\ $\omega\in\Omega$.
\end{lemma}
\begin{proof}
	The proof in \cite{DFGTV18a} may be followed, making the substitutions as in Lemma \ref{lemmaA1} and its proof.
	The first inequality reads 
	$$
	    \esssup_\omega \tilde{\mathcal{L}}^{RN}_\omega f\le \nu_{\sigma^{RN}\omega}(|\tilde{\mathcal{L}}^{RN}_\omega f|)+C_{var}(\tilde{\mathcal{L}}^{RN}_\omega f)\le (1+C_{var}a/2)\nu_\omega(|f|)=1+C_{var}a/2,
    $$
	where we have used axiom (V3) \cite{DFGTV18a} and the weak contracting property of $\tilde{\mathcal{L}}^{RN}_\omega$ in the $\nu_\omega$ norm.
	The rest of the proof follows as in \cite{DFGTV18a}.
\end{proof}

Let $\BV_{\omega,0}=\{\phi\in \BV: \int \phi\ d\nu_\omega=0\}$.

\begin{lemma}[General weight version of Lemma A.4 \cite{DFGTV18a}]
	\label{lemmaA4}
	
\end{lemma}
\begin{proof}
	Identical to \cite{DFGTV18a}, substituting as per Lemma \ref{lemmaA1} and Lemma \ref{lemmaA2} and their proofs above, and using (\ref{H2}).
\end{proof}

\begin{lemma}[General weight version of Lemma 5 \cite{DFGTV18a}]
	\label{lemma5}
	$\essinf_\omega \phi_{\omega}\ge c/2$ for a.e.\ $\omega$.
\end{lemma}
\begin{proof}
	Identical to \cite{DFGTV18a} with the appropriate substitutions.
\end{proof}

\section{A summary of checks that relevant results from \cite{crimmins_stability_2019} can be applied to $\BV$}
\label{appA}

The stability result Theorem 4.8 \cite{crimmins_stability_2019} assumes that the underlying Banach space is separable, however this separability assumption is only used to obtain measurability of various objects (and in fact Theorem 3.9 \cite{crimmins_stability_2019} may be applied to sequential dynamics).
We use these results for the non-separable space $\BV$ in the proof of Lemma \ref{harrylemma2}, and we therefore need to check that all relevant results in \cite{crimmins_stability_2019} hold for $\BV$, under the $m$-continuity assumption on $\omega\mapsto\mathcal{L}_{\omega,0}$;  the latter will provide the required measurability.
All section, theorem, proposition, and lemma numbers below refer to numbering in \cite{crimmins_stability_2019}.

There are no issues of measurability in Section 3, including Theorem 3.9, until Section 4, so we begin our justifications from Section 4.
\vspace{.5cm}

\textit{Theorem 4.8:}
This is the main stability theorem.
We would substitute ``separable strongly measurable random dynamical system'' with $m$-continuous random dynamical system. This theorem relies on Propositions 4.15 and 4.16, Lemma 4.21, Proposition 4.22, and Lemma 4.23.
\vspace{.5cm}

\emph{Proposition 4.15:} There is no measurability involved.
\vspace{.5cm}

\emph{Proposition 4.16:} Uses Lemmas 4.17 and 4.19. We note that the bounds (95)--(97) in the proof are simpler in our application of this result as our top space is one-dimensional.
\vspace{.5cm}

\emph{Lemma 4.17:} This may be replaced by Theorem 17 \cite{FLQ2}, which treats the $m$-continuous setting.  This removes any use of Lemma B.1 and Lemma B.7.
\vspace{.5cm}

\emph{Lemma 4.18:} This uses Lemma 4.17 and the fact that compositions of $m$-continuous maps are $m$-continuous. The latter replaces the use of Lemma A.5 \cite{GTQ14}, which is used in several results.  This replacement will not be mentioned further.
\vspace{.5cm}

\emph{Lemma 4.19:} We will assume that $\omega\mapsto\Pi_\omega$ is $m$-continuous in the statement of the lemma.  At the start of the proof we would now instead have $\omega\mapsto \Pi_\omega(X)$ is $m$-continuous by the definition of $m$-continuity (see e.g.\ (4) in \cite{FLQ2}); this removes the use of Lemma B.2 in the proof. 
Lian's thesis is quoted regarding measurable maps/bases connected with a measurable space $\Pi_\omega(X)$. 
In our application, $\Pi_\omega$ has rank 1 and therefore stating that there is a (in our case $m$-continuous) map $e:\Omega\to \Pi_\omega(X)$ is trivial.  One proceeds similarly for the dual basis.
\vspace{.5cm}

\emph{Lemma 4.21:} This concerns measurability and integrability of $\omega\mapsto \det(\mathcal{L}_\omega^n|E_{i,\omega})$ and $\omega\mapsto \log\|\mathcal{L}_\omega^n|E_{i,\omega}\|$ where $E_{i,\omega}$ is an Oseledets space.
$m$-continuity of $\omega\mapsto E_{i,\omega}$ is provided by Theorem 17 \cite{FLQ2}, removing the need for Lemma B.2.
The $m$-continuity of $\omega\mapsto \log\|\mathcal{L}_\omega^n|E_{i,\omega}\|$ follows from Lemma 7 \cite{FLQ2}.
Because in our application setting we only require one-dimensional $E_{i,\omega}$, the determinants are given by norms and there is nothing more to do concerning determinants.
This removes the need for Proposition B.8.
Lemma B.16 \cite{GTQ14} may be replaced with Lemma 7 \cite{FLQ2} to cover the $m$-continuous setting.
Proposition B.6 is not required in the $P$-continuity setting.
\vspace{.5cm}


\emph{Proposition 4.22:} There is no measurability involved.
\vspace{.5cm}

\emph{Lemma 4.23:} There is no measurability involved.

\section{A $\var$--$\nu_{\omega,0}(|\cdot|)$ Lasota--Yorke inequality}\label{appC}

Recall from Section~\ref{sec: existence} that $\cZ_{\om,0}^{(n)}$ denotes the partition of monotonicity of $T_\om^n$ and that $\sA_{\om,0}^{(n)}$ is the collection of all finite partitions of $[0,1]$ such that
\begin{align}\label{eq: def A partition App}
	\var_{A_i}(g_{\om,0}^{(n)})\leq 2\norm{g_{\om,0}^{(n)}}_{\infty}
\end{align}
for each $\cA=\set{A_i}\in\sA_{\om,0}^{(n)}$.
Given $\cA\in\sA_{\om,0}^{(n)}$, we set $\cZ_{\om,*,\ep}^{(n)}:=\set{Z\in \widehat\cZ_{\om,\ep}^{(n)}(\cA): Z\sub X_{\om,n-1,\ep} }$ where $\widehat\cZ_{\om,\ep}^{(n)}(\cA)$ is the coarsest partition amongst all those finer than $\cA$ and $\cZ_{\om,0}^{(n)}$ such that all elements of $\widehat\cZ_{\om,\ep}^{(n)}(\cA)$ are either disjoint from $X_{\om,n-1,\ep}$ or contained in $X_{\om,n-1,\ep}$. Then \eqref{eq: def A partition App} implies that 
\begin{align}\label{eq: def A partition for g_ep App}
	\var_{Z}(g_{\om,\ep}^{(n)})\leq 2\norm{g_{\om,0}^{(n)}}_{\infty}
\end{align}
for each $Z\in \cZ_{\om,*,\ep}^{(n)}$.
We now prove a Lasota--Yorke inequality inspired by \cite{AFGTV21}.
\begin{lemma}\label{closed ly ineq App} 
	For any $f\in\BV(I)$ we have 
	\begin{align*}
		\var(\cL_{\om,\ep}^n(f))\leq 9\norm{g_{\om,\ep}^{(n)}}_{\infty}\var(f)+
		\frac{8\norm{g_{\om,\ep}^{(n)}}_{\infty}}{\min_{Z\in\cZ_{\om,*,\ep}^{(n)}(A)}\nu_{\om,0}(Z)}\nu_{\om,0}(|f|).
	\end{align*}
\end{lemma}
\begin{proof}
	Since $\cL_{\om,\ep}^n(f)=\cL_{\om,0}^n(f\cdot\hat X_{\om,n-1,\ep})$, if $Z\in\widehat\cZ_{\om,\ep}^{(n)}(\cA)\bs\cZ_{\om,*,\ep}^{(n)}$, then  $Z\cap X_{\om,n-1,\ep}=\emptyset$, and thus, we have $\cL_{\om,\ep}^n(f\ind_Z)=0$ for each $f\in\BV(I)$. Thus, considering only on intervals $Z$ in $\cZ_{\om,*,\ep}^{(n)}$, we are able to write 
	\begin{align}\label{eq: ly ineq 1}
		\cL_{\om,\ep}^nf=\sum_{Z\in\cZ_{\om,*,\ep}^{(n)}}(\ind_Z f g_{\om,\ep}^{(n)})\circ T_{\om,Z}^{-n}
	\end{align} 
	where 
	$$	
	T_{\om,Z}^{-n}:T_\om^n(I_{\om,\ep})\to Z
	$$ 
	is the inverse branch which takes $T_\om^n(x)$ to $x$ for each $x\in Z$. Now, since 
	$$
	\ind_Z\circ T_{\om,Z}^{-n}=\ind_{T_\om^n(Z)},
	$$
	we can rewrite \eqref{eq: ly ineq 1} as 
	\begin{align}\label{eq: closed ly ineq 2}
		\cL_{\om,\ep}^nf=\sum_{Z\in\cZ_{\om,*,\ep}^{(n)}}\ind_{T_\om^n(Z)} \lt((f g_{\om,\ep}^{(n)})\circ T_{\om,Z}^{-n}\rt).
	\end{align}
	So,
	\begin{align}\label{closed var tr op sum}
		\var(\cL_{\om,\ep}^nf)\leq \sum_{Z\in\cZ_{\om,*,\ep}^{(n)}}\var\lt(\ind_{T_\om^n(Z)} \lt((f g_{\om,\ep}^{(n)})\circ T_{\om,Z}^{-n}\rt)\rt).
	\end{align}
	Now for each $Z\in\cZ_{\om,*,\ep}^{(n)}$ we have 
	\begin{align}
		&\var\lt(\ind_{T_\om^n(Z)} \lt((f g_{\om,\ep}^{(n)})\circ T_{\om,Z}^{-n}\rt)\rt)
		\leq \var_Z(f g_{\om,\ep}^{(n)})+2\sup_Z\absval{f g_{\om,\ep}^{(n)}}
		\nonumber\\
		&\qquad\qquad\leq 3\var_Z(f g_{\om,\ep}^{(n)})+2\inf_Z\absval{f g_{\om,\ep}^{(n)}}
		\nonumber\\
		&\qquad\qquad\leq 3\norm{g_{\om,\ep}^{(n)}}_{\infty}\var_Z(f)+3\sup_Z|f|\var_Z(g_{\om,\ep}^{(n)})+2\norm{g_{\om,\ep}^{(n)}}_{\infty}\inf_Z|f|
		\nonumber\\
		&\qquad\qquad\leq 
		3\norm{g_{\om,\ep}^{(n)}}_{\infty}\var_Z(f)+6\norm{g_{\om,\ep}^{(n)}}_{\infty}\sup_Z|f|+2\norm{g_{\om,\ep}^{(n)}}_{\infty}\inf_Z|f|
		\nonumber\\
		&\qquad\qquad\leq 
		9\norm{g_{\om,\ep}^{(n)}}_{\infty}\var_Z(f)+8\norm{g_{\om,\ep}^{(n)}}_{\infty}\inf_Z|f|
		\nonumber\\
		&\qquad\qquad\leq
		9\norm{g_{\om,\ep}^{(n)}}_{\infty}\var_Z(f)+8\norm{g_{\om,\ep}^{(n)}}_{\infty}\frac{\nu_{\om,0}(|f\rvert_Z|)}{\nu_{\om,0}(Z)}.
		\label{closed var ineq over partition}
	\end{align}
	Using \eqref{closed var ineq over partition}, we may further estimate \eqref{closed var tr op sum} as
	\begin{eqnarray}
		\nonumber	\var(\cL_{\om,\ep}^nf)
		&\leq& 
		\sum_{Z\in\cZ_{\om,*,\ep}^{(n)}} \lt(9\norm{g_{\om,\ep}^{(n)}}_{\infty}\var_Z(f)+8\norm{g_{\om,\ep}^{(n)}}_{\infty}\frac{\nu_{\om,0}(|f\rvert_Z|)}{\nu_{\om,0}(Z)}\rt)
		\\
		\label{ORLYLY}
		&\leq& 
		9\norm{g_{\om,\ep}^{(n)}}_{\infty}\var(f)+
		\frac{8\norm{g_{\om,\ep}^{(n)}}_{\infty}}{\min_{Z\in\cZ_{\om,*,\ep}^{(n)}(A)}\nu_{\om,0}(Z)}\nu_{\om,0}(|f|),
	\end{eqnarray}
	and thus we are done.
\end{proof}
\begin{remark}\label{Alt E9 Remark}
	Note that we could have used $\Leb$ or any probability measure in \eqref{closed var ineq over partition} rather than $\nu_{\om,0}$. Furthermore, Lemma~\ref{closed ly ineq App} could be applied to Section~\ref{sec: existence} with a measure other than $\nu_{\om,0}$ if the appropriate changes are made to the assumption \eqref{E9} so that a uniform-in-$\om$ lower bound similar to \eqref{LY LB calc} may be calculated. 
	
	In particular if we replace \eqref{E9} with the following:
	\begin{enumerate}[align=left,leftmargin=*,labelsep=\parindent]
	\item[\mylabel{E9a}{E9a}]
	There exists $k_o(n')\in\NN$ and $\dl>0$ such that for $m$-a.e. $\om\in\Om$, all $\ep>0$ sufficiently small, we have $\Leb(Z)>\dl$ for all $Z\in\cZ_{\om,*,\ep}^{(n')}(\cA)$,
	\item[\mylabel{E9b}{E9b}]
	There exists $c>0$ such that $\essinf_\om |T_\om'|>c$,
\end{enumerate}
then the claims of Section~\ref{sec: existence} hold with $\nu_{\om,0}$ in \eqref{closed var ineq over partition} replaced with $\Leb$.

Indeed, to obtain a replacement for \eqref{LY LB calc} one could use \eqref{E9a} and \eqref{E9b} to get 
\begin{align*}
    \Leb(Z)
    &=
    \Leb(\~\cL_{\om,0}^{k_o(n')}\ind_Z)
    \geq 
    \frac{\inf g_{\om,0}^{(k_o(n'))}\inf J_\om^{(k_o(n'))}}{\lm_{\om,0}^{k_o(n')}}\Leb(P_\om^{k_o(n')}\ind_Z)
    \\
    &\geq \essinf_\om \frac{\inf g_{\om,0}^{(k_o(n'))}\inf J_\om^{(k_o(n'))}}{\lm_{\om,0}^{k_o(n')}}\Leb(Z)
    >0
\end{align*}
for all $Z\in\cZ_{\om,*,\ep}^{(n')}(\cA)$, where we have also used \eqref{E3} for the final inequality.
As \eqref{E9} is only used to prove \eqref{LY LB calc}, the remainder of Section~\ref{sec: existence} can be carried out with the appropriate notational changes. In fact, the proof of Lemma~\ref{harrylemma2} can be simplified by replacing $\nu_{\om,0}$ with $\Leb$ as Lemma 5.2 of \cite{BFGTM14} would no longer be needed. 

\end{remark}

\begin{remark}
Note that the $2$ appearing in \eqref{eq: def A partition App}, and thus the $9$ and $8$ appearing in \eqref{ORLYLY}, are not optimal. See \cite{AFGTV20,AFGTV21} for how these estimates can be improved. 
\end{remark}

\section{Proof of Claim \protect{\eqref{item 6}} of Theorem \protect{\ref{existence theorem}}}\label{appDec}
\begin{proof}[Proof of Claim \eqref{item 6} of Theorem \ref{existence theorem}:]

Define the fully normalized operator $\sL_{\om,\ep}:\cB_\om\to\cB_{\sg\om}$ given by
$$
    \sL_{\om,\ep}(f):=\frac{1}{\rho_{\om,\ep}h_{\sg\om,\ep}}\cL_{\om,\ep}(f\cdot h_{\om,\ep}).
$$
Then we have that $\sL_{\om,\ep}\ind = \ind$ and $\mu_{\sg\om,\ep}(\sL_{\om,\ep}f)=\mu_{\om,\ep}(f)$.
Following the proof of Theorem 12.2 \cite{AFGTV21} we can prove the following similar statement to Claim \eqref{item 5} of Theorem \ref{existence theorem}: 	For each $h\in\BV$, $m$-a.e. $\om\in\Om$ and all $n\in\NN$ we have
\begin{align}\label{exp conv fn trop}
    \norm{\sL_{\om,\ep}^n h - \mu_{\om,\ep}(h)\ind}_{\cB_{\sg^n\om}}
    =\norm{\sL_{\om,\ep}^n \hat h}_{\cB_{\sg^n\om}}
    \leq D\norm{h}_{\cB_\om}\kp_\ep^n,
\end{align}
where $\hat h := h - \mu_{\om,\ep}(h)$ and $D$, $\kp_\ep$ are as in Claim \eqref{item 5} of Theorem \ref{existence theorem}.
Using standard arguments (see Theorem 11.1 \cite{AFGTV20}) we have that
\begin{align*}
    \absval{
				\mu_{\om,\ep}
				\lt(\lt(f\circ T_{\om}^n\rt)h \rt)
				-
				\mu_{\sg^{n}\om,\ep}(f)\mu_{\om,\ep}(h)
			}
    = \mu_{\sg^n\om,\ep}\lt(\lt|f\sL_{\om,\ep}^n\hat h\rt|\rt).
\end{align*}
Note that at this stage we are unable to apply \eqref{exp conv fn trop} as the $\|\cdot\|_{\cB_\om}$ norm and the measure $\mu_{\om,\ep}$ are incompatible.
Now from the third statement of Claim \eqref{item 5} of Theorem \ref{existence theorem} we have that 
\begin{align*}
    \lt|\mu_{\sg^n\om,\ep}\lt(\lt|f\sL_{\om,\ep}^n\hat h\rt|\rt)
    - 
    \frac{\vrho_{\sg^n\om,\ep}\lt(\lt|f \sL_{\om,\ep}^n\hat h\rt|\hat X_{\sg^n\om,n,\ep}\rt)}
    {\vrho_{\sg^n\om,\ep}\lt( X_{\sg^n\om,n,\ep}\rt)}
    \rt|
    \leq D\norm{f\sL_{\om,\ep}^n\hat h}_{\cB_{\sg^n\om}}\kp_\ep^n, 
\end{align*}
and thus we must have that 
\begin{align}
    \mu_{\sg^n\om,\ep}\lt(\lt|f\sL_{\om,\ep}^n\hat h\rt|\rt)
    \leq 
    D\norm{f\sL_{\om,\ep}^n\hat h}_{\cB_{\sg^n\om}}\kp_\ep^n +
    \frac{\vrho_{\sg^n\om,\ep}\lt(\lt|f\sL_{\om,\ep}^n\hat h\rt|\hat X_{\sg^n\om,n,\ep}\rt)}
    {\vrho_{\sg^n\om,\ep}(X_{\sg^n\om,n,\ep})}. \label{dec of corr proof final ineq1}
\end{align}
Using \eqref{exp conv fn trop} and \eqref{normequiv2}, we have that 
\begin{align}
\frac{\vrho_{\sg^n\om,\ep}\lt(\lt|f\sL_{\om,\ep}^n\hat h\rt|\hat X_{\sg^n\om,n,\ep}\rt)}
{\vrho_{\sg^n\om,\ep}(X_{\sg^n\om,n,\ep})}
&\leq \norm{f\sL_{\om,\ep}^n\hat h}_{\sg^n\om,\infty}
\leq \norm{f}_{\sg^n\om,\infty}\norm{\sL_{\om,\ep}^n\hat h}_{\cB_{\sg^n\om}}
\leq D\norm{f}_{\infty,\om}\|h\|_{\cB_\om}\kp_\ep^n.
\label{dec of corr proof final ineq2}
\end{align}
Combining \eqref{dec of corr proof final ineq1} and \eqref{dec of corr proof final ineq2} and using \eqref{exp conv fn trop} again we see that 
\begin{align*}
&\absval{
\mu_{\om,\ep}
\lt(\lt(f\circ T_{\om}^n\rt)h \rt)
-
\mu_{\sg^{n}\om,\ep}(f)\mu_{\om,\ep}(h)
}
\leq  \mu_{\sg^n\om,\ep}\lt(\lt|f\sL_{\om,\ep}^n\hat h\rt|\rt)
\\
&\qquad\qquad\leq
D\norm{f\sL_{\om,\ep}^n\hat h}_{\cB_{\sg^n\om}}\kp_\ep^n +
\frac{\vrho_{\sg^n\om,\ep}\lt(\lt|f\sL_{\om,\ep}^n\hat h\rt|\hat X_{\sg^n\om,n,\ep}\rt)}
{\vrho_{\sg^n\om,\ep}(X_{\sg^n\om,n,\ep})}
\\
&\qquad\qquad\leq 
D\norm{f\sL_{\om,\ep}^n\hat h}_{\cB_{\sg^n\om}}\kp_\ep^n +
D\norm{f}_{\infty,\om}\|h\|_{\cB_\om}\kp_\ep^n
\\
&\qquad\qquad\leq 
D^2\norm{f}_{\cB_\om}\norm{h}_{\cB_\om}\kp_\ep^{2n} +
D\norm{f}_{\infty,\om}\|h\|_{\cB_\om}\kp_\ep^n
\\
&\qquad\qquad\leq 
\~D\norm{f}_{\infty,\om}\|h\|_{\cB_\om}\kp_\ep^n
\end{align*}
for all $n$ sufficiently large,
and thus the proof of Claim \eqref{item 6} of Theorem \ref{existence theorem} is complete.

\end{proof}
\bibliographystyle{plain}
\bibliography{URP}

\end{document}